\documentclass[11pt,reqno]{amsart}

\usepackage[T1]{fontenc}
\usepackage[utf8]{inputenc}
\usepackage{amsmath,amssymb,amsthm,mathrsfs,bm,empheq}
\usepackage{booktabs,graphicx}
\usepackage{tikz}
\usetikzlibrary{arrows.meta,calc}
\usepackage{xcolor}
\usepackage{enumitem}
\usepackage{float}
\usepackage{placeins}
\usepackage{geometry}
\usepackage{hyperref}
\hypersetup{hidelinks}
\allowdisplaybreaks[4]

\theoremstyle{plain}
\newtheorem{lemma}{Lemma}[section]
\newtheorem{theorem}[lemma]{Theorem}
\newtheorem{proposition}[lemma]{Proposition}

\newtheorem{corollary}[lemma]{Corollary}
\newtheorem{assumption}{Assumption}[section]
\newtheorem{example}[lemma]{Example}
\theoremstyle{remark}
\newtheorem{remark}[lemma]{Remark}

\newcommand{\cP}{\mathcal P}
\newcommand{\cH}{\mathcal H}
\newcommand{\cL}{\mathcal L}
\newcommand{\cF}{\mathcal F}
\newcommand{\cU}{\mathcal U}
\newcommand{\bbR}{\mathbb R}
\newcommand{\bbS}{\mathbb S}
\newcommand{\gradG}{\nabla_G}
\newcommand{\divG}{\operatorname{div}_G}
\newcommand{\DeltaG}{\Delta_G}
\newcommand{\ip}[2]{(#1,#2)}
\newcommand{\norm}[1]{\lVert#1\rVert}

\makeatletter
\@namedef{subjclassname@2020}{%
  \textup{2020} Mathematics Subject Classification}
\makeatother

\begin{document}

\title[Semi-Lagrangian approximation for graph master equations]
{Semi-Lagrangian Approximation of Master Equations with Individual Noise on Finite Graphs}
\author{Jianbo Cui, Tonghe Dang}
\address{Department of Applied Mathematics, The Hong Kong Polytechnic
University, Hung Hom, Kowloon, Hong Kong, SAR, China}
\email{jianbo.cui@polyu.edu.hk; tonghe.dang@polyu.edu.hk(Corresponding author)}
\thanks{This work is supported by MOST National Key R\&D Program No. 2024FA1015900, the Hong Kong Research Grant Council GRF grant 15302823, GRF grant 15301025, NSFC/RGC Joint Research Scheme N$\_$PolyU5141/24, NSFC grant 12522119, NSFC grant 12301526, internal funds (P0041274, P0045336)  from Hong Kong Polytechnic University,  and the
CAS AMSS-PolyU Joint Laboratory of Applied Mathematics.}

\begin{abstract} 
We develop and analyze a fully discrete semi-Lagrangian method for vector master equations associated with mean field games (MFGs) with individual noise on finite graphs. The method discretizes both the characteristic evolution in time and the population variable on the probability simplex. A central challenge is the degeneracy of the transport coefficients near the boundary of the open simplex, where no boundary condition is prescribed. To overcome this difficulty in the analysis, we introduce an interior localization based on exact density characteristics, which keeps numerical arrivals within compact regions of uniform coefficient control. Combined with a closure argument for the coupled numerical recursion, this localization yields 
 local solvability of the updates and quantitative error bounds on compact subsets of the open simplex. Under additional $\mathcal C^2$ regularity of the exact solution and suitable space--time mesh coupling, these bounds imply first-order accuracy. Furthermore, the density and value trajectories generated by the numerical master feedback converge to the corresponding MFG characteristics at the same rate. In the potential setting, the associated reconstructed density--flux pair is admissible and asymptotically optimal for the planning problem, with an action gap quadratic in the feedback error. Numerical experiments illustrate the convergence rates, feedback responses to population redistribution, and the effects of graph topology.

\end{abstract}


%65M12(1991–now)Stability and convergence of numerical methods for initial value and initial-boundary value problems involving PDEs

%35Q89(2020–now)PDEs in connection with mean field game theory {For calculus of variations and optimal control, see 49N80; for game theory, see 91A16}

%49N80(2020–now)Mean field games and control {For partial differential equations, see 35Q89; for game theory, see 91A16}

%35R02(2010–now)PDEs on graphs and networks (ramified or polygonal spaces)

%65M15(1973–now)Error bounds for initial value and initial-boundary value problems involving PDEs
\maketitle

\section{Introduction}  
The master equation describes equilibrium values in mean field games (MFGs)
as functions of time, the individual state, and the population distribution
\cite{CardaliaguetDelarueLasryLions2019}.  
MFGs provide a framework for strategic
interactions in large populations, where the influence of each individual
agent on the aggregate state is negligible
\cite{AchdouLauriere2020,LasryLions2006a,LasryLions2006b,LasryLions2007,
HuangMalhameCaines2006}.
For a prescribed initial population distribution, an MFG equilibrium is
characterized by a coupled density--value evolution, whereas the master equation provides a feedback representation of such equilibria over the population distribution. 
Beyond its characteristic
connection with MFG systems, the master field also plays an important role
in finite-player approximation and fluctuation analysis
\cite{BayraktarCohen2018,CecchinPelino2019}.
When agents can occupy only finitely many states, such as locations in a
transportation network or regimes in a switching system, the state space can
be naturally represented by a finite graph. Its vertices describe the
possible states, while its edges specify the admissible transitions. MFGs and their master equations
on such state spaces 
arise naturally in networked decision problems,
including congestion management, resource allocation, and mobility decisions  
\cite{GomesMohrSouza2013}.

 Let $G=(V,E,\omega)$ be a finite connected weighted graph with the vertex set
$V=\{1,\ldots,d\}$, edge set $E\subset V\times V$ and edge weight matrix $\omega=(\omega_{ij})_{1\leq i,j\leq d}$.  In this paper, we study the numerical approximation of
the following vector master equation with individual noise on the graph $G$:
\begin{align}\label{eq_intro}
 \partial_tu(t,\mu)
 +D_\mu u(t,\mu)
 \bigl[\divG\bm B(\mu,\gradG u(t,\mu))+\DeltaG\mu\bigr]
 =\bm H(\mu,\gradG u(t,\mu))-\DeltaG u(t,\mu)
\end{align}
with the terminal condition $u(T,\mu)=g(\mu)$ for
$(t,\mu)\in[0,T]\times\cP^\circ(G),$ where $g:\mathcal P(G)\to\mathbb R^d$ is a terminal map.  Here $T>0$ is the terminal time, $\cP^\circ(G)$ is the interior of the probability simplex over the
graph vertices, and $u_i(t,\mu),i\in V$ is the equilibrium value of an agent at
vertex $i$ when the population distribution is $\mu$. The operators
$\gradG$, $\divG$, and $\DeltaG$ denote the graph gradient, divergence, and
Laplacian, respectively. In particular, the terms
$\DeltaG\mu$ and $-\DeltaG u$ represent independent Markov jumps at the
levels of the population law and the individual value, respectively. The derivative $D_\mu u$ is taken with respect to the population distribution along the tangent space of the probability simplex. The coefficients $\bm H$ and $\bm B$ denote the
Hamiltonian and transport flux,  respectively. Precise definitions of these operators and coefficients, together with the assumptions, are
given in Section~\ref{notation}.

In contrast to continuous-state models \cite{CardaliaguetDelarueLasryLions2019,BurgerDiFrancescoMarkowichWolfram2014}, the graph in \eqref{eq_intro}
represents the individual state space itself rather than a spatial
discretization. Its topology and edge weights enter directly into the
population transport and individual-noise operators; the relation with the
corresponding Euclidean operators is discussed in Section~\ref{role}.
For suitable edge mobilities, discrete optimal transport equips the
probability simplex with a Wasserstein-type geometry, under which the graph
heat equation can be interpreted as the gradient flow of the entropy
\cite{ChowHuangLiZhou2012,Maas2011}. This geometric viewpoint has motivated
studies of discrete transport
\cite{ErbarRumpfSchmitzerSimon2020,Mielke2011,ChowLiZhou2018} and graph Wasserstein geodesics
\cite{GangboLiMou2019,ErbarMaas2012,
Mielke2013,CuiDieciZhouSISC2022}, Wasserstein Hamiltonian dynamics
\cite{ChowLiZhou2019,CuiDieciZhou2022}, and
Hamilton--Jacobi equations on graph Wasserstein spaces
\cite{GangboMouSwiech2024}. 
Master equations and variational formulations for finite-state MFGs
with nonlinear activation functions were developed in
\cite{GaoLiLiu2024}. For the individual-noise master equation \eqref{eq_intro},
\cite{GangboMunozWuZhang2026} recently established classical well-posedness
and the MFG characteristic representation, together with a variational
formulation in the potential case.

Despite these analytical developments, 
numerical approximation of master equations remains much less developed.
For continuous-state master equations and related partial differential equations, existing results provide convergence rates for finite-state approximations  \cite{BertucciCecchin2024} and finite-particle approximations \cite{GermainPhamWarin2022}. A numerical method based on McKean--Vlasov FBSDEs with a convergence rate bound was analyzed in \cite{ChassagneuxCrisanDelarue2019}. 
 In the finite-state setting,  \cite{CohenLauriereZell2024} developed deep backward and deep Galerkin methods and proved uniform error bounds for the master field in terms of their loss functions.  \cite{CuiDangGraphMFG2026} recently analyzed a structure-preserving time-staggered scheme for the graph MFG system underlying \eqref{eq_intro} and established its discrete variational formulation in the potential case. This scheme, however, approximates the MFG density--value trajectory  associated with a prescribed initial distribution, rather than the master field over the population simplex. To the best of our knowledge, no fully discrete approximation with quantitative convergence estimates has previously been established for the individual-noise graph master equation \eqref{eq_intro}.

The numerical approximation of \eqref{eq_intro} 
involves
several intertwined difficulties. First, the equation is posed on the open probability simplex without prescribed boundary data, while the graph mobilities may degenerate and the nonlinear coefficients may lose uniform regularity as a density component approaches zero. 
Second, the transport field in the population variable depends nonlinearly on the unknown master field through \(\bm B(\mu,\gradG u)\), while the value dynamics simultaneously depend on the population distribution. Thus, the population transport and the value equation are strongly coupled. Third, the numerical approximation concerns the master field over a region of population distributions rather than a single characteristic trajectory, requiring estimates that are uniform with respect to the population state.

In this paper, we develop a fully discrete semi-Lagrangian
method for the master equation \eqref{eq_intro} with individual noise on graphs, by discretizing
both the characteristic evolution in time and the population variable on the
probability simplex. The method combines an implicit
density update and an implicit treatment of the graph-Laplacian terms
with continuous piecewise affine reconstruction of the next-layer master
field at off-grid population states. 
The resulting
implicit update leads to a coupled nonlinear density--value system at each
simplex node. The key to our analysis is an interior localization mechanism based on the exact density characteristics, together with a closure argument for the coupled, solution-dependent numerical recursion.  
For each compact subset of $\cP^\circ(G)$, we construct a
sequence of time-dependent interior submeshes with a separation of order \(\tau\) between the exact one-step density arrivals and the boundary of
the next active region, where $\tau$ is the time step size.  
This separation accommodates the numerical arrival error and ensures that the required interpolation elements remain in a region of uniform coefficient control.  
On these localized submeshes, two nested fixed-point arguments establish local existence and uniqueness of the numerical density--value updates near the exact characteristic pairs.  The derivative bound required by these arguments involves the next-layer nodal error through an inverse estimate, linking the solvability of the current update to the accuracy of the already constructed field. By controlling this dependence as we proceed backward in time, we close the coupled recursion and obtain quantitative error bounds on every compact subset of \(\mathcal P^\circ(G)\).  
Under additional
$\mathcal C^2$-regularity of the exact solution, the method achieves first-order convergence in $\tau$; see Theorem~\ref{conv}.  

An important feature of the master field is that, once computed, it provides a feedback law that can be used to generate the MFG dynamics from different population distributions \cite{CecchinPelino2019,CardaliaguetDelarueLasryLions2019}.   
We therefore evolve the density using the numerical master field and obtain discrete approximations
of the MFG density--value trajectories.  
We show
that the propagation of the master field error through the
feedback dynamics remains controlled and that the resulting density and
value trajectories converge to the corresponding exact MFG
characteristics with the same order as the master field approximation (see Corollary~\ref{feedback}). 
In the potential setting, the same feedback
construction further provides an
approximate solution of the associated planning problem. We reconstruct a flux from the numerical density
trajectory so that the resulting density--flux pair satisfies the
continuity equation and is therefore admissible. When its action is
compared with the optimal value, the variational structure eliminates
the first-order contribution along admissible perturbations, leaving a
quadratic remainder.
Consequently, the action gap is quadratic in the
feedback error and is controlled by the square of the master-field error
bound. This yields asymptotic optimality of the reconstructed numerical
pair and shows that the loss in the planning objective is of higher
order than the underlying trajectory and flux errors (see
Theorem~\ref{actiontheorem}). Finally, numerical experiments report the
observed convergence rates and examine the feedback response to
population redistribution and the influence of graph topology on the resulting
dynamics.

The paper is organized as follows. Section~\ref{role} motivates computing the master field from its feedback interpretation and discusses comparisons with the classical Euclidean setting.  Section~\ref{notation} introduces the graph operators, formulates the master equation and its coefficient assumptions, and recalls the well-posedness result for the exact equation.  
Section~\ref{discretization} presents the fully discrete semi-Lagrangian scheme and the main
results. Section~\ref{properties} develops the interpolation, graph-resolvent, characteristic-localization, and local-solvability estimates used in the analysis, while Section~\ref{analysis} proves the convergence results.  
Section~\ref{numerics} presents the numerical experiments. Auxiliary proofs are collected in the appendix.

\section{Numerical motivation and comparison with the Euclidean setting}
\label{role}
This section discusses the feedback interpretation of the master field and illustrates its relation to a direct MFG computation. We then compare the graph master equation with its Euclidean counterpart. The graph notation and precise equations are introduced in Section~\ref{notation}.

\subsection{Why compute the master field?}
For a prescribed initial density, a direct MFG computation determines one
forward density trajectory together with its associated backward value
trajectory. The master equation provides  a different object: the feedback
field
\[
 (t,\mu)\longmapsto u(t,\mu)\in\bbR^d,
 \qquad (t,\mu)\in[0,T]\times\cP^\circ(G).
\]
For each current population state $(t,\mu)$, this field generates, for
$s\in[t,T]$, the characteristic
\begin{align}\label{feedback-characteristic}
 \dot\rho(s)
 =\divG\bm B\bigl(\rho(s),\gradG u(s,\rho(s))\bigr)
  +\DeltaG\rho(s),
 \qquad
 \rho(t)=\mu,
 \qquad
 \phi(s)=u(s,\rho(s)).
\end{align}
As shown in Section~\ref{notation}, the resulting pair satisfies the graph
MFG system associated with $(t,\mu)$. Hence, for a prescribed population state,
solving the MFG system directly and evaluating the master field along its
characteristic yield the same density--value trajectory. The advantage of the
master field representation becomes apparent when feedback is needed at
population states away from a previously computed MFG trajectory. A direct MFG computation determines the value variable only along one trajectory and must be repeated for a new initial population state, whereas the master field provides a feedback representation over the population simplex.  
The left panel of Figure~\ref{feedbackfamily} summarizes this distinction.  
Such repeated feedback queries arise, for example, in
network route-choice models, where the current distribution of agents may be
updated from observations and used to predict the subsequent population
dynamics; see \cite{SalhabLeNyMalhame2018}.

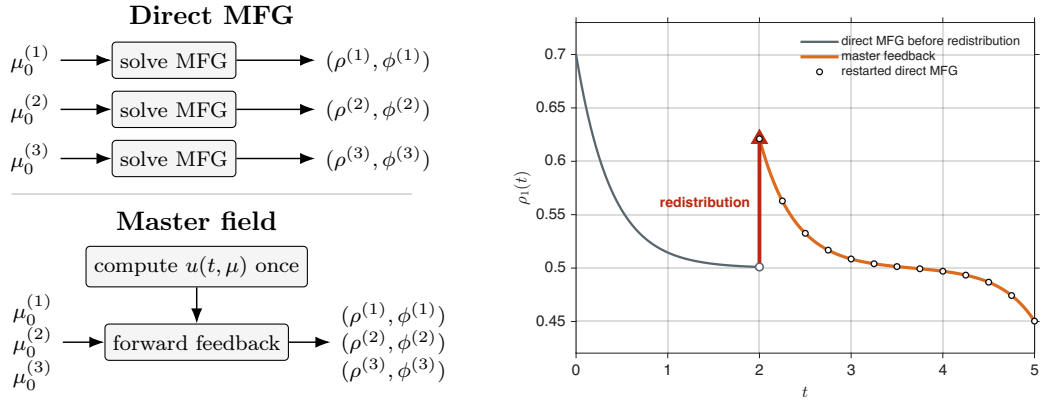
\begin{figure}[htbp]
\centering
\begin{minipage}[t]{0.49\textwidth}
\centering
\begin{tikzpicture}[
  x=1cm,y=1cm,
  every node/.style={font=\scriptsize},
  box/.style={draw,rounded corners=1.5pt,fill=gray!8,
              minimum height=0.48cm,align=center,inner xsep=3pt},
  flow/.style={-{Latex[length=2mm]},line width=0.55pt}
]
\node[font=\small\bfseries] at (2.65,3.15) {Direct MFG};
\foreach \j/\y in {1/2.55,2/1.90,3/1.25}{
  \node (mu\j) at (0.25,\y) {$\mu_0^{(\j)}$};
  \node[box] (solve\j) at (2.15,\y) {solve MFG};
  \node (path\j) at (4.85,\y) {$(\rho^{(\j)},\phi^{(\j)})$};
  \draw[flow] (mu\j) -- (solve\j);
  \draw[flow] (solve\j) -- (path\j);
}
\draw[gray!55] (0,0.78) -- (5.3,0.78);
\begin{scope}[xshift=-0.2cm]
\node[font=\small\bfseries] at (2.65,0.42) {Master field};
\node[box] (master) at (2.65,-0.18) {compute $u(t,\mu)$ once};
\node[box] (forward) at (2.65,-1.18) {forward feedback};
\node[align=right,anchor=east] (mus) at ([xshift=-0.55cm]forward.west)
  {$\mu_0^{(1)}$\\$\mu_0^{(2)}$\\$\mu_0^{(3)}$};
\node[align=left,anchor=west] (rhos) at ([xshift=0.55cm]forward.east)
  {$(\rho^{(1)},\phi^{(1)})$\\
   $(\rho^{(2)},\phi^{(2)})$\\
   $(\rho^{(3)},\phi^{(3)})$};
\draw[flow] (master) -- (forward);
\draw[flow] (mus) -- (forward);
\draw[flow] (forward) -- (rhos);
\end{scope}
\end{tikzpicture}

\smallskip
\end{minipage}\hspace{-0.5em}
\begin{minipage}[t]{0.45\textwidth}
\centering
\includegraphics[width=\linewidth]
{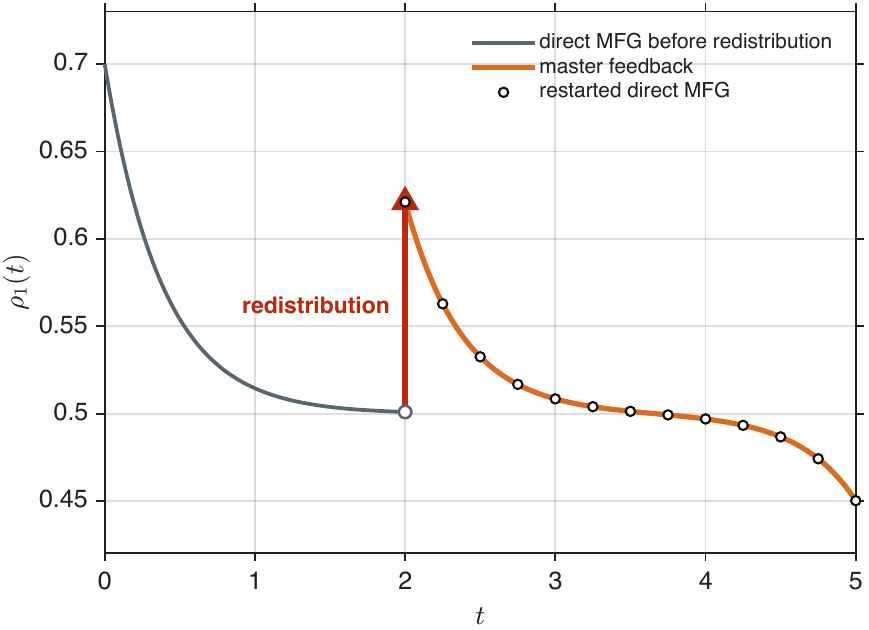}

\smallskip 
\end{minipage}
\caption{Left: A direct MFG computation produces a density--value trajectory for one
prescribed initial population state, whereas a precomputed master field can be
reused as feedback for different population states.
Right: At $t=2$, redistribution moves the population from the baseline MFG
trajectory to a new off-trajectory state. From this same state, the solid orange
curve is generated by the precomputed master feedback, while the open markers
show an independently restarted direct MFG trajectory.
Parameters: $\omega_{12}=1$, $\cF(\mu)=-\frac{3}{2}\norm{\mu}^2$,
$g(\mu)=2(\mu-\mu_{\rm tar})$ with $\mu_{\rm tar}=(1/3,2/3)$,
$\epsilon=0.01$, 128 mesh intervals, and $\tau=5/640$.
The direct MFG uses the same time step. The maximum Euclidean distance between
the densities generated by the precomputed master feedback and the direct MFG on time interval $[2,5]$ is
$2.17\times10^{-6}$.}
\label{feedbackfamily}
\end{figure}

The right panel of Figure~\ref{feedbackfamily} gives a numerical illustration of this feedback
reuse. We consider $d=2$ and first compute the
numerical master field on $[0,5]\times\mathcal P_\epsilon(G)$ with
$\epsilon=0.01$. Starting from \(\mu_0=(0.70,0.30)\), the baseline MFG density trajectory satisfies \(\rho_1(2^-)=0.5\). At $t=2$, we impose an 
instantaneous redistribution that transfers a mass of $0.12$ from vertex
$2$ to vertex $1$, moving the population to $\rho_1(2^+)=0.62$. This redistribution produces a new state away from the baseline trajectory. Since the master field has already been computed on the computational simplex, the post-redistribution trajectory is generated from \eqref{feedback-characteristic} with \(\rho(2^+)
\) as the new initial population state, without recomputing the backward value equation. By contrast, a direct MFG computation requires restarting the coupled forward--backward system on \([2,T]\) from \(\rho(2^+)\). As shown in the right panel of Figure \ref{feedbackfamily}, the trajectory generated by the precomputed master feedback nearly coincides with the independently restarted direct MFG trajectory, providing a numerical validation of the computed feedback.  
This feedback reuse requires numerically approximating \(\mu\mapsto u(t,\mu)\) over a region of the probability simplex. Section~\ref{numerics} further examines redistribution responses and their dependence on graph topology.

\subsection{Comparison with the Euclidean setting}
\label{euclidean-comparison}

We next compare the individual-noise structure of the graph master equation
with its Euclidean counterpart by considering the stochastic dynamics of the
individual state.  In a Euclidean state space, individual noise is usually generated
by Brownian motion, whereas on a finite graph it is generated by the jumps of
a continuous-time Markov chain \cite{ChowHuangLiZhou2012}.  Their
infinitesimal generators determine the corresponding diffusion operators in
the master equations.

Let $\mathrm dX_t=\sqrt{2\nu}\,\mathrm dW_t$ on the periodic state space
$\mathbb T^q$, where $q\in\mathbb N_+,$ \(\nu>0\) is the diffusion coefficient and $W_t$ is a \(q\)-dimensional Brownian motion, and let $Y_t$ be a continuous-time Markov chain on $V$ with
transition rate $\omega_{ij}$ from $i$ to $j$.  Their infinitesimal generators
satisfy
\begin{align*}
&\lim_{r\downarrow0}\frac1r
 \mathbb E\bigl[\psi(X_{t+r})-\psi(X_t)\mid X_t=x\bigr]
 =\nu\Delta_x\psi(x),\quad \psi\in\mathcal C^2(\mathbb T^q),\\
&\lim_{r\downarrow0}\frac1r
 \mathbb E\bigl[v(Y_{t+r})-v(Y_t)\mid Y_t=i\bigr] 
 =\sum_{j\sim i}\omega_{ij}(v_j-v_i)=(\DeltaG v)_i,
 \quad v:V\to\mathbb R, 
 \end{align*} where $j\sim i$ means that $(i,j)\in E.$ Thus, the Euclidean and graph Laplacians are the infinitesimal generators of Brownian diffusion and Markov jump dynamics, respectively. In the graph setting, the vertices form the state space of the agents, while the topology and edge weights determine the admissible transitions and their rates; see
\cite{ChowHuangLiZhou2012,Maas2011}.  In the fully discrete method introduced in Section \ref{discretization}, the graph remains fixed, while the population variable is discretized on a simplicial mesh.

This correspondence carries over to the master equations.  Let $U(t,x,m)$
denote a smooth Euclidean master field, where $t\in[0,T],x\in\mathbb T^q$ is the state
of the representative agent and $m\in\mathcal P(\mathbb T^q)$ is the population distribution. Individual Brownian noise contributes a Laplacian acting on the individual state and an additional diffusion term acting through the population distribution. In the Euclidean setting, these terms take the form 
$$ \nu\Delta_xU(t,x,m) \quad\text{and}\quad \nu\int_{\mathbb T^q} \Delta_y\frac{\delta U}{\delta m}(t,x,m,y)\,m(\mathrm dy), $$ 
respectively, where \(\delta U/\delta m\) denotes the functional derivative with respect to the population measure. Their finite-state counterparts are the graph Laplacian  \(\Delta_Gu\) acting on the value index and \(D_\mu u[\Delta_G\mu]\), respectively.  
Together with population transport,
the population component of \eqref{eq_intro} is 
$ 
 D_\mu u\bigl[\divG\bm B(\mu,\gradG u)+\DeltaG\mu\bigr].$   Thus the graph topology and edge
weights enter the master equation through both the individual-state dynamics and the evolution of the population distribution. 

The main analytical issue arises near the boundary of the population
simplex, which consists of densities with at least one vanishing component.
For the logarithmic-mean model given in Example \ref{exam}, $\theta(\mu_i,\mu_j)=0$ when
$\mu_i\mu_j=0$, and the mobility and its derivatives are not uniformly
regular as $\mu$ approaches $\partial\cP(G)$.  The master equation is posed
on $\cP^\circ(G)$, with no Dirichlet or Neumann data prescribed on the
simplex boundary. Hence, the convergence analysis is carried out on interior
regions that stay a positive distance from $\partial\cP(G)$, where the
coefficients are uniformly controlled and no boundary values are needed.
This is the reason for the characteristic localization introduced below (see Proposition~\ref{tube}).

\section{Preliminaries and master equations on graphs}\label{notation}

In this section, we introduce the graph notation, formulate the vector master
equation, and state the structural assumptions.  We then state the existence,
uniqueness, and regularity of the classical solution of \eqref{eq_intro}, with the
logarithmic-mean quadratic model as an example.

\subsection{Notation}

Let $G=(V,E,\omega)$ be a finite connected undirected graph without
self-loops, where $V=\{1,\ldots,d\}$ is the vertex set with $d\ge2$, $E$ is the edge set, and $\omega=(\omega_{ij})_{1\leq i,j\leq d}$ is the weight matrix.  Every undirected edge
occurs in $E$ with both orientations, and $\omega_{ij}=\omega_{ji}>0$ on
$E$; set $\omega_{ij}=0$ when $(i,j)\notin E$.  We write $j\sim i$ when $(i,j)\in E$. 
For $x,y\in\bbR^d$, let 
$\ip{x}{y}=\sum_{i=1}^d x_i y_i$, $\norm{x}=\ip{x}{x}^{1/2}$, and
$\norm{x}_{\ell^\infty}=\max_{i=1}^d|x_i|$.  Let
$\bbS^{d\times d}$ be the space of skew-symmetric matrices $p$, with
$p_{ji}=-p_{ij}$ and $p_{ij}=0$ when $(i,j)\notin E$, equipped with
\begin{align*}
\ip{p}{q}_E=\frac12\sum_{(i,j)\in E}p_{ij}q_{ij},\quad 
\norm{p}_E=\ip{p}{p}_E^{1/2}.
\end{align*}  We also write
$\norm{p}_{\ell^\infty(E)}=\max_{(i,j)\in E}|p_{ij}|$. 
For $v\in\bbR^d$ and $p\in\mathbb S^{d\times d}$, define
$ 
(\gradG v)_{ij}:=\sqrt{\omega_{ij}}(v_i-v_j)$ and $(\divG p)_i:=-\sum_{j\sim i}\sqrt{\omega_{ij}}p_{ij}.$ 
Then
\begin{equation}\label{ibp}
 \ip{\divG p}{v}=-\ip{p}{\gradG v}_E.
\end{equation}
The graph Laplacian is $\DeltaG:=\divG\gradG$, so
\begin{equation}\label{lap}
 (\DeltaG v)_i=\sum_{j\sim i}\omega_{ij}(v_j-v_i),
 \qquad
 \ip{\DeltaG v}{w}=-\ip{\gradG v}{\gradG w}_E.
\end{equation} 

The probability simplex of the graph $G$ is $
 \cP(G)=\bigl\{\mu\in[0,1]^d:\textstyle\sum_{i=1}^d\mu_i=1\bigr\}.$
For $0\le\epsilon<1/d$, set
$
 \cP_\epsilon=\bigl\{\mu\in\cP(G):\mu_i\ge\epsilon,\;1\leq i\leq d\bigr\},\;
 \cP_\epsilon^\circ
 =\bigl\{\mu\in\cP(G):\mu_i>\epsilon,
 \ 1\le i\le d\bigr\}.
$ 
Thus $\cP_\epsilon^\circ$ is the relative interior of $\cP_\epsilon$ in $\bbR_1^d$, and
$\cP_0=\cP(G)$ and
$\cP_0^\circ=\cP^\circ(G)$. 
The tangent space and affine hull are, respectively,
$$
 \bbR_0^d=\bigl\{\eta\in\bbR^d:\textstyle\sum_{i=1}^d\eta_i=0\bigr\},
 \qquad
 \bbR_1^d=\bigl\{\mu\in\bbR^d:\textstyle\sum_{i=1}^d\mu_i=1\bigr\}.
$$
If $f:\cP^\circ(G)\to\bbR$ is differentiable, then
$D_\mu f(\mu)\in\mathcal L(\bbR_0^d,\bbR)$. Choose a vector
$q_f(\mu)\in\bbR^d$ such that
$D_\mu f(\mu)[\eta]=\ip{q_f(\mu)}{\eta}$ for every
$\eta\in\bbR_0^d$, and define
$\nabla_{\mathcal W}f(\mu):=\gradG q_f(\mu)$.
This definition is independent of the representative, since any two choices of
$q_f(\mu)$ differ by a multiple of $\bm1$, and $\gradG\bm1=0$. Here, we denote by
$\bm1=(1,\ldots,1)^T\in\bbR^d$. 
For the vector master field $u=(u_1,\ldots,u_d)$, 
$\gradG u(t,\mu)$ takes graph differences in the agent-state
index;  
$\nabla_{\mathcal W}u_i(t,\mu)=\gradG q_{u_i}(t,\mu)$ is obtained
from the derivative of the scalar component $u_i$ with respect to the
population variable.

For a finite set $A$, $|A|$ denotes its cardinality. 
For a subset $A$ of a Euclidean space,
$\operatorname{diam}(A):=\sup_{x,y\in A}\norm{x-y}$.
For positive quantities $a$ and $b$,
we write $a\asymp b$ if there are constants $0<c\le C$, independent of discretization
parameters, such that $cb\le a\le Cb$.  
Throughout the paper, we use $\mu$ for a generic population variable in the master equation and
$\rho$ for a density trajectory along the associated characteristics.  
For maps defined on $\cP^\circ(G)$, 
all derivatives in the density variable act on
the tangent space $\bbR_0^d$.

\subsection{The vector master equation}

Let
$\bm H:\cP^\circ(G)\times\bbS^{d\times d}\to\bbR^d$ and
$\bm B:\cP^\circ(G)\times\bbS^{d\times d}\to\bbS^{d\times d}$ be the
Hamiltonian and transport coefficients, respectively, and let
$g:\cP(G)\to\bbR^d$ be the terminal map. 
We study the vector master field
$u:[0,T]\times\cP^\circ(G)\to\bbR^d$ solving \eqref{eq_intro},
with the terminal condition $u(T,\mu)=g(\mu)$.  
We write
\begin{align}\label{mathcalA}
 \mathcal A(\mu,z):=\divG\bm B(\mu,\gradG z)+\DeltaG\mu.
\end{align} 
This is the density velocity  in \eqref{eq_intro}. Because $\bm B(\mu,\gradG z)$ is skew-symmetric, $\sum_i(\divG\bm B)_i=0$; likewise $\sum_i(\DeltaG\mu)_i=0$. Hence $\mathcal A(\mu,z)\in\bbR_0^d$, so the characteristic dynamics preserve total mass. 
For $(t,\mu)\in[0,T]\times\cP^\circ(G)$, the characteristic system
associated with \eqref{eq_intro} is
\begin{equation}\label{mfg}
\begin{cases}
 \dot\phi(s)=\bm H(\rho(s),\gradG\phi(s))-\DeltaG\phi(s),\\
\dot\rho(s)=\mathcal A(\rho(s),\phi(s)),\\
 \rho(t)=\mu,\qquad \phi(T)=g(\rho(T)),\qquad t\le s\le T,
\end{cases}
\end{equation}
which is referred to as the mean field game (MFG) on the graph.
Indeed, if $u$ solves \eqref{eq_intro}, define, for
$(t,\mu)\in[0,T]\times\cP^\circ(G)$ and $t\le s\le T$,
\begin{equation}\label{representation}
\begin{aligned}
 \dot\rho(s;t,\mu)
 &=\mathcal A\bigl(\rho(s;t,\mu),u(s,\rho(s;t,\mu))\bigr),
 &\qquad \rho(t;t,\mu)&=\mu,\\
 \phi(s;t,\mu)&:=u\bigl(s,\rho(s;t,\mu)\bigr).
\end{aligned}
\end{equation}
Then 
$
 \frac{\mathrm d}{\mathrm ds}\phi(s;t,\mu) =\bm H\bigl(\rho(s;t,\mu),\gradG\phi(s;t,\mu)\bigr)
  -\DeltaG\phi(s;t,\mu)
$ 
and $\phi(T;t,\mu)=g(\rho(T;t,\mu))$.  Thus
$(\rho(\cdot;t,\mu),\phi(\cdot;t,\mu))$ satisfies \eqref{mfg}.
This characteristic system will supply the density and value
variables used by the semi-Lagrangian discretization (see \eqref{scheme}). 

Equation \eqref{eq_intro} has the following equivalent componentwise formulation, which makes the individual-noise contribution explicit.  Define
$ (\Delta_{\rm ind}u)_i=-\ip{\nabla_{\mathcal W}u_i}{\gradG\mu}_E.$
By \eqref{ibp}, we have 
$$
 D_\mu u_i\bigl[\divG\bm B(\mu,\gradG u)+\DeltaG\mu\bigr]
 =-\ip{\nabla_{\mathcal W}u_i}{\bm B(\mu,\gradG u)}_E
 +(\Delta_{\rm ind}u)_i.
$$
Thus \eqref{eq_intro} is equivalent to
$$
 \partial_tu_i-\ip{\nabla_{\mathcal W}u_i}{\bm B(\mu,\gradG u)}_E
 +(\Delta_{\rm ind}u)_i
 =\bm H_i(\mu,\gradG u)-(\DeltaG u)_i,
 \qquad 1\le i\le d,
$$
which is the componentwise form of the master equation with the individual-noise
operator; see also \cite{GangboMunozWuZhang2026}. 
To obtain the existence of the exact solution, we introduce the following regularity, coercivity, and boundary degeneracy
conditions for the coefficients; see \cite[(3.1)--(3.4)]{GangboMunozWuZhang2026}.

\begin{assumption}
\label{structure}
Let
$\bm H\in\mathcal C^1(\cP^\circ(G)\times\bbS^{d\times d};\bbR^d)$,
$\bm B\in\mathcal C^1(\cP^\circ(G)\times\bbS^{d\times d};
\bbS^{d\times d})$, and 
$g\in\mathcal C^1(\cP^{\circ}(G);\bbR^d)$. For some $C_0,C_H,C_g\ge0$,
\begin{align}\label{data}
 \ip{\bm B(\mu,p)}{p}_E\ge\ip{\bm H(\mu,p)}{\mu}-C_0,\quad 
\bm H_i(\mu,p)\ge-C_H\quad
 \text{for }(\mu,p)\in\cP^\circ(G)\times\bbS^{d\times d},
 \end{align}
 and $
 \norm{g(\mu)}_{\ell^\infty}\le C_g
 $ for $\mu\in\cP(G)$.
For every edge $(i,j)$ and every
$(\mu,p)\in\cP^\circ(G)\times\bbS^{d\times d}$,
\begin{equation}\label{boundary}
 |\bm B_{ij}(\mu,p)|
 \le(\mu_i+\mu_j)\mathfrak a(\mu_j/\mu_i,p),
\end{equation}
where $\mathfrak a:(0,\infty)\times\bbS^{d\times d}\to[0,\infty)$ is locally
bounded, and $\mathfrak a(u,p)\to0$ as $u\to0^+$ or $u\to\infty$,
locally uniformly in $p\in\bbS^{d\times d}$.
\end{assumption}

To ensure the uniqueness of the exact solution, we impose the terminal and differential Lasry--Lions monotonicity conditions in the
next assumption (see \cite[(3.7)]{GangboMunozWuZhang2026}). 

\begin{assumption}\label{mono}
The terminal map $g$ is monotone:
\begin{equation}\label{terminal}
 \ip{g(\mu)-g(\nu)}{\mu-\nu}\ge0,
 \qquad \mu,\nu\in\cP^\circ(G).
\end{equation}
Moreover, the differential Lasry--Lions condition holds: for
$(\mu,p)\in\cP^\circ(G)\times\bbS^{d\times d}$ and
$(\eta,r)\in\bbR_0^d\times\bbS^{d\times d}$,
\begin{equation}\label{dll}
 \mathcal M_{\mu,p}(\eta,r)<0
 \qquad\text{whenever }(\eta,r)\ne(0,0),
\end{equation}
where
\[
 \mathcal M_{\mu,p}(\eta,r)
 :=\ip{D_\mu\bm H(\mu,p)[\eta]+D_p\bm H(\mu,p)[r]}{\eta}
 -\ip{D_\mu\bm B(\mu,p)[\eta]+D_p\bm B(\mu,p)[r]}{r}_E.
\]
\end{assumption} 
We next present an example of coefficients satisfying the two assumptions.
\begin{example}\label{exam}
For $\mu\in\cP^\circ(G)$ and
$p\in\bbS^{d\times d}$, let
$\bm H(\mu,p)=\nabla_\mu\cH_\theta(\mu,-p)-\alpha\mu$ and
$\bm B(\mu,p)=\theta(\mu)p$ with 
$[\theta(\mu)p]_{ij}:=\theta(\mu_i,\mu_j)p_{ij}$ for $(i,j)\in E$, where 
\begin{equation}\label{model}
 \cH_\theta(\mu,q)
 =\frac14\sum_{(i,j)\in E}\theta(\mu_i,\mu_j)q_{ij}^2,
 \qquad
 \theta(r,s)=\frac{r-s}{\log r-\log s},
\end{equation}
with $\theta(r,r)=r$.  
For $\alpha>0$, these coefficients satisfy the conditions on $\bm H$ and
$\bm B$ in Assumption~\ref{structure} and the differential condition
\eqref{dll}; see \cite[Example~2.4]{CuiDangGraphMFG2026}.
Together with a terminal map $g$ satisfying the remaining conditions in
Assumptions~\ref{structure} and \ref{mono} (e.g. $g(\mu)=(\mu-\mu_{\rm tar})$ with some fixed $\mu_{\rm tar}\in\mathcal P^{\circ}(G)$), the model satisfies both
assumptions.
\end{example}

The next proposition gives the existence, uniqueness, regularity, and
Lasry--Lions monotonicity of the classical solution to \eqref{eq_intro}.
The $\mathcal C^1$-well-posedness follows by adapting the proofs of
\cite[Propositions~3.5--3.6]{GangboMunozWuZhang2026} to the simplex.
This adaptation, together with the $\mathcal C^2$ regularity and
monotonicity, is justified in the appendix.

\begin{proposition}\label{regularity}
Let Assumptions~\ref{structure} and \ref{mono} hold.  Then
\eqref{eq_intro} admits a unique
classical solution $u\in\mathcal C^1([0,T]\times\cP^\circ(G);\bbR^d)$. More precisely, 
for every $(t,\mu)\in[0,T]\times\cP^\circ(G)$, let $(\rho,\phi)$ be the
unique solution of \eqref{mfg}.  Then
\[
u(t,\mu)=\phi(t),\qquad
\phi(s)=u(s,\rho(s)),\quad t\le s\le T.
\] 
Moreover, $u(t,\cdot)$ is monotone:
\begin{equation}\label{umono}
 \ip{u(t,\mu)-u(t,\nu)}{\mu-\nu}\ge0,
 \qquad t\in[0,T],\quad \mu,\nu\in\cP^\circ(G).
\end{equation}
If in addition,
$
 \bm H\in\mathcal C^2(\cP^\circ(G)\times\bbS^{d\times d};\bbR^d),\;
 \bm B\in\mathcal C^2(\cP^\circ(G)\times\bbS^{d\times d};
 \bbS^{d\times d}),$
and $g|_{\cP^\circ(G)}\in
\mathcal C^2(\cP^\circ(G);\bbR^d)$, then
$u\in\mathcal C^2([0,T]\times\cP^\circ(G);\bbR^d)$.  Consequently,
$D_\mu u$ is locally Lipschitz in $\mu$, uniformly in $t$, namely, for every compact
convex set
$K\Subset\cP^\circ(G)$, there is $L_K>0$ such that
\begin{equation}\label{dulip}
 \norm{D_\mu u(t,\mu)-D_\mu u(t,\nu)}_{\mathcal L(\bbR_0^d,\bbR^d)}
 \le L_K\norm{\mu-\nu},
 \qquad t\in[0,T],\quad \mu,\nu\in K.
\end{equation}
\end{proposition}
\section{Fully discrete semi-Lagrangian method and main results}\label{discretization}

This section presents the fully discrete semi-Lagrangian method for
\eqref{eq_intro} and states its main result, which establishes local well-posedness and error estimates. As a corollary, we derive convergence estimates for the density and value trajectories generated by the numerical feedback.   
In the potential setting, we further present the asymptotic optimality of the associated numerical feedback flux and a quadratic estimate for its action gap.  

\subsection{Fully discrete semi-Lagrangian method}\label{sec4-1}

Let $N\ge1$, $\tau=T/N$, and
$t_n=n\tau$ for $0\le n\le N$.  Let
$\epsilon\in[0,1/d)$ and let $J\ge d$ be an integer. Take the
uniform simplicial triangulation $\mathcal T_J$ of $\cP(G)$ as specified in \cite[Section~4.1]{CuiDangMou2026}, whose vertex
set is
\[
 \Bigl\{\frac kJ:k=(k_1,\ldots,k_d)\in\mathbb N_0^d,
 \ \sum_{i=1}^dk_i=J\Bigr\}.
\]
Define the affine map 
\[
 F_\epsilon:\cP(G)\to\cP_\epsilon,\qquad
F_\epsilon(\mu):=\epsilon(1,\ldots,1)+(1-d\epsilon)\mu,
\]
and set
$ 
 h=h(J,\epsilon):=(1-d\epsilon)
 \max_{S\in\mathcal T_J}\operatorname{diam}(S).
$ 
The triangulation of $\cP_\epsilon$ and its vertex set are
\begin{equation}
 \mathfrak T_{h,\epsilon}
 :=\bigl\{F_\epsilon(S):S\in\mathcal T_J\bigr\},
 \quad
 \mathfrak M_{h,\epsilon}
 :=\Bigl\{\epsilon(1,\ldots,1)+(1-d\epsilon)\frac kJ:
 k\in\mathbb N_0^d,\ \sum_{i=1}^dk_i=J\Bigr\}.
\end{equation}
Since $F_\epsilon$ scales all distances by $1-d\epsilon$, the resulting
family is shape-regular and quasi-uniform.  Hence there is
$c>0$, independent of $J$,
such that
$
ch\le\operatorname{diam}(S)\le h,\; S\in\mathfrak T_{h,\epsilon}. 
$ 
Moreover,
$ 
 h\asymp(1-d\epsilon)J^{-1}.
$ 
When $\epsilon=0$, we write
$\mathfrak M_h:=\mathfrak M_{h,0}$ and
$\mathfrak T_h:=\mathfrak T_{h,0}=\mathcal T_J$.

\begin{figure}[htbp]
\centering
\begin{tikzpicture}[x=1cm,y=1cm,>=Latex]
  \tikzset{
    mesh/.style={black!25,line width=0.45pt},
    boundary/.style={black!65,line width=0.75pt},
    nodept/.style={circle,fill=white,draw=black!55,line width=0.4pt,
      inner sep=0pt,minimum size=2.7pt}
  }

  \begin{scope}[xshift=-3.5cm]
    \coordinate (A) at (0,3.0);
    \coordinate (B) at (-2.1,0);
    \coordinate (C) at (2.1,0);
    \node[font=\small] at (0,3.65)
      {reference mesh $\mathcal T_J$};
    \coordinate (Pone) at (barycentric cs:A=2,B=1,C=1);
    \coordinate (Ptwo) at (barycentric cs:A=1,B=2,C=1);
    \coordinate (Pthree) at (barycentric cs:A=1,B=1,C=2);
    \filldraw[fill=blue!14,draw=blue!65!black,line width=0.7pt]
      (Pone)--(Ptwo)--(Pthree)--cycle;
    \foreach \j in {0,...,4}{
      \pgfmathsetmacro{\s}{\j/4}
      \draw[mesh] ($(A)!\s!(B)$)--($(A)!\s!(C)$);
      \draw[mesh] ($(B)!\s!(A)$)--($(B)!\s!(C)$);
      \draw[mesh] ($(C)!\s!(A)$)--($(C)!\s!(B)$);
    }
    \draw[boundary] (A)--(B)--(C)--cycle;
    \foreach \kone in {0,...,4}{
      \pgfmathtruncatemacro{\rem}{4-\kone}
      \foreach \ktwo in {0,...,\rem}{
        \pgfmathtruncatemacro{\kthree}{4-\kone-\ktwo}
        \node[nodept] at
          (barycentric cs:A=\kone,B=\ktwo,C=\kthree) {};
      }
    }
    \node[above=2pt,font=\scriptsize] at (A) {$e_1$};
    \node[below left=1pt,font=\scriptsize] at (B) {$e_2$};
    \node[below right=1pt,font=\scriptsize] at (C) {$e_3$};
    \node[font=\scriptsize] at (0,-0.55) {$\mathcal P(G)$, $J=4$};
  \end{scope}

  \draw[-{Latex[length=2mm]},line width=0.9pt]
    (-0.85,1.5)--(0.85,1.5)
    node[midway,above=3pt,font=\small] {$F_\epsilon$};

  \begin{scope}[xshift=3.5cm]
    \coordinate (Ao) at (0,3.0);
    \coordinate (Bo) at (-2.1,0);
    \coordinate (Co) at (2.1,0);
    \coordinate (G) at (0,1.0);
    \coordinate (An) at ($(G)!0.62!(Ao)$);
    \coordinate (Bn) at ($(G)!0.62!(Bo)$);
    \coordinate (Cn) at ($(G)!0.62!(Co)$);
    \coordinate (Eo) at ($(Bo)!0.5!(Co)$);
    \coordinate (En) at ($(Bn)!0.5!(Cn)$);
    \node[font=\small] at (0,3.65)
      {mesh $\mathfrak T_{h,\epsilon}$};
    \draw[black!35,dashed,line width=0.65pt] (Ao)--(Bo)--(Co)--cycle;
    \coordinate (Qone) at (barycentric cs:An=2,Bn=1,Cn=1);
    \coordinate (Qtwo) at (barycentric cs:An=1,Bn=2,Cn=1);
    \coordinate (Qthree) at (barycentric cs:An=1,Bn=1,Cn=2);
    \filldraw[fill=blue!14,draw=blue!65!black,line width=0.7pt]
      (Qone)--(Qtwo)--(Qthree)--cycle;
    \foreach \j in {0,...,4}{
      \pgfmathsetmacro{\s}{\j/4}
      \draw[mesh] ($(An)!\s!(Bn)$)--($(An)!\s!(Cn)$);
      \draw[mesh] ($(Bn)!\s!(An)$)--($(Bn)!\s!(Cn)$);
      \draw[mesh] ($(Cn)!\s!(An)$)--($(Cn)!\s!(Bn)$);
    }
    \draw[blue!65!black,line width=0.8pt] (An)--(Bn)--(Cn)--cycle;
    \draw[red!70!black,line width=0.8pt] (Eo)--(En);
    \draw[red!70!black,line width=0.8pt]
      ($(Eo)+(-0.08,0)$)--($(Eo)+(0.08,0)$);
    \draw[red!70!black,line width=0.8pt]
      ($(En)+(-0.08,0)$)--($(En)+(0.08,0)$);
    \node[right=2pt,font=\scriptsize,text=red!70!black]
      at ($(Eo)!0.5!(En)$) {$\epsilon$};
    \foreach \kone in {0,...,4}{
      \pgfmathtruncatemacro{\rem}{4-\kone}
      \foreach \ktwo in {0,...,\rem}{
        \pgfmathtruncatemacro{\kthree}{4-\kone-\ktwo}
        \node[nodept] at
          (barycentric cs:An=\kone,Bn=\ktwo,Cn=\kthree) {};
      }
    }
    \node[font=\scriptsize,text=black!55] at (1.55,2.45)
      {$\mathcal P(G)$};
    \node[font=\scriptsize,text=blue!65!black] at (0,-0.55)
      {$\mathcal P_\epsilon$};
  \end{scope}
\end{tikzpicture}
\caption{The mesh map $F_\epsilon$ for $d=3$ and $J=4$.
The homothety sends $\mathcal T_J$ on $\mathcal P(G)$ to
$\mathfrak T_{h,\epsilon}$ on $\mathcal P_\epsilon$.  The highlighted 
elements correspond under $F_\epsilon$.}
\end{figure}
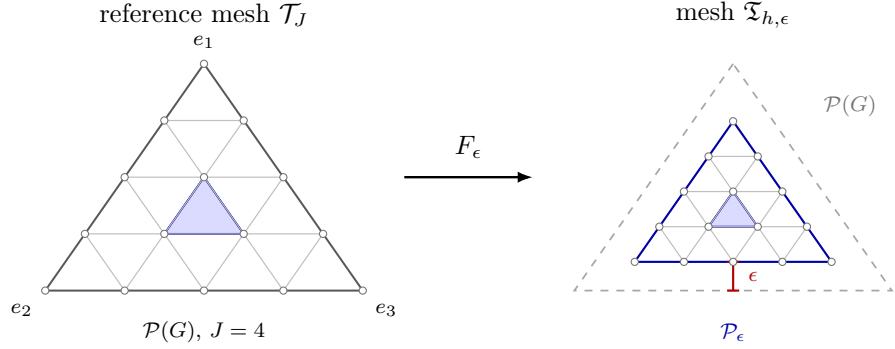

We enumerate the nodes as
$\mu^\ell\in\mathfrak M_{h,\epsilon}$,
$\ell=1,\ldots,|\mathfrak M_{h,\epsilon}|$.
For a map $v:\cP_\epsilon\to\bbR^d$, let
$P_{h,\epsilon}v=(v(\mu^\ell))_{\mu^\ell\in\mathfrak M_{h,\epsilon}}$
be its nodal restriction.  For a nodal field
$z_h:\mathfrak M_{h,\epsilon}\to\bbR^d$, let
$\mathcal I_{h,\epsilon}z_h$ denote its continuous piecewise affine
reconstruction on $\cP_\epsilon$.  When $\epsilon=0$, we write
$P_h:=P_{h,0}$ and $\mathcal I_h:=\mathcal I_{h,0}$.
For $\mu\in\cP_\epsilon$, choose an element
$S\in\mathfrak T_{h,\epsilon}$ containing
$\mu$, and denote its vertices by $\mu^{\ell_1},\ldots,\mu^{\ell_d}$.
The barycentric coordinates of $\mu$ in $S$ are the unique numbers
$\lambda_1(\mu),\ldots,\lambda_d(\mu)$ satisfying
\begin{align*}
\lambda_r(\mu)\ge0,\quad  \sum_{r=1}^d\lambda_r(\mu)=1,\quad\text{ and }\;
\mu=\sum_{r=1}^d\lambda_r(\mu)\mu^{\ell_r}.
\end{align*}  Then 
\begin{equation}\label{interpolant}
 (\mathcal I_{h,\epsilon}z_h)(\mu)
 =\sum_{r=1}^d\lambda_r(\mu)z_h(\mu^{\ell_r}).
\end{equation} 
For the full grid $\mathfrak M_{h,\epsilon}$, set
$ 
 \mathfrak M_{h,\epsilon}^\circ:=\mathfrak M_{h,\epsilon}\cap\cP_\epsilon^\circ,\;
 \partial\mathfrak M_{h,\epsilon}:=\mathfrak M_{h,\epsilon}\cap\partial\cP_\epsilon.$ 
For all sufficiently small $h$ and each
$y\in\partial\mathfrak M_{h,\epsilon}$, choose and fix a nearest
interior node $y^\circ\in\mathfrak M_{h,\epsilon}^\circ$, so that
$\norm{y^\circ-y}=\min_{x\in\mathfrak M_{h,\epsilon}^\circ}\norm{x-y}$.
For $z_h:\mathfrak M_{h,\epsilon}^\circ\to\bbR^d$, define its full grid
extension $\mathcal E_hz_h:\mathfrak M_{h,\epsilon}\to\bbR^d$ by 
\begin{equation}\label{gridextension}
 (\mathcal E_hz_h)(y):=
 \begin{cases}
 z_h(y),&y\in\mathfrak M_{h,\epsilon}^\circ,\\
 z_h(y^\circ),&y\in\partial\mathfrak M_{h,\epsilon}. 
 \end{cases}
\end{equation}

Now we present the fully discrete semi-Lagrangian scheme. Starting from $u_h^N= P_{h,\epsilon}g$, proceed backward for
$n=N-1,\ldots,0$.  
At every node $\mu^\ell\in\mathfrak M_{h,\epsilon}^\circ$, seek
$(\mu_\ell^+,u_\ell^n)\in\cP_\epsilon^\circ\times\bbR^d$ satisfying
\begin{subequations}\label{scheme}
\begin{empheq}[left=\empheqlbrace]{align}
 \mu_\ell^+
 &=\mu^\ell+\tau
 \mathcal A(\mu_\ell^+,u_\ell^n),
 \label{schemerho}\\
 (I-\tau\DeltaG)u_\ell^n
 &=\mathcal I_{h,\epsilon}u_h^{n+1} (\mu_\ell^+)
 -\tau\bm H(\mu_\ell^+,\gradG u_\ell^n),
 \label{schemeu}
\end{empheq}
\end{subequations} where $\mathcal A$ is given in \eqref{mathcalA}. 
The values $u_\ell^n$ define the $n$th time level nodal field on the full grid by
\[
 u_h^n:=\mathcal E_h
 \bigl((u_\ell^n)_{\mu^\ell\in\mathfrak M_{h,\epsilon}^\circ}\bigr).
\]
The extension supplies $\mathcal I_{h,\epsilon}u_h^n$ on elements
touching $\partial\cP_\epsilon$ at the next backward step. The equations above describe the full-grid implementation. 
For the convergence analysis, we restrict the scheme to the localized active node sets and  establish its local solvability there in
Theorem~\ref{conv}. 
Equation~\eqref{schemerho} is an implicit one-step discretization of the
density equation in the characteristic system~\eqref{mfg}, with
$\rho(t_n)=\mu^\ell$ and $\mu_\ell^+$ approximating $\rho(t_{n+1})$.
The arrival point $\mu_\ell^+$ is generally not a mesh node, and  thus the
next-layer field is evaluated there through the interpolated value
 $\mathcal I_{h,\epsilon}u_h^{n+1} (\mu_\ell^+)$. This characteristic update together with the off-grid evaluation constitutes the semi-Lagrangian step; see
Figure~\ref{schemefig} for $d=3$. 
Semi-Lagrangian methods have been widely used for MFGs on Euclidean domains; see, e.g. \cite{CarliniSilva2015,CamilliTangZhou2025,ChowdhuryErslandJakobsen2023}  and the references therein. Here the density characteristic evolves on the probability simplex, while the graph representing the individual state space remains fixed. 
\begin{figure}[!htbp]
\centering
\begin{tikzpicture}[x=1cm,y=1cm,>=Latex]
  \tikzset{
    mesh/.style={black!25,line width=0.45pt},
    nodept/.style={circle,fill=white,draw=black!55,line width=0.4pt,
      inner sep=0pt,minimum size=2.7pt}
  }

  \begin{scope}[xshift=-3.5cm]
    \coordinate (AoL) at (0,3.0);
    \coordinate (BoL) at (-2.1,0);
    \coordinate (CoL) at (2.1,0);
    \coordinate (AnL) at (barycentric cs:AoL=2,BoL=1,CoL=1);
    \coordinate (BnL) at (barycentric cs:AoL=1,BoL=2,CoL=1);
    \coordinate (CnL) at (barycentric cs:AoL=1,BoL=1,CoL=2);
    \node[font=\small] at (0,3.65) {time level $t_n$};
    \draw[black!65,line width=0.75pt]
      (AoL)--(BoL)--(CoL)--cycle;
    \foreach \j in {0,...,4}{
      \pgfmathsetmacro{\s}{\j/4}
      \draw[mesh] ($(AoL)!\s!(BoL)$)--($(AoL)!\s!(CoL)$);
      \draw[mesh] ($(BoL)!\s!(AoL)$)--($(BoL)!\s!(CoL)$);
      \draw[mesh] ($(CoL)!\s!(AoL)$)--($(CoL)!\s!(BoL)$);
    }
    \foreach \kone in {0,...,4}{
      \pgfmathtruncatemacro{\rem}{4-\kone}
      \foreach \ktwo in {0,...,\rem}{
        \pgfmathtruncatemacro{\kthree}{4-\kone-\ktwo}
        \node[nodept] at
          (barycentric cs:AoL=\kone,BoL=\ktwo,CoL=\kthree) {};
      }
    }
    % The departure density \mu^\ell is a node of the J=4 mesh.
    \coordinate (Start) at (AnL);
    \filldraw[fill=blue!65!black,draw=white,line width=0.7pt]
      (Start) circle (2.4pt);
    \node[anchor=south east,font=\scriptsize,text=blue!65!black,
      fill=white,inner sep=1pt]
      at ($(Start)+(-0.05,0.12)$) {$\mu^\ell$};
    \node[font=\scriptsize,text=black!55] at (1.55,2.45)
      {$\mathcal P(G)$};
  \end{scope}

  \begin{scope}[xshift=3.5cm]
    \coordinate (AoR) at (0,3.0);
    \coordinate (BoR) at (-2.1,0);
    \coordinate (CoR) at (2.1,0);
    \coordinate (AnR) at (barycentric cs:AoR=2,BoR=1,CoR=1);
    \coordinate (BnR) at (barycentric cs:AoR=1,BoR=2,CoR=1);
    \coordinate (CnR) at (barycentric cs:AoR=1,BoR=1,CoR=2);
    \node[font=\small] at (0,3.65) {time level $t_{n+1}$};
    \draw[black!65,line width=0.75pt]
      (AoR)--(BoR)--(CoR)--cycle;
    \foreach \j in {0,...,4}{
      \pgfmathsetmacro{\s}{\j/4}
      \draw[mesh] ($(AoR)!\s!(BoR)$)--($(AoR)!\s!(CoR)$);
      \draw[mesh] ($(BoR)!\s!(AoR)$)--($(BoR)!\s!(CoR)$);
      \draw[mesh] ($(CoR)!\s!(AoR)$)--($(CoR)!\s!(BoR)$);
    }
    \coordinate (Aone) at
      (barycentric cs:AnR=2,BnR=1,CnR=1);
    \coordinate (Atwo) at
      (barycentric cs:AnR=1,BnR=2,CnR=1);
    \coordinate (Athree) at
      (barycentric cs:AnR=1,BnR=1,CnR=2);
    \filldraw[fill=orange!25,draw=orange!70!black,line width=0.8pt]
      (Aone)--(Atwo)--(Athree)--cycle;
    \foreach \kone in {0,...,4}{
      \pgfmathtruncatemacro{\rem}{4-\kone}
      \foreach \ktwo in {0,...,\rem}{
        \pgfmathtruncatemacro{\kthree}{4-\kone-\ktwo}
        \node[nodept] at
          (barycentric cs:AoR=\kone,BoR=\ktwo,CoR=\kthree) {};
      }
    }
    \coordinate (Arrival) at
      (barycentric cs:Aone=2,Atwo=3,Athree=3);
    \filldraw[fill=red!70!black,draw=white,line width=0.7pt]
      (Arrival) circle (2.5pt);
    \node[above right=2pt,font=\scriptsize,text=red!70!black]
      at (Arrival) {$\mu_\ell^+$};
    \node[above=3pt,font=\scriptsize] at (Aone) {$a_1$};
    \node[below left=2pt,font=\scriptsize] at (Atwo) {$a_2$};
    \node[below right=2pt,font=\scriptsize] at (Athree) {$a_3$};
    \node[font=\scriptsize,text=black!55] at (1.55,2.45)
      {$\mathcal P(G)$};
  \end{scope}

  \draw[-{Latex[length=2mm]},black!70,line width=0.9pt]
    (Start)--(Arrival)
    node[midway,above=4pt,fill=white,inner sep=2pt,font=\scriptsize]
      {implicit characteristic step};
\end{tikzpicture}
\caption{Fully discrete semi-Lagrangian step for $d=3$.
Starting from an interior node
$\mu^\ell\in\mathfrak M_{h,\epsilon}^\circ$, the implicit update
\eqref{schemerho} determines an off-grid arrival point
$\mu_\ell^+\in\cP_\epsilon^\circ$. 
Let $\mathsf S_\ell^+\in\mathfrak T_{h,\epsilon}$ be the displayed triangle
with $\mu_\ell^+\in\mathsf S_\ell^+$, and let
$a_1,a_2,a_3\in\mathfrak M_{h,\epsilon}$ be its vertices.  Write
$\mu_\ell^+=\sum_{r=1}^3\lambda_r(\mu_\ell^+)a_r$, where
$\lambda_r(\mu_\ell^+)\ge0$ and
$\sum_{r=1}^3\lambda_r(\mu_\ell^+)=1$.  Then
$\mathcal I_{h,\epsilon}u_h^{n+1} (\mu_\ell^+)=\sum_{r=1}^3
\lambda_r(\mu_\ell^+)u_h^{n+1}(a_r)$, which is the value used in
\eqref{schemeu}.}
\label{schemefig}
\end{figure}
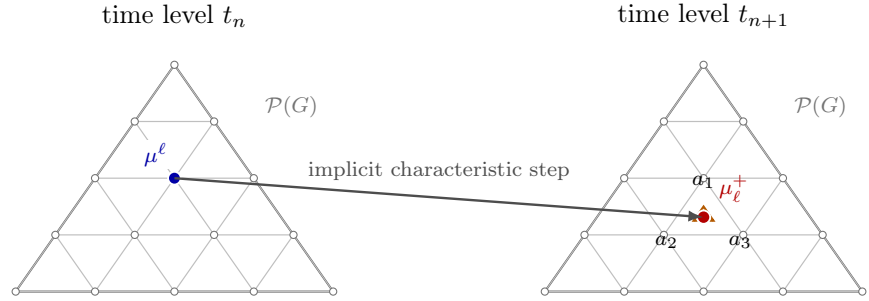

\begin{remark} 
The mixed-endpoint choice $(\mu_\ell^+,u_\ell^n)$ in \eqref{scheme} preserves Lasry--Lions monotonicity at the time-discrete level, provided that the next-layer field is monotone. 
Indeed, let $\mathcal V$ be a Lasry--Lions monotone next-layer field, and let $(\rho^{(a)},\phi^{(a)})$ be two solutions of the corresponding time-discrete one-step system with initial densities $\mu^{(a)}$, $a=1,2$. Graph integration by parts and the integrated form of \eqref{dll} give
$ 
 \ip{\phi^{(1)}-\phi^{(2)}}{\mu^{(1)}-\mu^{(2)}}
 \ge
 \ip{\mathcal V(\rho^{(1)})-\mathcal V(\rho^{(2)})}{\rho^{(1)}-\rho^{(2)}}
 \ge 0.
$ 
Related time-staggered
discretizations are used in numerical schemes for MFG systems
\cite{CuiDangGraphMFG2026,AchdouCapuzzoDolcetta2010,
AchdouCamilliCapuzzoDolcetta2013}.  
This time-discrete structure does not, however, guarantee Lasry--Lions
monotonicity of the fully discrete master field, since componentwise
piecewise affine reconstruction need not preserve this property.
Constructing a monotonicity-preserving fully
discrete scheme for the master equation on the probability simplex is
left for future work. 
\end{remark}

\subsection{Main results} 
This subsection states the main results for the numerical master field, the reconstructed MFG characteristics, and the feedback flux in the potential case. 
For $0\le t\le s\le T$ and $\mu\in\cP^\circ(G)$, define the
density flow associated with \eqref{representation} by 
\begin{equation}\label{flowmap}
\Phi_{t,s}(\mu):=\rho(s;t,\mu). 
\end{equation} 
The following theorem gives the convergence estimate for the semi-Lagrangian approximation.

\begin{theorem}\label{conv}
Let Assumptions~\ref{structure} and \ref{mono} hold, and let $u$ be the
classical solution  of \eqref{eq_intro} given by Proposition~\ref{regularity}. For every compact set $K\Subset\mathcal P^\circ(G)$, there exist 
constants $\epsilon_K,c_K,\tau_K>0$ such that the following hold.

Fix $0\le\epsilon<\epsilon_K$, and 
choose $N=N(h)$ and $\tau=T/N(h)\leq\tau_K $.  
If
\begin{equation}\label{coupling}
h\le c_K\tau,\quad 
 \frac{h\omega_K(h)}{\tau}\to0\;\;(\text{as }h\to0),
 \quad
 \frac{\tau^2}{h}\to0\;\;(\text{as }h\to0),
\end{equation}
then for all sufficiently small $h$, there exist unions of mesh
elements $\Omega_{n,h}$ such that 
$ 
K\subset\Omega_{n,h}\subset\mathcal P_{\epsilon_K}
\subset\mathcal P_\epsilon,\;0\le n\le N,
$ 
and for every $0\le n<N$ and 
$\mu^\ell\in\mathfrak M_{h,\epsilon}\cap\Omega_{n,h}$,
 there exists a unique pair $(\mu_\ell^+,u_\ell^n)$ solving \eqref{scheme} and  satisfying 
\[
 \norm{\mu_\ell^+-\Phi_{t_n,t_{n+1}}(\mu^\ell)}\le\frac{c_K\tau}{2},
 \qquad
 \norm{u_\ell^n-u(t_n,\mu^\ell)}\le1.
\]
Moreover, its piecewise affine reconstruction satisfies
\begin{equation}\label{error}
 \max_{0\le n\le N}
 \norm{\mathcal I_{h,\epsilon}u_h^n-u(t_n,\cdot)}_{L^\infty(K)}
 \le C_K\bigl(\tau+\frac{h\omega_K(h)}{\tau}\bigr), 
\end{equation} where $
 \omega_K(r):=
 \sup_{\substack{0\le t\le T,\ \mu,\nu\in\mathcal P_{\epsilon_K}\\
                  \norm{\mu-\nu}\le r}}
 \norm{D_\mu u(t,\mu)-D_\mu u(t,\nu)}_{
 \mathcal L(\bbR_0^d,\bbR^d)}. 
$ 
In particular, if $u\in\mathcal C^2$, then $\omega_K(h)\le Ch$ and under \eqref{coupling},  the method is first-order accurate in $\tau$. 
\end{theorem}
\begin{remark}\label{rate}
(i) The conditions in \eqref{coupling} can be attained under the
$\mathcal C^1$-regularity of $u$. Indeed, set
$
 N(h):=\lceil
 \frac{T}{h^{1/2}(h+\omega_K(h))^{1/4}}
 \rceil,\;
 \tau:=\frac{T}{N(h)},$ where $\lceil\cdot\rceil$ is the ceiling function. 
For all sufficiently small $h$, we have 
$ 
 \frac12 h^{1/2}(h+\omega_K(h))^{1/4}
 \le \tau\le
 h^{1/2}(h+\omega_K(h))^{1/4}.$ 
Consequently,
$ 
 \frac{h}{\tau}
 \le 2\frac{h^{1/2}}{(h+\omega_K(h))^{1/4}}
 \le 2h^{1/4}\to0,
$ 
so $h\le c_K\tau$ for all sufficiently small $h$. Moreover,
$ 
 \frac{\tau^2}{h}
 \le(h+\omega_K(h))^{1/2}\to0$ and $
\frac{h\omega_K(h)}{\tau}
 \le 2h^{1/2}(h+\omega_K(h))^{3/4}\to0$ as $h\to0.$

(ii) Under the assumptions of Theorem~\ref{conv}, suppose that
$D_\mu u$ is $\beta$-H\"older continuous in $\mu$ on
$[0,T]\times \mathcal P_{\epsilon_K}$
 for some $0<\beta\le1$.  Choosing
$\tau\asymp h^{(1+\beta)/2}$, with the proportionality constant
chosen so that $h\le c_K\tau$ when $\beta=1$, gives 
\[
 \max_{0\le n\le N}
 \norm{\mathcal I_{h,\epsilon}u_h^n-u(t_n,\cdot)}_{L^\infty(K)}
 \le C_Kh^{(1+\beta)/2}.
\]  
Under the additional $\mathcal C^2$ hypotheses of
Proposition~\ref{regularity}, $D_\mu u$ is locally Lipschitz in $\mu$,
and the case $\beta=1$ yields first-order accuracy. 
\end{remark}
Once the master field is computed,  it can be used as feedback to generate
MFG characteristics from prescribed initial population states. The
following corollary shows convergence of the corresponding
numerical density and value trajectories. 
\begin{corollary}\label{feedback}
Let the conditions of Theorem~\ref{conv} hold and let
$u_h=(u_h^n)_{n=0}^N$ be the solution given
in Theorem~\ref{conv}. For $\mu_0\in K$, define the numerical feedback trajectory
$(\widehat\rho_h^n,\widehat\phi_h^n)_{n=0}^N$ recursively.
Set $\widehat\rho_h^0=\mu_0$. Given $\widehat\rho_h^n$, set $
\widehat\phi_h^n
:=
(\mathcal I_{h,\epsilon}u_h^n)(\widehat\rho_h^n),$ 
and for $0\le n<N$, determine $\widehat\rho_h^{n+1}$ from 
\begin{equation}\label{feedbackstep}
 \widehat\rho_h^{n+1}
 =\widehat\rho_h^n+
 \tau\mathcal A\bigl(\widehat\rho_h^{n+1},
 \widehat\phi_h^n\bigr).
\end{equation} 
For all sufficiently small $h$, \eqref{feedbackstep} has a
unique solution
$\widehat\rho_h^{n+1}\in\Omega_{n+1,h}\Subset\cP^\circ(G)$ at every step. 
Moreover,
\begin{equation}\label{feedbackerror}
 \sup_{\mu_0\in K}\max_{0\le n\le N}
 \Bigl[
 \norm{\widehat\rho_h^n-\rho(t_n;0,\mu_0)}
 +\norm{\widehat\phi_h^n-\phi(t_n;0,\mu_0)}
 \Bigr]
 \le C_K\bigl(\tau+\frac{h\omega_K(h)}{\tau}\bigr).
\end{equation}
\end{corollary}

In the potential setting, the exact master feedback generates a
minimizing density--flux pair for the associated planning problem.
We next show that the numerical feedback is admissible and quantify
its deviation from optimality through an action estimate. For this
purpose, we impose the following potential structure assumptions, which are
taken from \cite[Section~4.1]{GangboMunozWuZhang2026}.

\begin{assumption}\label{potential}
Let
$\cL:[0,\infty)^d\times\bbS^{d\times d}\to\bbR\cup\{+\infty\}$ be jointly
convex and lower semicontinuous, with
$\cL\in\mathcal C^\infty((0,\infty)^d\times\bbS^{d\times d})$,
$\cL(\mu,0)=0$, and
$D_{mm}^2\cL(\mu,m)[r,r]>0$ for every 
$(\mu,m)\in(0,1)^d\times\bbS^{d\times d}$ and  every $r\in\bbS^{d\times d}\setminus\{0\}$.   For some $p_0>1$ and
$c_{\cL}>0$,
\[
 \cL(\mu,m)\ge c_{\cL}\bigl(\norm{m}_E^{p_0}-1\bigr),
 \qquad \forall\,(\mu,m)\in\cP(G)\times\bbS^{d\times d}.
\]
The functions
$\ip{m}{D_m\cL(\mu,m)}_E
+\ip{\mu}{\nabla_\mu\cL(\mu,m)}$ and
$\min\limits_{i=1,\ldots,d}\bigl[-D_{\mu_i}\cL(\mu,m)\bigr]$ are bounded from below
on $(0,\infty)^d\times\bbS^{d\times d}$.
Let
$
\cH(\mu,p):=\sup_{m\in\bbS^{d\times d}}
 \bigl\{\ip{m}{p}_E-\cL(\mu,m)\bigr\}.
$ 
Assume that
$\cH\in\mathcal C^\infty((0,1)^d\times\bbS^{d\times d})$.
Suppose that
$\cF\in\mathcal C^2([0,1]^d)\cap\mathcal C^\infty((0,1)^d)$ satisfies $
D^2\cF(\mu)[\eta,\eta]<0
\text{ for every }\mu\in(0,1)^d,
\eta\in\bbR_0^d\setminus\{0\},$ and that
$\cU_T\in\mathcal C^2([0,1]^d)\cap\mathcal C^\infty((0,1)^d)$ is convex on
$\cP(G)$.  The coefficients of \eqref{eq_intro} read 
\begin{equation}\label{potentialcoeff}
\bm H(\mu,p)=\nabla_\mu\cH(\mu,-p)+\nabla_\mu\cF(\mu),\quad
 \bm B(\mu,p)=-D_p\cH(\mu,-p),\quad
g(\mu)=\nabla_\mu\cU_T(\mu).
\end{equation}
Here $\nabla_\mu$ denotes the Euclidean gradient of the scalar
functions.  
\end{assumption}

We next recall the associated mean field planning problem. For $\mu\in\cP^\circ(G)$, an admissible pair consists of
$\rho\in\mathcal C([t,T];\cP(G))$ and
$m\in L^{p_0}((t,T);\bbS^{d\times d})$ satisfying
$\rho(t)=\mu$ and $\dot\rho+\divG m=0$ in the sense of distributions, where $p_0>1$ is the coercivity exponent in
Assumption~\ref{potential}. Define
\begin{equation}\label{action}
 \begin{split}
 \mathcal J_t(\rho,m)
 &:=
 \int_t^T\bigl[
 \cL\bigl(\rho(s),m(s)+\gradG\rho(s)\bigr)-\cF(\rho(s))
 \bigr]\,\mathrm ds+\cU_T(\rho(T)),\\
 \cU(t,\mu)&:=\inf\bigl\{\mathcal J_t(\rho,m):(\rho,m)
 \text{ is admissible with }\rho(t)=\mu\bigr\}.
 \end{split}
\end{equation}
By \cite[Theorem~1.3]{GangboMunozWuZhang2026}, we have 
$ 
 \nabla_{\mathcal W}\cU(t,\mu)=\gradG u(t,\mu).
$ 
Moreover, the pair $(\rho,m^*)$ is the unique minimizer of
\eqref{action}, where $\rho$ is the characteristic generated by $u$ and
\begin{align}\label{mstar}
 m^*(s)=D_p\cH\bigl(\rho(s),-\gradG u(s,\rho(s))\bigr)-\gradG\rho(s).
\end{align}

\begin{remark}
The logarithmic-mean model of Example~\ref{exam} is covered by
Assumption~\ref{potential}. Let
\[
 \cL_\theta(\mu,m)
 =\frac14\sum_{(i,j)\in E}
   \frac{m_{ij}^2}{\theta(\mu_i,\mu_j)},
 \qquad
 \cF(\mu)=-\frac{\alpha}{2}\norm{\mu}^2,
 \qquad
 \cU_T(\mu)=\norm{\mu-\mu_{\rm tar}}^2,
\]
where $\mu_{\rm tar}\in\cP^\circ(G)$ is fixed. The displayed formula
for $\cL_\theta$ is defined on the open simplex, and on the boundary, we
use its lower-semicontinuous extension. 
For $\mu\in\cP^\circ(G)$, the Legendre transform of
$\cL_\theta(\mu,\cdot)$ is $\cH_\theta(\mu,\cdot)$ defined in
\eqref{model}. Moreover,
$
\nabla_\mu\cF(\mu)=-\alpha\mu,\;
\nabla_\mu\cU_T(\mu)=2(\mu-\mu_{\rm tar}),$ 
and
$
-D_p\cH_\theta(\mu,-p)=\theta(\mu)p.
$ 
Therefore, \eqref{potentialcoeff} gives
\[
 \bm H(\mu,p)
 =\nabla_\mu\cH_\theta(\mu,-p)-\alpha\mu,
 \qquad
 \bm B(\mu,p)=\theta(\mu)p,
 \qquad
 g(\mu)=2(\mu-\mu_{\rm tar}).
\]
Thus these choices reproduce the Hamiltonian and transport
coefficients of Example~\ref{exam} and specify the corresponding
terminal map. 
\end{remark}

The following theorem gives a variational consequence of the feedback
convergence. The reconstructed density--flux pair is admissible for the
planning problem, and its action gap is controlled by the square of the
feedback error. Thus the numerical feedback is asymptotically optimal
at the level of the planning cost.

\begin{theorem}\label{actiontheorem}
Let the conditions of Corollary~\ref{feedback} and
Assumption~\ref{potential} hold.  For $\mu_0\in K$, let
$(\widehat\rho_h^n,\widehat\phi_h^n)_{n=0}^N$ be the numerical feedback
trajectory defined in Corollary~\ref{feedback}.  For $0\le n<N$, define
\begin{equation}\label{flux}
 m_h^{n+1}:=
 D_p\cH\bigl(\widehat\rho_h^{n+1},
 -\gradG\widehat\phi_h^n\bigr)
 -\gradG\widehat\rho_h^{n+1}.
\end{equation}
For $t\in[t_n,t_{n+1}]$, set
\begin{equation}\label{reconstruction}
 \overline\rho_h(t)
 :=\frac{t_{n+1}-t}{\tau}\widehat\rho_h^n
 +\frac{t-t_n}{\tau}\widehat\rho_h^{n+1},
 \qquad
 \overline m_h(t):=m_h^{n+1}.
\end{equation}
Then $(\overline\rho_h,\overline m_h)$ is admissible with
$\overline\rho_h(0)=\mu_0$, and
\begin{equation}\label{actionerror}
0\le
 \mathcal J_0(\overline\rho_h,\overline m_h)-\cU(0,\mu_0)
 \le C_K\bigl(\tau+\frac{h\omega_K(h)}{\tau}\bigr)^2,
 \qquad \mu_0\in K.
\end{equation}
Thus the control flux generated by the numerical master feedback is
asymptotically optimal, uniformly for $\mu_0\in K$.
\end{theorem} 
When $D_\mu u$ is Lipschitz in $\mu$, one has
$\omega_K(h)\le Ch$. Hence, choosing $\tau\asymp h$ as in
Remark~\ref{rate}(ii), Theorem~\ref{actiontheorem} gives, under
Assumption~\ref{potential},
\[
0\le
\mathcal J_0(\overline\rho_h,\overline m_h)-\cU(0,\mu_0)
\le C_Kh^2,
\qquad \mu_0\in K.
\]

\section{A priori estimates and local solvability}\label{properties}

In this section, we first establish the interpolation and graph-resolvent estimates. Then we 
construct characteristic tubes and prove the local solvability of the fully discrete solution at each update.

\subsection{Interpolation and graph resolvent}

The interpolation bounds on the fixed mesh are collected in the
following lemma. For each mesh element
$S\in\mathfrak T_{h,\epsilon}$, derivatives of piecewise affine
interpolants are understood elementwise, and the $L^\infty(S)$ norms are
essential suprema over the relative interior of $S$. 

\begin{lemma}
\label{interpolation}
Fix $\epsilon\in[0,1/d)$.
There is a constant $C_I>0$, independent of $h$,  such that
\begin{align}
 \norm{\mathcal I_{h,\epsilon}z_h}_{L^\infty(\cP_\epsilon)}
 &\le\max_{\mu^\ell\in\mathfrak M_{h,\epsilon}}\norm{z_h(\mu^\ell)},
 \label{istab}\\
 \norm{v-\mathcal I_{h,\epsilon}P_{h,\epsilon}v}_{L^\infty(\cP_\epsilon)}
 &\le C_Ih\omega_{D_\mu v}(h),
 \label{ierror}\\
 \max_{S\in\mathfrak T_{h,\epsilon}}
 \norm{D_\mu(\mathcal I_{h,\epsilon}P_{h,\epsilon}v)}_{L^\infty(S)}
 &\le C_I\norm{D_\mu v}_{L^\infty(\cP_\epsilon)},
 \label{ilip}\\
 \max_{S\in\mathfrak T_{h,\epsilon}}
 \norm{D_\mu(\mathcal I_{h,\epsilon}P_{h,\epsilon}v)-D_\mu v}_{L^\infty(S)}
 &\le C_I\omega_{D_\mu v}(h),
 \label{igrad}\\
 \max_{S\in\mathfrak T_{h,\epsilon}}
 \norm{D_\mu(\mathcal I_{h,\epsilon}z_h)}_{L^\infty(S)}
 &\le C_Ih^{-1}\max_{\mu^\ell\in\mathfrak M_{h,\epsilon}}\norm{z_h(\mu^\ell)}.
 \label{inverse}
\end{align}
Here, $z_h:\mathfrak M_{h,\epsilon}\to\bbR^d$, $v:\cP_\epsilon\to \bbR^d$ is a 
$\mathcal C^1$ map, and
$
 \omega_{D_\mu v}(r):=
\sup_{\substack{x,y\in\cP_\epsilon\\
                  \norm{x-y}\le r}}
 \norm{D_\mu v(x)-D_\mu v(y)}_{\mathcal L(\bbR_0^d,\bbR^d)}.
$ 
\end{lemma}
\begin{proof}
\textit{Proof of \eqref{istab}.}
Fix $x\in\cP_\epsilon$ and let $S\in\mathfrak T_{h,\epsilon}$ contain $x$.
Denote its vertices, which belong to $\mathfrak M_{h,\epsilon}$, by
$a_1,\ldots,a_d$, and their barycentric coordinates by
$\lambda_1,\ldots,\lambda_d$.  Then
\begin{equation}\label{bary}
 (\mathcal I_{h,\epsilon}z_h)(x)=\sum_{j=1}^d\lambda_j(x)z_h(a_j),
 \qquad \lambda_j(x)\ge0,
 \qquad \sum_{j=1}^d\lambda_j(x)=1.
\end{equation}
Using the triangle inequality in \eqref{bary} proves
\eqref{istab}.

\textit{Proof of \eqref{ierror}.}
By the fundamental theorem of calculus, we have
\begin{align*}
 v(a_j)-v(x)-D_\mu v(x)(a_j-x)
 =\int_0^1
 [D_\mu v(x+s(a_j-x))-D_\mu v(x)](a_j-x)\,\mathrm ds.
\end{align*}
Multiplying this identity by $\lambda_j(x)$ and summing over $j$ leads to
\begin{align*}
 (\mathcal I_{h,\epsilon}P_{h,\epsilon}v)(x)-v(x)
 &=\sum_{j=1}^d\lambda_j(x)\bigl(v(a_j)-v(x)\bigr)=D_\mu v(x)\sum_{j=1}^d\lambda_j(x)(a_j-x)\\
 &\quad+\sum_{j=1}^d\lambda_j(x)\int_0^1
 \bigl[D_\mu v(x+s(a_j-x))-D_\mu v(x)\bigr]
 (a_j-x)\,\mathrm ds.
\end{align*}
The first term on the right-hand side vanishes because
$\sum_j\lambda_j(x)a_j=x$ and $\sum_j\lambda_j(x)=1$.
The segment joining $x$ to $a_j$ lies in
$S\subset\cP_\epsilon$ and thus
$\norm{a_j-x}\le\operatorname{diam}(S)\le h$, which leads to 
$ 
 \norm{(\mathcal I_{h,\epsilon}P_{h,\epsilon}v)(x)-v(x)}
 \le\sum_{j=1}^d\lambda_j(x)h\omega_{D_\mu v}(h)
 =h\omega_{D_\mu v}(h).
$ 
Taking the supremum over $x\in\cP_\epsilon$ proves
\eqref{ierror}.

\textit{Proof of \eqref{ilip}.}
Choose $a_1$ as a reference vertex. 
Since
$\lambda_1=1-\sum_{j=2}^d\lambda_j$, the interpolant can be written as
$ 
 (\mathcal I_{h,\epsilon}z_h)(x)
 =z_h(a_1)+\sum_{j=2}^d\lambda_j(x)
 \bigl(z_h(a_j)-z_h(a_1)\bigr).
$ 
Hence, for every tangent direction $\xi\in\bbR_0^d$,
\begin{equation}\label{derivative}
 D_\mu(\mathcal I_{h,\epsilon}z_h)|_S[\xi]
 =\sum_{j=2}^d\bigl(z_h(a_j)-z_h(a_1)\bigr)
 D_\mu\lambda_j[\xi], 
\end{equation}
where $D_\mu\lambda_j\in\mathcal L(\bbR_0^d,\bbR)$.  
Let $F_j$ be the face of $S$ opposite to $a_j$, and set
$H_j:=\operatorname{dist}(a_j,\operatorname{aff}(F_j))$,
where $\operatorname{aff}(F_j)$ denotes the affine hull of $F_j.$ 
Since $\lambda_j(a_j)=1$ and
$\lambda_j=0$ on $F_j$,
$ 
\norm{D_\mu\lambda_j}_{\mathcal L(\bbR_0^d,\bbR)}=H_j^{-1}.
$ 
By the shape regularity and quasi-uniformity of the triangulations
$\mathfrak T_{h,\epsilon}$ constructed in
Section~\ref{sec4-1}, 
there is $c>0$ independent of $h$, $S$, and $j$ such that 
$ 
 H_j\ge c\operatorname{diam}(S),\;
\operatorname{diam}(S)\ge ch$. This implies 
$\norm{D_\mu\lambda_j}_{\mathcal L(\bbR_0^d,\bbR)}
 \le \frac{1}{c\operatorname{diam}(S)}\le C_Ih^{-1}.
$
Inserting 
\begin{align*} 
 \norm{v(a_j)-v(a_1)}
 \le\int_0^1
 \norm{D_\mu v\bigl(a_1+s(a_j-a_1)\bigr)}
 \norm{a_j-a_1}\,\mathrm ds
 \le h\norm{D_\mu v}_{L^\infty(S)},\;\;j=2,\ldots,d,
\end{align*} and $z_h=P_{h,\epsilon}v$ 
into \eqref{derivative} implies 
\begin{align*}
\norm{D_\mu(\mathcal I_{h,\epsilon}P_{h,\epsilon}v)|_S}
 &=\sup_{\norm{\xi}\le1}
 \norm{\sum_{j=2}^d
 \bigl(v(a_j)-v(a_1)\bigr)D_\mu\lambda_j[\xi]}\\&\le \sum_{j=2}^d
 \norm{v(a_j)-v(a_1)}
 \norm{D_\mu\lambda_j}\le C_I\norm{D_\mu v}_{L^\infty(S)}.
\end{align*} 
Taking the maximum of the preceding elementwise estimate over
$S\in\mathfrak T_{h,\epsilon}$ proves \eqref{ilip}.

\textit{Proof of \eqref{igrad}.}
For $x\in S$,  since
$ 
 x=a_1+\sum_{j=2}^d\lambda_j(x)(a_j-a_1),
$ 
differentiation in a direction $\xi\in\bbR_0^d$ gives 
$ 
 \sum_{j=2}^d(a_j-a_1)D_\mu\lambda_j[\xi]=\xi.
$ 
Applying $D_\mu v(x)$ to this identity results in
$ 
 D_\mu v(x)[\xi]
 =\sum_{j=2}^dD_\mu v(x)[a_j-a_1]D_\mu\lambda_j[\xi].
$ 
Subtracting this equality from \eqref{derivative} with $z_h=P_{h,\epsilon}v$ gives
\begin{align*}
 \bigl(D_\mu(\mathcal I_{h,\epsilon}P_{h,\epsilon}v)|_S-D_\mu v(x)\bigr)[\xi]=\sum_{j=2}^d
 \bigl[v(a_j)-v(a_1)-D_\mu v(x)[a_j-a_1]\bigr]
 D_\mu\lambda_j[\xi].
\end{align*}
\iffalse
By taking the supremum over $\norm{\xi}\le1$, we obtain 
\begin{align*}
\norm{D_\mu(\mathcal I_{h,\epsilon}P_{h,\epsilon}v)|_S-D_\mu v(x)}
 &\le\sum_{j=2}^d
 \norm{v(a_j)-v(a_1)-D_\mu v(x)[a_j-a_1]}
 \norm{D_\mu\lambda_j}. 
 \end{align*} \fi
By the fundamental theorem of calculus, 
we have 
$ 
 \norm{v(a_j)-v(a_1)-D_\mu v(x)[a_j-a_1]}
 \le h\,\omega_{D_\mu v}(h).
$  Taking the supremum over $\norm{\xi}\le1$ and using $\norm{D_\mu\lambda_j}\le C_Ih^{-1}$ in the previous display therefore yields
$ 
\norm{D_\mu(\mathcal I_{h,\epsilon}P_{h,\epsilon}v)|_S-D_\mu v(x)}
 \le C_I\omega_{D_\mu v}(h),
$ 
which proves \eqref{igrad}. 

\textit{Proof of \eqref{inverse}.}
For arbitrary $z_h$, put
$M_h:=\max_{\mu^\ell\in\mathfrak M_{h,\epsilon}}\norm{z_h(\mu^\ell)}$.  Then by \eqref{derivative},
\begin{align*}
 \norm{D_\mu(\mathcal I_{h,\epsilon}z_h)|_S}
 \le\sum_{j=2}^d
 \norm{z_h(a_j)-z_h(a_1)}
 \norm{D_\mu\lambda_j}\le 2M_h\sum_{j=2}^d\norm{D_\mu\lambda_j}
 \le C_Ih^{-1}M_h.
\end{align*}
Taking the maximum over $S\in\mathfrak T_{h,\epsilon}$ proves
\eqref{inverse}.
\end{proof}

The following lemma shows elementary resolvent properties of the graph
Laplacian.
\begin{lemma}
\label{resolvent}
For every $\tau>0$, $S_\tau=(I-\tau\DeltaG)^{-1}$ exists and satisfies $S_\tau\bm1=\bm1$, 
\begin{equation}\label{resolventid}
 \norm{S_\tau z}_{\ell^\infty}\le\norm{z}_{\ell^\infty},
 \qquad \norm{S_\tau z}\le\norm{z},
 \qquad \ip{S_\tau z}{\bm1}=\ip{z}{\bm1},\quad z\in\mathbb R^d.
\end{equation}
\end{lemma}
\begin{proof}
By \eqref{lap}, $-\DeltaG$ is symmetric positive semidefinite.
Thus every eigenvalue of $I-\tau\DeltaG$ is at least one, which proves
invertibility and the $\ell^2$ estimate in
\eqref{resolventid}.  If $y=S_\tau z$, then $(I-\tau\DeltaG)y=z$.  Taking the inner product with $\bm1$ and using the symmetry of $\DeltaG$ together with $\DeltaG\bm1=0$ gives
$\ip{y}{\bm1}=\ip{z}{\bm1}$, proving the last identity in \eqref{resolventid}. 
Let $y=S_\tau z$.  If $y_k=\max_i y_i$, then
$(\DeltaG y)_k=\sum_{j\sim k}\omega_{kj}(y_j-y_k)\le0$, and hence
$y_k\le y_k-\tau(\DeltaG y)_k=z_k\le\max_i z_i$.
Since $S_\tau$ is linear, $-y=S_\tau(-z)$. Applying the preceding maximum
estimate to $-y$ gives
\[
 -\min_i y_i=\max_i(-y_i)\le\max_i(-z_i)=-\min_i z_i,
\]
and therefore $\min_i y_i\ge\min_i z_i$. Consequently,
$ 
 \max_i y_i\le\max_i z_i\le\norm{z}_{\ell^\infty}$ and $
 -\min_i y_i\le-\min_i z_i\le\norm{z}_{\ell^\infty},$ 
which proves the $\ell^\infty$ estimate in \eqref{resolventid}. Finally,
$\DeltaG\bm1=0$, so
$S_\tau\bm1=\bm1$.
\end{proof}

\subsection{Characteristic localization and local solvability}

The following proposition provides the geometric margin needed by the nonlinear update. Starting from a compact family of exact characteristics, we construct tubes whose radii grow faster than the separation of nearby trajectories. Taking whole mesh elements around smaller tubes gives the active submeshes.

\begin{proposition}\label{tube}
Let Assumptions~\ref{structure} and \ref{mono} hold, and let $K\Subset\cP^\circ(G)$ be compact.  There are a compact set
$\mathcal Q_K\Subset\cP^\circ(G)$ and constants
$\epsilon_K,c_K,\tau_K>0$ such that, if
$0\le\epsilon<\epsilon_K$, $0<\tau\le\tau_K$, and $h\le c_K\tau$,
there exists a family $\bigl(\Omega_{n,h}\bigr)_{n=0}^N$ of sets, each
formed as a union of elements of the triangulation
$\mathfrak T_{h,\epsilon}$, such that
\begin{equation}\label{separation}
 K\subset\Omega_{n,h}\subset\mathcal Q_K,\qquad
 \operatorname{dist}\big(
 \Phi_{t_n,t_{n+1}}(\Omega_{n,h}),
 \bbR_1^d\setminus\Omega_{n+1,h}\big)
 \ge c_K\tau
\end{equation}
for $0\le n<N$. 
\end{proposition}

\begin{proof} The proof has four steps. In Steps 1--2, we construct uniform interior tubes
around the exact characteristics starting from the compact set $K$ and establish their
stability under the exact flow. In Steps 3--4, we cover smaller tubes by mesh
elements and combine
the flow estimate to obtain the separation
\eqref{separation}. 

\textit{Step 1. Construction of an interior tube around the exact flow.} We first collect all exact density characteristics starting from \(K\) and show that they remain in a common compact subset of the open simplex.  
For $0\le t\le T$, set
$ \Gamma(t):=\{\Phi_{r,t}(\mu):0\le r\le t,\ \mu\in K\},
$ where we recall that $\Phi_{r,t}(\mu)$ is the exact density flow
map defined in \eqref{flowmap}. 
The parameter set
$ \{(r,t,\mu):0\le r\le t\le T,\ \mu\in K\}$ is compact.
Proposition~\ref{regularity} ensures that, for every
$(r,\mu)\in[0,T]\times\cP^\circ(G)$, the density characteristic exists
uniquely on $[r,T]$ and remains in $\cP^\circ(G)$.  Moreover, $\mathcal A(\mu,u(t,\mu))$ is continuous in $t$ and locally
Lipschitz in $\mu$ on $\cP^\circ(G)$.  The continuous-dependence theorem
for ordinary differential equations therefore shows that
$(r,t,\mu)\mapsto\Phi_{r,t}(\mu)$ is continuous on the parameter set above. Hence the image of this compact parameter set is compact and contained in \(\cP^\circ(G)\). In particular,  
\begin{equation}
 \Gamma_*:=\bigcup_{0\le t\le T}\Gamma(t)
 \Subset\cP^\circ(G),
 \qquad
 d_K:=\operatorname{dist}_{\bbR_1^d}
 \bigl(\Gamma_*,\partial\cP(G)\bigr)>0.
\end{equation}
All distances between subsets of the density affine space below are taken with respect to the Euclidean metric induced on $\bbR_1^d$.
Moreover, $K\subset\Gamma(t)$ because $\Phi_{t,t}$ is the identity, and
the flow property implies 
$\Phi_{t,r}(\Gamma(t))\subset\Gamma(r)$ for $0\le t\le r\le T$.

Choose $0<\bar r<d_K/2$ and set
$ 
\mathcal Q_K
 :=\bigl\{\mu\in\bbR_1^d:
 \operatorname{dist}(\mu,\Gamma_*)\le2\bar r\bigr\}
 \Subset\cP^\circ(G).$ 
 This neighborhood provides a common region on which the exact vector field
is uniformly Lipschitz. Hence 
on $[0,T]\times \mathcal Q_K$, there is
$L>0$ such that
$\norm{\mathcal A(\mu,u(t,\mu))-\mathcal A(\nu,u(t,\nu))}\le L\norm{\mu-\nu}$.
Fix $\gamma>0$, choose $r_0>0$ with
$r_0e^{(L+\gamma)T}<\bar r$, and define
\[
 r(t):=r_0e^{(L+\gamma)t},\qquad
 D(t):=\{\mu\in\bbR_1^d:
 \operatorname{dist}(\mu,\Gamma(t))\le r(t)\}.
\]
Since $r(t)\le r_0e^{(L+\gamma)T}<\bar r$, we obtain 
\[
 \operatorname{dist}\bigl(D(t),\partial\cP(G)\bigr)
 \ge \operatorname{dist}\bigl(\Gamma(t),\partial\cP(G)\bigr)-r(t)
 >\bar r,
 \qquad 0\le t\le T.
\] 
Hence $D(t)\Subset\cP^\circ(G)$ uniformly in $t$.

\textit{Step 2. Stability of the exact flow inside the tube.} We show that the exact flow maps $D(t)$ strictly inside the later tube, namely, for $0\le t\le q\le T$ and $\mu\in D(t)$,
\begin{equation}\label{flowdistance}
\operatorname{dist} \bigl(\Phi_{t,q}(\mu),\Gamma(q)\bigr)
\le e^{-\gamma(q-t)}r(q).
\end{equation}   
By the compactness of
$\Gamma(t)$, there exists $\nu\in\Gamma(t)$ such that
$ 
 \norm{\mu-\nu}=\operatorname{dist}\bigl(\mu,\Gamma(t)\bigr).
$ Since the Lipschitz constant $L$ is available on
$\mathcal Q_K$, we first verify that the two trajectories remain in this
region. Suppose otherwise that their distance reaches $\bar r$, and let
$s_*\in[t,q]$ be the first time such that
$ 
\|\Phi_{t,s_*}(\mu)-\Phi_{t,s_*}(\nu)\|=\bar r.
$ For $t\le s\le s_*$, we have
$\Phi_{t,s}(\nu)\in\Gamma(s)\subset\Gamma_*$ and
$ 
\|\Phi_{t,s}(\mu)-\Phi_{t,s}(\nu)\|\le\bar r.
$ 
Hence both trajectories lie in $\mathcal Q_K$, and Gronwall's inequality
gives  
\begin{align*}
 \bar r
 &=\norm{\Phi_{t,s_*}(\mu)-\Phi_{t,s_*}(\nu)}\le e^{L(s_*-t)}\norm{\mu-\nu}
 =e^{L(s_*-t)}\operatorname{dist}\bigl(\mu,\Gamma(t)\bigr)\\
 &\le e^{L(s_*-t)}r(t)
 =e^{-\gamma(s_*-t)}r(s_*)<\bar r,
\end{align*}
a contradiction.   Thus the two trajectories remain within distance
$\bar r$ on $[t,q]$, so the same Gronwall estimate applies on the whole
interval.  
Since $\Phi_{t,q}(\nu)\in\Gamma(q)$, 
\begin{align*} \operatorname{dist}\bigl(\Phi_{t,q}(\mu),\Gamma(q)\bigr)
 &\le\norm{\Phi_{t,q}(\mu)-\Phi_{t,q}(\nu)}
 \le e^{L(q-t)}\operatorname{dist}\bigl(\mu,\Gamma(t)\bigr)\le e^{L(q-t)}r(t)
 =e^{-\gamma(q-t)}r(q). 
\end{align*}
which proves \eqref{flowdistance}.

\textit{Step 3. Construction of the mesh-dependent active regions.} We now construct the active regions containing elements of meshes from the continuous tubes obtained
above. Since $r(t)<\bar r$, we have $D(t)\subset\mathcal Q_K$. Set
$
\epsilon_K:=\frac12\min_{\mu\in\mathcal Q_K}\min_{1\le i\le d}\mu_i>0.
$ 
Then $\mathcal Q_K\Subset\cP_\epsilon^\circ$ whenever
$0\le\epsilon<\epsilon_K$, so the triangulation $\mathfrak T_{h,\epsilon}$ covers all the regions
constructed below.  
For $0\le n\le N$, define the smaller tube 
\[
 D_n^-:=\{\mu\in\bbR_1^d:
 \operatorname{dist}(\mu,\Gamma(t_n))
 \le e^{-\gamma\tau/2}r(t_n)\},
\]
and let $\Omega_{n,h}$ be the union of all elements
$S\in\mathfrak T_{h,\epsilon}$ satisfying $S\cap D_n^-\ne\varnothing$. The gap between $D_n^-$ and $D(t_n)$ is of order $\tau$. Indeed, for
$\gamma\tau\le1$, 
\begin{align} \label{tubegap1}
 r(t_n)\bigl(1-e^{-\gamma\tau/2}\bigr)\ge c_1\tau\quad \text{with }c_1:=\frac{r_0\gamma}{4}.
\end{align} 
Choose $\tau_K>0$ such that $\gamma\tau_K\le1$, and choose
$0<c_K\le c_1$. We claim that, whenever
$0<\tau\le\tau_K$ and $h\le c_K\tau$,
\begin{equation}\label{tubeinclusions}
 K\subset D_n^-\subset\Omega_{n,h}\subset D(t_n)
 \subset\mathcal Q_K.
\end{equation}
 Indeed, $K\subset\Gamma(t_n)$ gives $K\subset D_n^-$.
Since $D_n^-\subset \mathcal Q_K\subset \cP_\epsilon^\circ$, the triangulation
covers $D_n^-$, and hence $D_n^-\subset\Omega_{n,h}$ by definition. 
If $\mu\in\Omega_{n,h}$, there exist
$\nu\in D_n^-$ and a selected element $S$ containing both $\mu$ and
$\nu$. Therefore,
\begin{align}\label{apply2}
 \operatorname{dist}\bigl(\mu,\Gamma(t_n)\bigr)
 \le\norm{\mu-\nu}
   +\operatorname{dist}\bigl(\nu,\Gamma(t_n)\bigr)
 \le h+e^{-\gamma\tau/2}r(t_n)
 \le r(t_n),
\end{align} 
where the last inequality follows from
$h\le c_K\tau\le c_1\tau$ and \eqref{tubegap1}.
Thus $\Omega_{n,h}\subset D(t_n)$, while
$D(t_n)\subset \mathcal Q_K$ was established above. This proves \eqref{tubeinclusions}.

\textit{Step 4. Proof of the one-step separation.} It remains to show that the exact one-step image of $\Omega_{n,h}$
lies a distance of order $\tau$ inside the next active region. We first quantify the gap that will be used below. For $\gamma\tau\le1$,
\begin{align}\label{tubegap2}
r(t_{n+1})
\bigl(e^{-\gamma\tau/2}-e^{-3\gamma\tau/2}\bigr)
\ge c_2\tau,
\end{align}
where one may take
$
c_2:=\frac{r_0\gamma e^{-1/2}}{2}.
$ 
Indeed,
$ 
e^{-\gamma\tau/2}-e^{-3\gamma\tau/2}
=
e^{-\gamma\tau/2}(1-e^{-\gamma\tau})
\ge \frac{\gamma e^{-1/2}}{2}\tau.
$ Reducing $\tau_K$ and $c_K$, if necessary, we may assume that
$ 
e^{L\tau_K}\le2,\;
0<c_K\le \min\{c_1,c_2/3\}.
$ 

 For $\mu\in\Omega_{n,h}$, applying \eqref{flowdistance} and \eqref{apply2}, we derive 
\[
 \operatorname{dist}\bigl(
 \Phi_{t_n,t_{n+1}}(\mu),\Gamma(t_{n+1})\bigr)
 \le e^{-3\gamma\tau/2}r(t_{n+1})+e^{L\tau}h.
\]
For every
$y\in\bbR_1^d\setminus D_{n+1}^-$, it follows from the definition of $D_{n+1}^-$ that 
$ 
 \operatorname{dist}\bigl(y,\Gamma(t_{n+1})\bigr)
 >e^{-\gamma\tau/2}r(t_{n+1}).
$ 
Hence
\begin{align*}
 \norm{\Phi_{t_n,t_{n+1}}(\mu)-y}
 &\ge \operatorname{dist}\bigl(y,\Gamma(t_{n+1})\bigr)
 -\operatorname{dist}\bigl(
 \Phi_{t_n,t_{n+1}}(\mu),\Gamma(t_{n+1})\bigr)\\
 &>r(t_{n+1})
 \bigl(e^{-\gamma\tau/2}-e^{-3\gamma\tau/2}\bigr)
 -e^{L\tau}h.
\end{align*} 
Taking the infimum over
$y\in\bbR_1^d\setminus D_{n+1}^-$ and using
\eqref{tubegap2}, $h\le c_K\tau$, and $e^{L\tau}\le2$ yield 
\begin{align*}
 \operatorname{dist}\bigl(
 \Phi_{t_n,t_{n+1}}(\mu),
 \bbR_1^d\setminus D_{n+1}^-\bigr)
 &\ge r(t_{n+1})
 \bigl(e^{-\gamma\tau/2}-e^{-3\gamma\tau/2}\bigr)
 -e^{L\tau}h
 \ge(c_2-2c_K)\tau\ge c_K\tau.
\end{align*}
Since $D_{n+1}^-\subset\Omega_{n+1,h}$, $
\bbR_1^d\setminus\Omega_{n+1,h}
 \subset\bbR_1^d\setminus D_{n+1}^-.$ 
Consequently,
\eqref{separation} holds and the proof is finished. 
\end{proof}
For the domains supplied by Proposition~\ref{tube}, set
\[
 \mathcal T_h^n:=\{S\in\mathfrak T_{h,\epsilon}:S\subset\Omega_{n,h}\},
 \qquad
\mathcal M_h^n:=\{\text{vertices of the elements in }\mathcal T_h^n\}.
\] 
The sets $\Omega_{n,h}$ will be used in the local solvability and convergence
analysis for the numerical scheme \eqref{scheme} (see Theorem~\ref{conv}). 
 For nodal data on
$\mathcal M_h^n$, formula \eqref{interpolant} defines
$\mathcal I_{h,\epsilon}$ on $\Omega_{n,h}$; $ P_{h,\epsilon}v$ denotes
the restriction of $v$ to $\mathcal M_h^n$.  For all sufficiently small $h$,
$\mathcal M_h^n\subset\mathfrak M_{h,\epsilon}^\circ$, and the reconstruction
on $\Omega_{n,h}$ uses the values at these interior nodes.   
The local solvability (Lemma \ref{local}) and error analysis (Section \ref{analysis}) of the numerical scheme require uniform
coefficient bounds on the relevant compact set of density--value
pairs. 
Let $K$ be the compact set fixed in Theorem~\ref{conv}, and let
$\mathcal Q_K$ be given by Proposition~\ref{tube}. Set 
\begin{equation}\label{localcylinder}
 R:=2+\norm{u}_{L^\infty([0,T]\times\mathcal Q_K)},
 \qquad
 \mathcal K_K:=\mathcal Q_K\times\overline B_R(0).
\end{equation}
By the $\mathcal C^1$ regularity of the coefficients and
compactness of $\mathcal K_K$, the bounds below hold with positive constants $C_{\mathcal A}$, $H_0$, $L_H$, and $L_{\mathcal A}$ for all
$(\mu,z),(\nu,w)\in\mathcal K_K$: 
\begin{align}
 &\norm{\mathcal A(\mu,z)}\le C_{\mathcal A},
 \quad \norm{\bm H(\mu,\gradG z)}\le H_0,\label{localbounds}\\&
 \norm{\bm H(\mu,\gradG z)-\bm H(\nu,\gradG w)}
 \le L_H\bigl(\norm{\mu-\nu}+\norm{z-w}\bigr),
 \label{localhlip}\\
 &\norm{\mathcal A(\mu,z)-\mathcal A(\nu,w)}
 \le L_{\mathcal A}\bigl(\norm{\mu-\nu}+\norm{z-w}\bigr).
 \label{localalip}
\end{align}

The next lemma converts an approximate characteristic pair into a numerical solution of \eqref{scheme}. Its hypotheses separate two quantities: the Lipschitz bound of the reconstructed next layer and the residual obtained by inserting the exact pair. In the convergence proof, both are controlled by the next-layer error, which allows local solvability and the error estimate to be propagated together.

\begin{lemma}\label{local} 
Let the conditions of Theorem~\ref{conv} hold. Fix $0\le n<N$ and $\mu^\ell\in\mathcal M_h^n$, and put
$
\mu_*^+:=\Phi_{t_n,t_{n+1}}(\mu^\ell),\;
 u_*^n:=u(t_n,\mu^\ell),$ and $
 \mathcal D_*:=\bbR_1^d\cap
 \overline B(\mu_*^+,c_K\tau/2) \subset\Omega_{n+1,h}.
$  
Assume that for some $\Lambda\ge0$, 
\begin{equation}\label{bootstrap}
 \norm{D_\mu (\mathcal I_{h,\epsilon}u_h^{n+1}) }_{L^\infty(S)}\le\Lambda
 \quad\text{for every }S\in\mathcal T_h^{n+1},
\end{equation}
and that for some $r_\rho\in\bbR_0^d$ and $r_z\in\bbR^d$,
\begin{subequations}
\begin{align}
 \mu_*^+&=\mu^\ell+\tau\mathcal A(\mu_*^+,u_*^n)+r_\rho,
 \label{referencerho}\\
 (I-\tau\DeltaG)u_*^n&=\mathcal I_{h,\epsilon}u_h^{n+1} (\mu_*^+)
 -\tau\bm H(\mu_*^+,\gradG u_*^n)+r_z.
 \label{referenceu}
\end{align}
\end{subequations}
There are $\kappa_0,c_0>0$  independent
of $n$ and $h$, such that if
\begin{equation}\label{smalllocal}
 \tau(1+\Lambda)\le\kappa_0,\qquad
 \norm{r_\rho}+\tau\norm{r_z}\le c_0\tau,
\end{equation}
then \eqref{scheme} has a unique solution $(\mu_\ell^+, u_\ell^n )\in \mathcal D_*\times\overline B(u_*^n,1)$. 
\end{lemma}

\begin{proof}
We solve the coupled system by two fixed-point arguments: Step~1 solves
the value equation for a fixed density, and Step~2 solves the resulting
density equation. 

\textit{Step 1: The value fixed point.}
Because $\mathcal D_*$ is convex and contained in
$\Omega_{n+1,h}$, integration of \eqref{bootstrap} along line segments gives
\begin{equation}\label{vlocal}
 \norm{\mathcal I_{h,\epsilon}u_h^{n+1} (\nu)-\mathcal I_{h,\epsilon}u_h^{n+1} (\widetilde\nu)}
 \le\Lambda\norm{\nu-\widetilde\nu},
 \qquad \nu,\widetilde\nu\in\mathcal D_*.
\end{equation} 
For fixed $\nu\in\mathcal D_*$, define
$ 
 \mathcal Q_\nu(z):=S_\tau\bigl[
 \mathcal I_{h,\epsilon}u_h^{n+1} (\nu)-\tau\bm H(\nu,\gradG z)\bigr]
$ 
on $\mathcal Z_*:=\overline B(u_*^n,1)$.  The fixed-point form of
\eqref{referenceu} is
$
u_*^n=S_\tau\bigl[\mathcal I_{h,\epsilon}u_h^{n+1} (\mu_*^+)
 -\tau\bm H(\mu_*^+,\gradG u_*^n)+r_z\bigr].
$ 
By \eqref{separation},
$\mathcal D_*\subset\Omega_{n+1,h}\subset\mathcal Q_K$.  Moreover,
$\mu^\ell\in\Omega_{n,h}\subset\mathcal Q_K$ and
\[
 \norm{u_*^n}\le R-2,
 \qquad \norm{z}\le\norm{u_*^n}+1\le R-1
 \quad\text{for }z\in\mathcal Z_*.
\]
Therefore $\mathcal D_*\times\mathcal Z_*\subset\mathcal K_K$, so the
local estimate \eqref{localhlip} is applicable. Together with   Lemma~\ref{resolvent} and
\eqref{vlocal}, we obtain \begin{align}
 \norm{\mathcal Q_\nu(u_*^n)-u_*^n}
 &\le(\Lambda+\tau L_H)\norm{\nu-\mu_*^+}+\norm{r_z},
 \label{valuecenter}\\
 \norm{\mathcal Q_\nu(z)-\mathcal Q_\nu(\widetilde z)}
 &\le\tau L_H\norm{z-\widetilde z}.
 \label{valuecontract}
\end{align}
Since $\norm{\nu-\mu_*^+}\le c_K\tau/2$, choose
$\kappa_0,c_0>0$ so that
$ 
\kappa_0L_H\le\frac12,
\;
 \frac{c_K}{2}\bigl(\kappa_0+L_H\kappa_0^2\bigr)+c_0
\le\frac12.$ Since $\tau>0$, the second condition in
\eqref{smalllocal} implies $\norm{r_z}\le c_0$. Hence
\begin{align*}
 (\Lambda+\tau L_H)\frac{c_K\tau}{2}+\norm{r_z}
 \le\frac{c_K}{2}
 \bigl(\kappa_0+L_H\kappa_0^2\bigr)+c_0
 \le\frac12,\quad 
 \tau L_H\le\kappa_0L_H\le\frac12.
\end{align*} 
Then $\mathcal Q_\nu$ maps $\mathcal Z_*$ into itself and is a contraction. Applying 
Banach's theorem, there exists a unique fixed point $z(\nu)\in\mathcal Z_*$ satisfying
\begin{equation}\label{value-eq}
 (I-\tau\DeltaG)z(\nu)={\bigl(\mathcal I_{h,\epsilon}u_h^{n+1}\bigr)(\nu)}
 -\tau\bm H(\nu,\gradG z(\nu)).
\end{equation} 
For $\nu,\widetilde\nu\in\mathcal D_*$, subtracting the two fixed-point
identities and using Lemma~\ref{resolvent}, \eqref{vlocal} and \eqref{localhlip} yields 
$ 
 (1-\tau L_H)\norm{z(\nu)-z(\widetilde\nu)}
 \le(\Lambda+\tau L_H)\norm{\nu-\widetilde\nu}.
$ 
Thus
\begin{equation}\label{zlip}
 \norm{z(\nu)-z(\widetilde\nu)}
 \le L_z\norm{\nu-\widetilde\nu},
 \qquad L_z:=\frac{\Lambda+\tau L_H}{1-\tau L_H}.
\end{equation}

\textit{Step 2: The density fixed point.}
Define $\mathcal R:\mathcal D_*\to\bbR_1^d$ by
$ 
 \mathcal R(\nu):=\mu^\ell+\tau\mathcal A(\nu,z(\nu))$.  By \eqref{localalip} and \eqref{zlip},
\begin{equation}\label{rhocontract}
 \norm{\mathcal R(\nu)-\mathcal R(\widetilde\nu)}
 \le q_\tau\norm{\nu-\widetilde\nu},
 \qquad q_\tau:=\tau L_{\mathcal A}(1+L_z).
\end{equation}
Moreover, \eqref{referencerho} implies
$ 
 \mathcal R(\mu_*^+)-\mu_*^+
 =\tau\bigl[\mathcal A(\mu_*^+,z(\mu_*^+))
 -\mathcal A(\mu_*^+,u_*^n)\bigr]-r_\rho.
$ 
Since $z(\mu_*^+)$ is the fixed point of
$\mathcal Q_{\mu_*^+}$,
$
 z(\mu_*^+)=\mathcal Q_{\mu_*^+}\bigl(z(\mu_*^+)\bigr).
$ 
Setting $\nu=\mu_*^+$ in \eqref{valuecenter}, we obtain
$ 
 \norm{\mathcal Q_{\mu_*^+}(u_*^n)-u_*^n}\le\norm{r_z}.
$ 
Therefore, by \eqref{valuecontract},
\begin{align*}
 \norm{z(\mu_*^+)-u_*^n}
 &\le\norm{\mathcal Q_{\mu_*^+}\bigl(z(\mu_*^+)\bigr)
          -\mathcal Q_{\mu_*^+}(u_*^n)}
      +\norm{\mathcal Q_{\mu_*^+}(u_*^n)-u_*^n}
      \le\tau L_H\norm{z(\mu_*^+)-u_*^n}+\norm{r_z}.
\end{align*}
Since $\tau L_H<1$, it follows that
$ 
 \norm{z(\mu_*^+)-u_*^n}\le\frac{\norm{r_z}}{1-\tau L_H}.
$ 
Consequently,
\begin{equation}\label{densitycenter}
 \norm{\mathcal R(\mu_*^+)-\mu_*^+}
 \le\frac{\tau L_{\mathcal A}}{1-\tau L_H}\norm{r_z}
 +\norm{r_\rho}.
\end{equation}
Reduce $\kappa_0$ and $c_0$ so that \eqref{smalllocal} implies
$q_\tau\le1/2$ and the right-hand side of \eqref{densitycenter} is at most
 $c_K\tau/4$.  Then, by \eqref{rhocontract}, for $\nu\in\mathcal D_*$,
$ 
 \norm{\mathcal R(\nu)-\mu_*^+}
 \le q_\tau\norm{\nu-\mu_*^+}
 +\norm{\mathcal R(\mu_*^+)-\mu_*^+}
 \le\frac{c_K\tau}{2}.$ 
Thus $\mathcal R$ is a contraction from $\mathcal D_*$
into itself, and thus it has a unique fixed point
$\mu_\ell^+\in\mathcal D_*$.  Set $u_\ell^n:=z(\mu_\ell^+)$.  The fixed-point
identities for $\mathcal R$ and $z$ recover \eqref{scheme}, and
$\mathcal D_*\subset\Omega_{n+1,h}$. 

If
$(\widetilde\mu,\widetilde u)\in
\mathcal D_*\times\overline B(u_*^n,1)$ is another solution of
\eqref{scheme}, then the uniqueness in \eqref{value-eq} implies 
$\widetilde u=z(\widetilde\mu)$, and the density equation yields
$\mathcal R(\widetilde\mu)=\widetilde\mu$.  The uniqueness of the fixed
point of $\mathcal R$ yields $\widetilde\mu=\mu_\ell^+$ and hence
$\widetilde u=u_\ell^n$. 
\end{proof}

\section{Convergence analysis}\label{analysis}

This section proves Theorem~\ref{conv}, Corollary~\ref{feedback}, and
Theorem~\ref{actiontheorem}. 
We first estimate the local residual and the stability of a numerical update. A backward induction then constructs the numerical solution and proves the master-field estimate. The feedback estimate follows by forward error propagation, and the action estimate by a feasible variation around the exact minimizer.

\subsection{Local error estimate}
The next lemma quantifies the local error of the numerical 
step along an exact characteristic. 
\begin{lemma}
\label{residual}
Let the conditions of Theorem~\ref{conv} hold. Let $0\le n<N$ and
$\mu\in\Omega_{n,h}$, and set $u_*^n:=u(t_n,\mu)$ and
$\mu_*^+:=\Phi_{t_n,t_{n+1}}(\mu)$. Then 
\begin{subequations}
\begin{align}
 &\mu_*^+
 =\mu+\tau\mathcal A(\mu_*^+,u_*^n)+r_\mu^n,
 \quad \norm{r_\mu^n}\le C\tau^2,
 \label{rhores}\\
 &(I-\tau\DeltaG)u_*^n
 =\bigl(\mathcal I_{h,\epsilon}P_{h,\epsilon}
 u(t_{n+1},\cdot)\bigr)(\mu_*^+) 
 -\tau\bm H(\mu_*^+,\gradG u_*^n)+r_u^n,
 &\norm{r_u^n}&\le C\bigl(\tau^2+h\omega_K(h)\bigr),
 \label{ures}
\end{align}
\end{subequations}
\end{lemma}
\begin{proof}
Let $\rho(s):=\rho(s;t_n,\mu)$ be given by
\eqref{representation}, and set $w(s)=u(s,\rho(s))$, so that
$w(t_n)=u_*^n$ and
$w(t_{n+1})=u(t_{n+1},\mu_*^+)$.
Since $\mu\in\Omega_{n,h}$, \eqref{tubeinclusions} and
\eqref{flowdistance} imply
\[
 \rho(s)\in D(s)\subset\mathcal Q_K,
 \qquad t_n\le s\le t_{n+1},
 \qquad
 \mu_*^+=\rho(t_{n+1})\in\Omega_{n+1,h},
\]
where the last inclusion follows from \eqref{separation}.
Integrating the density equation in \eqref{representation}
over $[t_n,t_{n+1}]$ yields 
\begin{align*}
 \mu_*^+-\mu&=\tau\mathcal A(\mu_*^+,u_*^n)+r_\mu^n\quad \text{with }\;
 r_\mu^n:=\int_{t_n}^{t_{n+1}}
 \bigl[\mathcal A(\rho(s),w(s))-\mathcal A(\mu_*^+,u_*^n)\bigr]\,\mathrm ds.
 \end{align*} 
We then estimate $\|r^n_{\mu}\|.$  Along the characteristic, \eqref{eq_intro} gives
\begin{equation}\label{chain}
 \dot w(s)=\bm H(\rho(s),\gradG w(s))-\DeltaG w(s).
\end{equation}
Moreover, 
$(\rho(s),w(s))\in\mathcal K_K$, where we recall $\mathcal K_K$ is given in \eqref{localcylinder}.   Equation
\eqref{localbounds} and the boundedness of $\DeltaG$ then imply, for some
$C>0$,
\begin{align}\label{udisp}
  \norm{\dot w(s)}\le C,\quad 
  \norm{w(s)-u_*^n}
  \le\int_{t_n}^s\norm{\dot w(r)}\,\mathrm dr
  \le C(s-t_n).
 \end{align} 
Moreover, integrating the characteristic backward from $t_{n+1}$ yields
\begin{align}\label{rhodisp}
  \norm{\rho(s)-\mu_*^+}
  &=\norm{-\int_s^{t_{n+1}}\dot\rho(r)\,\mathrm dr}
  =\norm{-\int_s^{t_{n+1}}
    \mathcal A(\rho(r),w(r))\,\mathrm dr}
  \le C_{\mathcal A}(t_{n+1}-s).
 \end{align} 
By \eqref{localalip},
\eqref{udisp}, and
\eqref{rhodisp},
$ 
 \norm{r_\mu^n}
 \le L_{\mathcal A}\int_{t_n}^{t_{n+1}}
 \bigl[C_{\mathcal A}(t_{n+1}-s)+C(s-t_n)\bigr]\,\mathrm ds\leq C\tau^2, 
$ 
where we use the fact that the pairs $(\rho(s),w(s))$ and
$(\mu_*^+,u_*^n)$ belong to $\mathcal K_K$.

Integration of \eqref{chain} yields
$ 
 w(t_{n+1})-u_*^n
 =\int_{t_n}^{t_{n+1}}
 \bigl[\bm H(\rho(s),\gradG w(s))-\DeltaG w(s)\bigr]\mathrm ds, 
$  which leads to 
\begin{align*}
 (I-\tau\DeltaG)u_*^n
 &=w(t_{n+1})-\tau\bm H(\mu_*^+,\gradG u_*^n)\\
 &\quad -\int_{t_n}^{t_{n+1}}
 \Bigl[\bm H(\rho(s),\gradG w(s))
       -\bm H(\mu_*^+,\gradG u_*^n)-\DeltaG\bigl(w(s)-u_*^n\bigr)\Bigr]\,\mathrm ds.  
\end{align*}
Hence we obtain \eqref{ures} with 
$r_u^n:=w(t_{n+1})
-\bigl(\mathcal I_{h,\epsilon}P_{h,\epsilon}
u(t_{n+1},\cdot)\bigr)(\mu_*^+)- \int_{t_n}^{t_{n+1}}
 \bigl[\bm H(\rho(s),\gradG w(s))
       -\bm H(\mu_*^+,\gradG u_*^n)-\DeltaG\bigl(w(s)-u_*^n\bigr)\bigr]\,\mathrm ds$.  
Since every element $S\in\mathcal T_h^{n+1}$ is contained in
$\Omega_{n+1,h}\subset\mathcal Q_K$, the elementwise interpolation
estimate \eqref{ierror} gives 
\begin{align}
\norm{r_u^n}
& \le C_Ih\omega_K(h)
 + \bigl(L_H+\norm{\DeltaG}_{\mathcal L(\bbR^d)}\bigr)
 \int_{t_n}^{t_{n+1}}
 \bigl[\norm{\rho(s)-\mu_*^+}+\norm{w(s)-u_*^n}\bigr]\,\mathrm ds\notag\\
 & \leq C(h\omega_K(h)
 +\tau^2), 
 \end{align}
 where we use \eqref{localhlip} and  the boundedness of
$\DeltaG$. 
 This proves both residual estimates with constants independent of
$n$, $h$, and $\mu$. 
\end{proof}

Set the global error at $t_n$ time level: 
$
E_n:=\max_{\mu^\ell\in\mathcal M_h^n}
 \norm{u_h^n(\mu^\ell)-u(t_n,\mu^\ell)},
$ 
We then compare the numerical solution with the exact characteristic, and derive the
one-step stability estimate used in the global induction.
\begin{lemma}
\label{stability} 
Let the conditions of Theorem~\ref{conv} hold. Fix $0\le n<N$ and
$\mu^\ell\in\mathcal M_h^n$, and set
$
 u_*^n:=u(t_n,\mu^\ell),\;
 \mu_*^+:=\Phi_{t_n,t_{n+1}}(\mu^\ell),\;
 \mathcal D_*:=\bbR_1^d\cap\overline B(\mu_*^+,c_K\tau/2).
$
Let $(\mu_\ell^+,u_\ell^n)$ be a solution of \eqref{scheme} in
$\mathcal D_*\times\overline B(u_*^n,1)$.
There are $\tau_0,C>0$, independent of
$n$ and $h$, such that, for $0<\tau\le\tau_0$, 
\begin{align}
 &\norm{\mu_\ell^+-\mu_*^+}
 \le C\tau E_{n+1}
 +C\bigl(\tau^2+\tau h\omega_K(h)\bigr),
 \label{rhostep}\\
 &\norm{u_\ell^n-u_*^n}
 \le(1+C\tau)E_{n+1}
+C\bigl(\tau^2+h\omega_K(h)\bigr),
 \label{ustep}
\end{align} 
where $\|\cdot\|$ is the Euclidean norm
in $\bbR^d$. \end{lemma}
\begin{proof} 
At the fixed node $\mu^{\ell}$, write $e_u:=u_\ell^n-u_*^n$ and 
$e_\mu:=\mu_\ell^+-\mu_*^+$. 
Since $\mu^\ell\in\mathcal M_h^n\subset\Omega_{n,h}$,
Lemma~\ref{residual} applies with $\mu=\mu^\ell$. Subtracting
\eqref{rhores} from
\eqref{schemerho} implies 
\begin{equation*}
e_\mu
=\tau\bigl[ \mathcal A(\mu_\ell^+,u_\ell^n)
-\mathcal A(\mu_*^+,u_*^n) \bigr]-r_\mu^n.
\end{equation*}
By the choice of the localized solution,
$\mu_\ell^+,\mu_*^+\in\mathcal D_*\subset\Omega_{n+1,h}
\subset\mathcal Q_K$ and
$\norm{u_\ell^n-u_*^n}\le1$.  Since 
$\norm{u_*^n}\le R-2$, both pairs
$(\mu_\ell^+,u_\ell^n)$ and $(\mu_*^+,u_*^n)$ belong to
$\mathcal K_K$, where we recall that $R$ and $\mathcal K_K$ are given in \eqref{localcylinder}.  Hence
\eqref{localalip} is applicable to the difference above.  Together
with Lemma~\ref{residual}, it implies 
$(1-\tau L_{\mathcal A})\norm{e_\mu}
\le\tau L_{\mathcal A}\norm{e_u}+C\tau^2$.
Choose $\tau$ so that
$\tau L_{\mathcal A}\le1/2$.  Then
\begin{equation}\label{rhofirst}
 \norm{e_\mu}
 \le2\tau L_{\mathcal A}\norm{e_u}+C\tau^2.
\end{equation}
Both $\mu_\ell^+$ and $\mu_*^+$ belong to the convex ball
$\mathcal D_*\subset\Omega_{n+1,h}$. Applying the elementwise estimate
\eqref{ilip} along the segment joining these two points gives 
\begin{align}\label{Lip-bound}
 \norm{\bigl(\mathcal I_{h,\epsilon}P_{h,\epsilon}
 u(t_{n+1},\cdot)\bigr)(\mu_\ell^+)
 -\bigl(\mathcal I_{h,\epsilon}P_{h,\epsilon}
 u(t_{n+1},\cdot)\bigr)(\mu_*^+)}
 \le C_I\norm{D_\mu u}_{L^\infty([0,T]\times\mathcal Q_K)}
 \norm{\mu_\ell^+-\mu_*^+}, 
\end{align}where $C_I$ is the interpolation constant in Lemma~\ref{interpolation}.
Choose $\mathsf S_\ell^+\in\mathcal T_h^{n+1}$ with
$\mu_\ell^+\in\mathsf S_\ell^+$, and denote its vertices by
$\mu^{\ell_1},\ldots,\mu^{\ell_d}$. Such an active element
exists because $\mu_\ell^+\in\mathcal D_*\subset\Omega_{n+1,h}$ and
$\Omega_{n+1,h}$ is the union of the elements in
$\mathcal T_h^{n+1}$; hence all its vertices belong to
$\mathcal M_h^{n+1}$. Then
\eqref{interpolant} and the definition of $E_{n+1}$ imply
\begin{align*}
\norm{\mathcal I_{h,\epsilon} u^{n+1}_h (\mu_\ell^+)
-\bigl(\mathcal I_{h,\epsilon}P_{h,\epsilon}
u(t_{n+1},\cdot)\bigr)(\mu_\ell^+)}\le\sum_{r=1}^d\lambda_r(\mu_\ell^+)
 \norm{u_h^{n+1}(\mu^{\ell_r})
       -u(t_{n+1},\mu^{\ell_r})}
 \le E_{n+1}.
\end{align*} 
Combining this estimate with \eqref{Lip-bound}, we obtain
\begin{align}
 &\norm{ \mathcal I_{h,\epsilon} u^{n+1}_h(\mu_\ell^+)
 - \bigl(\mathcal I_{h,\epsilon}P_{h,\epsilon}
 u(t_{n+1},\cdot)\bigr)(\mu_*^+)}
 \le E_{n+1}+{C_I\norm{D_\mu u}_{L^\infty([0,T]\times\mathcal Q_K)}}
 \norm{e_\mu}.
 \label{split}
\end{align}

Subtracting \eqref{ures} from
\eqref{schemeu} gives the exact identity
\begin{align*}
 (I-\tau\DeltaG)e_u
 &=\mathcal I_{h,\epsilon} u^{n+1}_h (\mu_\ell^+)
 -\bigl(\mathcal I_{h,\epsilon}P_{h,\epsilon}
 u(t_{n+1},\cdot)\bigr)(\mu_*^+)\\&\quad -\tau\bigl[
 \bm H(\mu_\ell^+,\gradG u_\ell^n)
 -\bm H(\mu_*^+,\gradG u_*^n)
 \bigr]-r_u^n.
\end{align*}
Apply Lemma~\ref{resolvent},
\eqref{localhlip},
\eqref{split}, and the residual estimate  \eqref{ures} to get
$$\norm{e_u}\le E_{n+1}+{C_I\norm{D_\mu u}_{L^\infty([0,T]\times\mathcal Q_K)}}\norm{e_\mu}
+\tau L_H\bigl(\norm{e_\mu}+\norm{e_u}\bigr)
+C\bigl(\tau^2+h\omega_K(h)\bigr).$$
Inserting \eqref{rhofirst} into this inequality yields \begin{equation}\label{upre}
 (1-q_\tau)\norm{e_u}
 \le E_{n+1}+C\bigl(\tau^2+h\omega_K(h)\bigr),
\end{equation} where $q_\tau=2\tau L_{\mathcal A}(C_I\norm{D_\mu u}_{L^\infty([0,T]\times\mathcal Q_K)}+\tau L_H)+\tau L_H$. 
Reduce
$\tau$ so that $q_\tau\le1/2$ and   then
$
 \frac1{1-q_\tau}=1+\frac{q_\tau}{1-q_\tau}
 \le1+2C\tau, 
$ 
which together with  \eqref{upre} proves
\eqref{ustep}.  By  \eqref{ustep} and \eqref{rhofirst}, we deduce 
\begin{align*}
 \norm{e_\mu}
 &\le2\tau L_{\mathcal A}\bigl[
 (1+C\tau)E_{n+1}
 +C\bigl(\tau^2+h\omega_K(h)\bigr)\bigr]+C\tau^2\le C\tau E_{n+1}
 +C\bigl(\tau^2+\tau h\omega_K(h)\bigr),
\end{align*}
which is \eqref{rhostep}. This finishes the proof. 
\end{proof}
\subsection{Proofs of main results}
This subsection presents proofs of Theorem~\ref{conv}, Corollary~\ref{feedback}, and
Theorem~\ref{actiontheorem}. 
We begin the proof of Theorem~\ref{conv} with a backward induction argument that has two tasks at each time level. The next-layer error controls the derivative of its reconstruction through \eqref{inverse}, and therefore the contraction condition in Lemma~\ref{local}. Once the current update has been constructed, Lemma~\ref{stability} gives its error. 
\begin{proof}[Proof of Theorem~\ref{conv}]
Since $h\to0$ and $\tau^2/h\to0$ in \eqref{coupling}, $\tau\to0$.
Moreover, the continuity of $D_\mu u$ on the compact set
$[0,T]\times\mathcal P_{\epsilon_K}$ implies, by uniform continuity, that
$\omega_K(h)\to0$.
Let $\tau_0>0$ be given in Lemma~\ref{stability}. Fix $C_0>0$ sufficiently large so that the estimates in Lemmas~\ref{residual} and~\ref{stability}
hold with $C=C_0$, and set $C_E:=C_0Te^{C_0T}$.
We fix $C_0$ and $C_E$ before choosing $h_0$ below and keep them unchanged
throughout the induction. Let $\kappa_0,c_0>0$ be given in
Lemma~\ref{local}.  By
\eqref{coupling}, we have $
 \tau+\frac{h\omega_K(h)}{\tau}\to0,\;
 \frac{\tau}{h}\bigl(
 \tau+\frac{h\omega_K(h)}{\tau}\bigr)
 =\frac{\tau^2}{h}+\omega_K(h)\to0.
$  Hence $h_0$ can be chosen so that, for $h\le h_0$, we have $\tau\le\min\{1,\tau_0,\tau_K\},$ \begin{align}\label{small}
& C_E\bigl(\tau+\frac{h\omega_K(h)}{\tau}\bigr)
 \le\frac12,\quad 
 \tau\bigl(1+{C_I\norm{D_\mu u}_{L^\infty([0,T]\times\mathcal Q_K)}}\bigr)+C_IC_E\bigl(
 \frac{\tau^2}{h}+\omega_K(h)\bigr)\le\kappa_0,\notag\\
 &C_0\tau+C_0\bigl(\tau^2+h\omega_K(h)\bigr)
 +C_E\bigl(\tau+\frac{h\omega_K(h)}{\tau}\bigr)\le c_0.
\end{align}

We use backward induction to prove, at each time level, the well-posedness of the numerical
solution and its 
error bound. The induction assertion at level $n$ is
\begin{equation}\label{induction}
 u_h^k\ \text{is constructed},
 \qquad
 E_k\le C_E\bigl(\tau+\frac{h\omega_K(h)}{\tau}\bigr),
 \qquad k=n,\ldots,N.
\end{equation}
Since $u_h^N(\mu^\ell)=g(\mu^\ell)=u(T,\mu^\ell)$ at every terminal
node, $E_N=0$, so the assertion holds for $n=N$.

Fix $0\le n<N$ and assume that the assertion \eqref{induction} holds for $n+1$. Thus  the layers $u_h^k$, $n+1\le k\le N$, have been constructed and
$E_k\le C_E\bigl(\tau+h\omega_K(h)/\tau\bigr)$ for these indices.
To construct layer $n$ on the active nodes, we apply Lemma \ref{local} and
verify its conditions.  Let $\mathcal Q_K$ and $\Omega_{n,h}$ be given by Proposition~\ref{tube}.  
Since
$\mathcal M_h^{n+1}\subset\Omega_{n+1,h}\subset\mathcal Q_K$,
one has
$
\max_{\mu^\ell\in\mathcal M_h^{n+1}}
 \norm{u(t_{n+1},\mu^\ell)}
 \le\norm{u}_{L^\infty([0,T]\times\mathcal Q_K)}=R-2.
$
By the definition of the nodal restriction $ P_{h,\epsilon}$, for every
$\mu^\ell\in\mathcal M_h^{n+1}$,
\[
 \norm{u_h^{n+1}(\mu^\ell)
 - P_{h,\epsilon}u(t_{n+1},\mu^\ell)}
 =\norm{u_h^{n+1}(\mu^\ell)-u(t_{n+1},\mu^\ell)}
 \le E_{n+1}.
\] 
Since every element $S\in\mathcal T_h^{n+1}$ is contained in
$\Omega_{n+1,h}\subset\mathcal Q_K$, the elementwise forms of
\eqref{istab}, \eqref{ilip}, and \eqref{inverse}, together with the
linearity of $\mathcal I_{h,\epsilon}$, give
$ \norm{\mathcal I_{h,\epsilon} P_{h,\epsilon}u(t_{n+1},\cdot)}_{L^\infty(\Omega_{n+1,h})}\le R-2,$
\begin{align}
& \max_{S\in\mathcal T_h^{n+1}}
 \norm{D_\mu (\mathcal I_{h,\epsilon} P_{h,\epsilon}u(t_{n+1},\cdot))}_{L^\infty(S)}
 \le {C_I\norm{D_\mu u}_{L^\infty([0,T]\times\mathcal Q_K)}},\notag\\
 %\label{inductexactlip}\\
& \norm{\mathcal I_{h,\epsilon}u^{n+1}_h-\mathcal I_{h,\epsilon} P_{h,\epsilon}u(t_{n+1},\cdot)}_{L^\infty(\Omega_{n+1,h})}
 \le E_{n+1},
\label{inductfieldlip_2}\\ 
 &\max_{S\in\mathcal T_h^{n+1}}
 \norm{D_\mu( \mathcal I_{h,\epsilon}u^{n+1}_h-\mathcal I_{h,\epsilon} P_{h,\epsilon}u(t_{n+1},\cdot))}_{L^\infty(S)}
 \le C_Ih^{-1}E_{n+1}.\notag
 %\label{inductfieldlip_1}
\end{align}
The triangle inequality and the induction hypothesis therefore verify
\eqref{bootstrap} with
\[
 \Lambda:={C_I\norm{D_\mu u}_{L^\infty([0,T]\times\mathcal Q_K)}}+C_Ih^{-1}E_{n+1}
 \le {C_I\norm{D_\mu u}_{L^\infty([0,T]\times\mathcal Q_K)}}+C_IC_Eh^{-1}\bigl(\tau+\frac{h\omega_K(h)}{\tau}\bigr).
\]
It follows from the inequality involving $\kappa_0$ in \eqref{small} that 
\begin{align*}
 \tau(1+\Lambda)
\le
\tau\bigl(1+{C_I\norm{D_\mu u}_{L^\infty([0,T]\times\mathcal Q_K)}}\bigr)+C_IC_E
 \bigl(\frac{\tau^2}{h}+\omega_K(h)\bigr)
 \le\kappa_0.
\end{align*}

For each $\mu^\ell\in\mathcal M_h^n$, use the exact reference pair
$(\mu_*^+,u_*^n)$ introduced in Lemma~\ref{local}. Since
$\mathcal M_h^n\subset\Omega_{n,h}$, Lemma~\ref{residual} applies with
$\mu=\mu^\ell$ and provides $r_\mu^n,r_u^n$ in
\eqref{rhores}--\eqref{ures}. Define 
\begin{equation}
 \widehat r_u^n:=r_u^n
 +\bigl(\mathcal I_{h,\epsilon}P_{h,\epsilon}
 u(t_{n+1},\cdot)\bigr)(\mu_*^+) 
 -{\bigl(\mathcal I_{h,\epsilon}u_h^{n+1}\bigr)(\mu_*^+)}.
\end{equation}
Then \eqref{rhores} and \eqref{ures} are exactly
\eqref{referencerho}--\eqref{referenceu} with
$r_\rho=r_\mu^n$ and $r_z=\widehat r_u^n$.
Since $\mu_*^+\in\Omega_{n+1,h}$, \eqref{inductfieldlip_2} and
Lemma~\ref{residual} yield
\[
 \norm{\widehat r_u^n}
 \le C_0\bigl(\tau^2+h\omega_K(h)\bigr)+E_{n+1},
 \qquad \norm{r_\mu^n}\le C_0\tau^2.
\]
The induction hypothesis \(E_{n+1}\le C_E(\tau+\frac{h\omega_K(h)}{\tau})\) and the last inequality in \eqref{small} ensure  
\[
 \norm{r_\mu^n}+\tau\norm{\widehat r_u^n}
 \le\tau\Bigl[C_0\tau+C_0\bigl(\tau^2+h\omega_K(h)\bigr)
 +C_E\bigl(\tau+\frac{h\omega_K(h)}{\tau}\bigr)\Bigr]
 \le c_0\tau.
\]
Thus the conditions in \eqref{smalllocal} hold.
We have verified all the conditions of Lemma~\ref{local}, and  thus there exists a unique pair
$(\mu_\ell^+,u_\ell^n)$ at every node $\mu^\ell\in\mathcal M_h^n$ and the
values $u_\ell^n$ define the layer $u_h^n$.
By the definition of $\mathcal D_*$ in Lemma~\ref{local}, for every
$\mu^{\ell}\in\mathcal M_h^n$,
$\mu_\ell^+\in\Omega_{n+1,h}$ and
$(\mu_\ell^+,u_\ell^n)\in\mathcal K_K$.
Applying Lemma~\ref{stability} at level $n$ and taking the maximum in \eqref{ustep} over
$\mu^\ell\in\mathcal M_h^n$ gives
\begin{equation}\label{recursion}
 E_n\le(1+C_0\tau)E_{n+1}
 +C_0\bigl(\tau^2+h\omega_K(h)\bigr).
\end{equation}
By the induction construction, the same recursion already holds at the
levels $k=n+1,\ldots,N-1$.  Starting from $E_N=0$ and iterating
\eqref{recursion}, we deduce  
\begin{align*}
 E_n
 &\le C_0\bigl(\tau^2+h\omega_K(h)\bigr)
 \sum_{j=0}^{N-n-1}(1+C_0\tau)^j\le C_0\bigl(\tau^2+h\omega_K(h)\bigr)
 \frac{T}{\tau}e^{C_0T}
 =C_E\bigl(\tau+\frac{h\omega_K(h)}{\tau}\bigr),
\end{align*} where we use 
$ 
 \sum_{j=0}^{N-n-1}(1+C_0\tau)^j
 \le Ne^{C_0T}=\frac{T}{\tau}e^{C_0T}
 $ and $C_E=C_0Te^{C_0T}$, and the factor \(1/\tau\) results from accumulating the local interpolation error \(h\omega_K(h)\) over \(\mathcal O(\tau^{-1})\) time steps. 
Thus the assertion \eqref{induction} also holds at $k=n$, with the same
constant $C_E$ fixed above.
This completes the induction.  

Finally, the localized interpolation estimates give
\[
 \begin{aligned}
 &\norm{\mathcal I_{h,\epsilon}u_h^n-u(t_n,\cdot)}_{L^\infty(\Omega_{n,h})}\le\norm{\mathcal I_{h,\epsilon}(u_h^n-P_{h,\epsilon}u(t_n,\cdot))}_{L^\infty(\Omega_{n,h})}
       +C_Ih\omega_K(h)\le E_n+C_Ih\omega_K(h).
 \end{aligned}
\]
Together with \eqref{induction}, this proves \begin{equation}\label{activeerror}
 \max_{0\le n\le N}
 \norm{\mathcal I_{h,\epsilon}u_h^n-u(t_n,\cdot)}_{L^\infty(\Omega_{n,h})}
 \le C_K\bigl(\tau+\frac{h\omega_K(h)}{\tau}\bigr)+C_Ih\omega_K(h) .
\end{equation} 
and hence \eqref{error}, since $K\subset\Omega_{n,h}$ and $\tau\le1$.
\end{proof}

\begin{proof}[Proof of Corollary~\ref{feedback}] 
For brevity, write
$\rho(t):=\rho(t;0,\mu_0)$ and $\phi(t):=\phi(t;0,\mu_0)$. 
Since
$\mu_0\in K$, one has
$\rho(t_n) \in\Gamma(t_n)\subset\Omega_{n,h}$.  By the construction in the proof
of Proposition~\ref{tube}, there is some $\sigma_K>0$ (e.g. $\sigma_K:=\frac12r_0e^{-\gamma\tau_K/2}\in (0, e^{-\gamma\tau/2}r(t_n))$),    
such that for all sufficiently
small $h$,
\begin{equation}\label{exactmargin}
 \{q\in\bbR_1^d:\operatorname{dist}(q,\Gamma(t_n))
 \le\sigma_K\}\subset\Omega_{n,h},
 \qquad 0\le n\le N.
\end{equation}
Suppose that
$\widehat\rho_h^n$ has been constructed and set
$e_n:=\norm{\widehat\rho_h^n-\rho(t_n)}$. 
If $e_n\le\sigma_K$, then \eqref{exactmargin} gives $\widehat\rho_h^n\in\Omega_{n,h}$, and the value
$\widehat\phi_h^n=(\mathcal I_{h,\epsilon}u_h^n)(\widehat\rho_h^n)$ is defined. By \eqref{activeerror} and \eqref{localcylinder}, $\norm{\widehat\phi_h^n}<R-1$ for sufficiently small $h$.
To solve \eqref{feedbackstep}, set
$\rho_*^{n+1}:=\Phi_{t_n,t_{n+1}}(\widehat\rho_h^n)$ and consider $
 \mathcal R_n(q):=\widehat\rho_h^n
 +\tau\mathcal A(q,\widehat\phi_h^n)
$
on the closed ball 
$ 
\bbR_1^d\cap
 \overline B(\rho_*^{n+1},c_K\tau/2),
$ which is contained in $\Omega_{n+1,h}$ 
by \eqref{separation}. 
Lemma~\ref{residual}, applied with
$\mu=\widehat\rho_h^n$, leads to 
\[
 \rho_*^{n+1}=\widehat\rho_h^n
 +\tau\mathcal A\bigl(\rho_*^{n+1},
 u(t_n,\widehat\rho_h^n)\bigr)+r_*^n,
 \qquad \norm{r_*^n}\le C\tau^2.
\]
Since
$\rho_*^{n+1},q,\widetilde q\in \bbR_1^d\cap
 \overline B(\rho_*^{n+1},c_K\tau/2) \subset\mathcal Q_K$ and
$\widehat\phi_h^n,u(t_n,\widehat\rho_h^n)\in\overline B_R(0)$,
the pairs occurring below belong to $\mathcal K_K$, and hence
\eqref{localalip} applies.  The definition of $\mathcal R_n$ and the
residual identity above imply
\begin{align*}
 \mathcal R_n(\rho_*^{n+1})-\rho_*^{n+1}
 =\tau\Bigl[
 \mathcal A(\rho_*^{n+1},\widehat\phi_h^n)
 -\mathcal A\bigl(\rho_*^{n+1},
 u(t_n,\widehat\rho_h^n)\bigr)\Bigr]-r_*^n.
\end{align*}
Since $\widehat\phi_h^n=(\mathcal I_{h,\epsilon}u_h^n)(\widehat\rho_h^n)$, estimate 
\eqref{activeerror} yields  
\[
 \begin{aligned}
 \norm{\mathcal R_n(\rho_*^{n+1})-\rho_*^{n+1}}\le\tau L_{\mathcal A}\norm{(\mathcal I_{h,\epsilon}u_h^n)(\widehat\rho_h^n)-u(t_n,\widehat\rho_h^n)}+\norm{r_*^n}\le\tau L_{\mathcal A}C_K\bigl(\tau+\frac{h\omega_K(h)}{\tau}\bigr)+C\tau^2.
 \end{aligned}
\] 
Moreover, for $q,\widetilde q\in \bbR_1^d\cap
 \overline B(\rho_*^{n+1},c_K\tau/2)$,
\begin{align*}
 \norm{\mathcal R_n(q)-\mathcal R_n(\widetilde q)}
 &=\tau\norm{\mathcal A(q,\widehat\phi_h^n)
 -\mathcal A(\widetilde q,\widehat\phi_h^n)}
 \le\tau L_{\mathcal A}\norm{q-\widetilde q}.
\end{align*} 
For sufficiently small $h$, $\tau L_{\mathcal A}\le1/2$ and
$\tau L_{\mathcal A}C_K\bigl(\tau+\frac{h\omega_K(h)}{\tau}\bigr)
+C\tau^2\le c_K\tau/4$. 
Hence $\mathcal R_n$ maps $\bbR_1^d\cap
 \overline B(\rho_*^{n+1},c_K\tau/2)$ into itself and is a contraction.
Its unique fixed point $\widehat\rho_h^{n+1}$ belongs to
$\bbR_1^d\cap
 \overline B(\rho_*^{n+1},c_K\tau/2)\subset\Omega_{n+1,h}$.  If
$q,\widetilde q\in\Omega_{n+1,h}$ both solve \eqref{feedbackstep}, then
$ 
 \norm{q-\widetilde q}
 \le\tau L_{\mathcal A}\norm{q-\widetilde q}.
$ 
Since $\tau L_{\mathcal A}<1$, one has $q=\widetilde q$.  Thus
$\widehat\rho_h^{n+1}$ is the unique solution of \eqref{feedbackstep} in
$\Omega_{n+1,h}$. 

Lemma~\ref{residual}, applied with $\mu=\rho(t_n)$, leads to  
\[\rho(t_{n+1})
 =\rho(t_n)+\tau\mathcal A\bigl(\rho(t_{n+1}),
 u(t_n,\rho(t_n))\bigr)+r^n,
 \qquad \norm{r^n}\le C\tau^2.
\]
Subtracting this identity from \eqref{feedbackstep} and using
\eqref{localalip} yields
\begin{align*}
 &(1-\tau L_{\mathcal A})e_{n+1}
 \le e_n+\tau L_{\mathcal A}
 \norm{\widehat\phi_h^n-\phi(t_n)}+C\tau^2.
 \end{align*}
Since $\cP^\circ(G)$ is convex and $\mathcal Q_K\Subset\cP^\circ(G)$, its convex hull $\operatorname{conv}(\mathcal Q_K)$ is still a compact subset of $\cP^\circ(G)$.  The segment joining $\widehat\rho_h^n$ and $\rho(t_n)$ therefore stays in this compact convex set, and hence 
\begin{align*}\norm{\widehat\phi_h^n-\phi(t_n)}
&\le\norm{(\mathcal I_{h,\epsilon}u_h^n) (\widehat\rho_h^n)-u(t_n,\widehat\rho_h^n)}
 +\norm{u(t_n,\widehat\rho_h^n)-u(t_n,\rho(t_n))}\\
&\le
C_K\bigl(\tau+\frac{h\omega_K(h)}{\tau}\bigr)+
\norm{D_\mu u}_{L^\infty([0,T]\times\operatorname{conv}(\mathcal Q_K))}e_n.
\end{align*}
For sufficiently small $\tau$,
$e_{n+1}\le(1+C\tau)e_n+
C\tau\bigl(\tau+\frac{h\omega_K(h)}{\tau}\bigr)$. 
Since $\tau+\frac{h\omega_K(h)}{\tau}\to0$, take $h$ sufficiently small that
$C_K\bigl(\tau+\frac{h\omega_K(h)}{\tau}\bigr)<\sigma_K$.
We prove by induction that
$\widehat\rho_h^k$ is constructed and
$e_k\le C_K\bigl(\tau+\frac{h\omega_K(h)}{\tau}\bigr)$
for $0\le k\le n$. At $n=0$, this follows from
$\widehat\rho_h^0=\rho(0)=\mu_0$. If the assertion holds at level $n$, then
$e_n<\sigma_K$, so \eqref{exactmargin} and the contraction argument
above construct $\widehat\rho_h^{n+1}$. Iterating the preceding recurrence gives
\[
 e_{n+1}\le C\tau\bigl(\tau+\frac{h\omega_K(h)}{\tau}\bigr) 
 \sum_{j=0}^{n}(1+C\tau)^j
 \le C_K\bigl(\tau+\frac{h\omega_K(h)}{\tau}\bigr).
\] 
This closes the induction. In particular,
$ 
 \max_{0\le n\le N}\norm{\widehat\rho_h^n-\rho(t_n)}
 \le C_K\bigl(\tau+\frac{h\omega_K(h)}{\tau}\bigr).
$ 
Moreover, for every $0\le n\le N$,
$ \norm{\widehat\phi_h^n-\phi(t_n)}
\le C_K\bigl(\tau+\frac{h\omega_K(h)}{\tau}\bigr)+
 \norm{D_\mu u}_{L^\infty([0,T]\times\operatorname{conv}(\mathcal Q_K))}
\norm{\widehat\rho_h^n-\rho(t_n)}
 \le C_K\bigl(\tau+\frac{h\omega_K(h)}{\tau}\bigr). $
Together with \eqref{activeerror}, this proves \eqref{feedbackerror}
uniformly for $\mu_0\in K$.
\end{proof}

\begin{proof}[Proof of Theorem~\ref{actiontheorem}]
We first verify admissibility of the reconstructed pair.  By \eqref{potentialcoeff}, we have 
$ 
 \mathcal A(\mu,z)
 =-\divG D_p\cH(\mu,-\gradG z)+\DeltaG\mu.
$ 
Consequently, \eqref{feedbackstep} and \eqref{flux} imply
\begin{equation}\label{discretecontinuity}
 \frac{\widehat\rho_h^{n+1}-\widehat\rho_h^n}{\tau}
 +\divG m_h^{n+1}=0.
\end{equation}
By Corollary~\ref{feedback}, 
$\widehat\rho_h^n\in\Omega_{n,h}\subset\mathcal Q_K$ for every $n$.
If
$\epsilon_*:=\min_{\mu\in\mathcal Q_K}\min_i\mu_i>0$, then
$\widehat\rho_h^n\in\cP_{\epsilon_*}$;  hence the convex interpolation in
\eqref{reconstruction} also belongs to
$\cP_{\epsilon_*}\Subset\cP^\circ(G)$. 
Thus \eqref{discretecontinuity} implies
\[
 \dot{\overline\rho}_h+\divG\overline m_h=0
 \quad\text{a.e. on }(0,T),\qquad
 \overline\rho_h(0)=\mu_0.
\]
The exact densities with initial values from $K$ belong to
$\Gamma(t_n)\subset\mathcal Q_K$.  Hence \eqref{feedbackerror} and
\eqref{localcylinder} imply, for all sufficiently small $h$,
$ 
\norm{\widehat\phi_h^n}
 \le \norm{\widehat\phi_h^n-\phi(t_n)}
 +\norm{u(t_n,\rho(t_n))}\le R-1, 
$
where $R$ is defined in
\eqref{localcylinder}.
Together with
$\widehat\rho_h^{n+1}\in\Omega_{n+1,h}\subset\mathcal Q_K$, this yields
\[
 \bigl(\widehat\rho_h^{n+1},-\gradG\widehat\phi_h^n\bigr)
 \in\mathcal Q_K\times
 \bigl\{-\gradG z:\norm{z}\le R-1\bigr\}
 \Subset\cP^\circ(G)\times\bbS^{d\times d}.
\]
Hence, $\overline m_h\in L^\infty((0,T);\bbS^{d\times d})$
and $(\overline\rho_h,\overline m_h)$ is admissible. 

We next estimate the difference between $(\overline\rho_h,\overline m_h)$ and the exact minimizing pair
$(\rho,m^*)$. Set
$w(t):=u(t,\rho(t))$.
 The exact characteristic belongs to
$\Gamma(t)\subset\mathcal Q_K$, and \eqref{localcylinder} implies 
$(\rho(t),w(t))\in\mathcal K_K$. Hence we have 
$ 
 \norm{\dot\rho(t)}
 =\norm{\mathcal A(\rho(t),w(t))}\le C_{\mathcal A}.
$ 
Along the characteristic, \eqref{eq_intro} becomes
$ 
 \dot w(t)=\bm H\bigl(\rho(t),\gradG w(t)\bigr)-\DeltaG w(t),$ which implies $
 \norm{\dot w(t)}\le H_0+\norm{\DeltaG}_{\mathcal L(\bbR^d)}(R-2).$ 
Thus $\rho$ and $w$ are Lipschitz in time, uniformly for
$\mu_0\in K$.
It follows from \eqref{reconstruction} and \eqref{feedbackerror} that
for $t\in[t_n,t_{n+1}]$,
\[
 \begin{split}
 & \norm{\overline\rho_h(t)-\rho(t)}
 \le \frac{t_{n+1}-t}{\tau}
 \norm{\widehat\rho_h^n-\rho(t_n)}
 +\frac{t-t_n}{\tau}
 \norm{\widehat\rho_h^{n+1}-\rho(t_{n+1})}\\
 &+\frac{t_{n+1}-t}{\tau}\norm{\rho(t_n)-\rho(t)}
 +\frac{t-t_n}{\tau}\norm{\rho(t_{n+1})-\rho(t)}
 \le C_K\bigl(
 \tau+\frac{h\omega_K(h)}{\tau}\bigr).
 \end{split}
\]
Consequently,
\begin{equation}\label{patherror}
 \sup_{\mu_0\in K}
 \norm{\overline\rho_h-\rho}_{L^\infty(0,T)}
 \le C_K\bigl(
 \tau+\frac{h\omega_K(h)}{\tau}\bigr).
\end{equation} 
For $t\in[t_n,t_{n+1}]$, the pairs
$
\bigl(\widehat\rho_h^{n+1},-\gradG\widehat\phi_h^n\bigr)$ and $
 \bigl(\rho(t),-\gradG u(t,\rho(t))\bigr)
$ belong to the compact set displayed above. Hence
$D_p\cH$ is Lipschitz on this set. 
For $t\in[t_n,t_{n+1}]$, \eqref{feedbackerror} and the
time-Lipschitz bounds imply
\begin{align*}
 \norm{\widehat\rho_h^{n+1}-\rho(t)}
 &\le\norm{\widehat\rho_h^{n+1}-\rho(t_{n+1})}
 +\norm{\rho(t_{n+1})-\rho(t)}\le C_K\bigl(\tau+
 \frac{h\omega_K(h)}{\tau}\bigr),\\
 \norm{\widehat\phi_h^n-u(t,\rho(t))}
 &=\norm{\widehat\phi_h^n-w(t)}
 \le\norm{\widehat\phi_h^n-w(t_n)}+\norm{w(t_n)-w(t)}
 \le C_K\bigl(\tau+
 \frac{h\omega_K(h)}{\tau}\bigr).
\end{align*} Recalling the exact optimal flux $m^*$ from \eqref{mstar}, 
we obtain
\[
 \begin{aligned}
 \norm{m_h^{n+1}-m^*(t)}_E
 &\le C\Bigl(
 \norm{\widehat\rho_h^{n+1}-\rho(t)}
 +\norm{\gradG(\widehat\phi_h^n-u(t,\rho(t)))}_E\Bigr)
 +\norm{\gradG(\widehat\rho_h^{n+1}-\rho(t))}_E\\
 &\le C_K\Bigl(
 \norm{\widehat\rho_h^{n+1}-\rho(t)}
 +\norm{\widehat\phi_h^n-u(t,\rho(t))}\Bigr)\le C_K\bigl(
 \tau+\frac{h\omega_K(h)}{\tau}\bigr).
 \end{aligned}
\] 
which leads to 
\begin{equation}\label{fluxerror}
 \sup_{\mu_0\in K}
 \norm{\overline m_h-m^*}_{L^\infty(0,T;\bbS^{d\times d})}
 \le C_K\bigl(
 \tau+\frac{h\omega_K(h)}{\tau}\bigr).
\end{equation}

We now use the optimality of $(\rho,m^*)$ and estimates \eqref{patherror}--\eqref{fluxerror} to derive the quadratic
action estimate. 
Put
$ 
\delta\rho_h:=\overline\rho_h-\rho,\;
\delta m_h:=\overline m_h-m^*.$ 
Both pairs $(\overline\rho_h,\overline m_h)$ and $(\rho,m^*)$
are admissible with
$\overline\rho_h(0)=\rho(0)=\mu_0$, and  $
 \dot{\delta\rho}_h+\divG\delta m_h=0,
\; \delta\rho_h(0)=0.
$ 
The compactness of $K$ and continuity of the characteristic flow ensure
\begin{equation}
 \delta_{K,T}:=\min_{\substack{\mu_0\in K,\ 0\le t\le T\\1\le i\le d}}
\rho_i(t)>0.
\end{equation}
By \eqref{patherror}, take $h$ sufficiently small that
$\norm{\delta\rho_h}_{L^\infty(0,T)}\le\delta_{K,T}/2$. Then, for
$-1\le\lambda\le1$,
\[
\rho_i(t)+\lambda\delta\rho_{h,i}(t)
 \ge\delta_{K,T}-\norm{\delta\rho_h}_{L^\infty(0,T)}
 \ge\frac{\delta_{K,T}}{2},
 \qquad 1\le i\le d.
\]
Thus $\rho(t)+\lambda\delta\rho_h(t)\in\cP^\circ(G)$ uniformly in $t$ and
$\mu_0\in K$. For $-1\le\lambda\le1$, define
\[
 \rho_\lambda:=\rho+\lambda\delta\rho_h,
 \qquad m_\lambda:=m^*+\lambda\delta m_h,
 \qquad j_h(\lambda):=\mathcal J_0(\rho_\lambda,m_\lambda).
\]
Each $(\rho_\lambda,m_\lambda)$ is admissible.  Since
$(\rho,m^*)$ minimizes
\eqref{action},
$ 
 j_h(\lambda)\ge j_h(0)$ for $-1\le\lambda\le1.
$ 
The map $j_h$ is differentiable and $0\in(-1,1)$, so $
 j_h'(0)=0.$ 
The identity defining $m^*$ reads
\[
 m^*(t)+\gradG\rho(t)
 =D_p\cH\bigl(\rho(t),-\gradG u(t,\rho(t))\bigr),
 \qquad 0\le t\le T.
\]
The exact arguments on the right-hand side remain in a compact set.  Hence,
for some $M_{K,T}>0$, \eqref{patherror}, \eqref{fluxerror}, and the
boundedness of $\gradG$ imply
\[
 \norm{m^*+\gradG\rho}_{L^\infty(0,T;\mathbb S^{d\times d})}
 \le M_{K,T},\qquad
 \norm{\delta m_h+\gradG\delta\rho_h}_{L^\infty(0,T;\mathbb S^{d\times d})}
 \to0.
\] 
Thus, for $-1\le\lambda\le1$,
\[
 \norm{m_\lambda+\gradG\rho_\lambda}_E
 =\norm{m^*+\gradG\rho
 +\lambda\bigl(\delta m_h+\gradG\delta\rho_h\bigr)}_E
 \le M_{K,T}+1.
\] Together with the lower bound on $\rho_\lambda$, the pairs
$(\rho_\lambda,m_\lambda+\gradG\rho_\lambda)$ belong to the compact set
$\cP_{\delta_{K,T}/2}\times
 \bigl\{q\in\bbS^{d\times d}:
 \norm{q}_E\le M_{K,T}+1\bigr\}$. The second derivatives of
$\cL$, $\cF$, and $\cU_T$ are bounded in this compact set. By the chain rule, we have 
\begin{align*}
 j_h''(\lambda)
 ={}&\int_0^T D^2\cL\bigl(\rho_\lambda,
 m_\lambda+\gradG\rho_\lambda\bigr)
 \Bigl[\bigl(\delta\rho_h,\delta m_h+\gradG\delta\rho_h\bigr),
 \bigl(\delta\rho_h,\delta m_h+\gradG\delta\rho_h\bigr)\Bigr]\mathrm dt\\
 &-\int_0^T D^2\cF(\rho_\lambda)[\delta\rho_h,\delta\rho_h]\mathrm dt
 +D^2\cU_T\bigl(\rho_\lambda(T)\bigr)[\delta\rho_h(T),\delta\rho_h(T)].
\end{align*}
The boundedness of these second derivatives and $2ab\le a^2+b^2$ imply
\begin{align*}
 |j_h''(\lambda)|
 &\le C\int_0^T\Bigl(
 \norm{\delta\rho_h(t)}^2
 +\norm{\delta m_h(t)+\gradG\delta\rho_h(t)}_E^2\Bigr)\,\mathrm dt
 +C\norm{\delta\rho_h(T)}^2\\
 &\le C_K\bigl(
 \tau+\frac{h\omega_K(h)}{\tau}\bigr)^2,
 \qquad 0\le\lambda\le1,
\end{align*}
where the last inequality follows from \eqref{patherror},
\eqref{fluxerror}, and the boundedness of $\gradG$.  Taylor's formula and
$j_h'(0)=0$ give
\[
\mathcal J_0(\overline\rho_h,\overline m_h)
 -\mathcal J_0(\rho,m^*)
 = j_h(1)-j_h(0)=\int_0^1(1-\lambda)j_h''(\lambda)\,\mathrm d\lambda
 \le C_K\bigl(
 \tau+\frac{h\omega_K(h)}{\tau}\bigr)^2.
\]
Since $(\rho,m^*)$ attains the
infimum in \eqref{action},
$\mathcal J_0(\rho,m^*)=\cU(0,\mu_0)$. The admissibility of
$(\overline\rho_h,\overline m_h)$ establishes the lower bound in
\eqref{actionerror}. This finishes the proof.  
\end{proof}

\section{Numerical experiments}\label{numerics}
In this section, we first describe the Picard iteration used to solve the
implicit nodal systems in \eqref{scheme}. We then report convergence
tests for the numerical master field and its feedback-generated MFG
trajectories, together with the
action gap in the potential setting. 
Further tests illustrate the response to population redistribution and changes of graph topology.

\subsection{Picard solver}
In this part, we use the Picard iteration in the numerical experiments to compute the solutions of the implicit nodal systems \eqref{scheme}. 
At the $n$th backward step, let
${\mathcal I_{h,\epsilon}u_h^{n+1}}$ be given. 
For each $\mu^\ell\in\mathfrak M_{h,\epsilon}^\circ$, we attempt to solve the
coupled nodal equations by Picard iteration. 
We initialize
$ 
 \rho^0=\mu^\ell,\; z^0=S_\tau\bigl[(\mathcal I_{h,\epsilon}u_h^{n+1})(\mu^\ell)\bigr].
$ 
Given $(\rho^k,z^k)$, first compute the candidate density and then the value variable:  
\begin{subequations}
\begin{align}
 \widetilde\rho^{k+1}
 &=S_\tau\bigl[\mu^\ell+
 \tau\divG\bm B(\rho^k,\gradG z^k)\bigr],\\
 \rho^{k+1}
 &=\rho^k+s_k(\widetilde\rho^{k+1}-\rho^k),\label{picarddensity}\\
 z^{k+1}
 &=S_\tau\bigl[\mathcal I_{h,\epsilon}u_h^{n+1}(\rho^{k+1})
 -\tau\bm H(\rho^{k+1},\gradG z^k)\bigr],
\end{align}
\end{subequations}
where $k\in\mathbb N_0$ denotes the nonlinear iteration index. Here, if $\widetilde\rho^{k+1}\in\cP_\epsilon^\circ$, set $s_k=1$.
Otherwise, define $\delta\rho^k:=\widetilde\rho^{k+1}-\rho^k$ and choose
\begin{equation}\label{solverfeasible}
 s_k=\min\Big\{1,\,
 \eta\min_{i:\delta\rho_i^k<0}
 \frac{\rho_i^k-\epsilon}{-\delta\rho_i^k}\Big\},
 \qquad 0<\eta<1.
\end{equation}
The density update \eqref{picarddensity}, with the above choice of $s_k$,
preserves unit mass and ensures
$\rho^{k+1}\in\cP_\epsilon^\circ$. 
Indeed, $\sum_i(\divG\bm B)_i=0$ and the mass identity in Lemma~\ref{resolvent} give $\sum_i\widetilde\rho_i^{k+1}=1$; the damped update is a convex combination of two unit-mass vectors.  The choice $0<\eta<1$ in \eqref{solverfeasible} keeps every component strictly above $\epsilon$.  
We terminate the iteration when the residual of the original nodal system satisfies
\begin{equation}
 \begin{aligned}
\mathfrak R_{n,\ell}(\rho,z)
 :=
 \begin{pmatrix}
  \rho-\mu^\ell-\tau\mathcal A(\rho,z)\\
  (I-\tau\DeltaG)z-\mathcal I_{h,\epsilon}u_h^{n+1} (\rho)+\tau\bm H(\rho,\gradG z)
 \end{pmatrix},\quad 
\|\mathfrak R_{n,\ell}(\rho^k,z^k)\|_{\ell^\infty} 
 \le\mathcal R_{\mathrm{tol}}.
 \end{aligned}
\end{equation}
When this criterion is met, we set
$\mu_\ell^+=\rho^k$ and $u_\ell^n=z^k$.

\subsection{Numerical experiments} Now we use the logarithmic-mean quadratic model \eqref{model} with
$\alpha=0.7$ and
$g(\mu)=2(\mu-\mu_{\rm tar})$ to implement the numerical experiments.  All computations use the fixed
truncation parameter $\epsilon=0.01$ and the same triangulation
$\mathfrak T_{h,\epsilon}$ at every time level.  
We set $\eta=0.95$ in 
\eqref{solverfeasible}, and use a residual tolerance of $\mathcal R_{\mathrm{tol}}=10^{-11}$
at every node. In all experiments reported in this section,
every nodal system reached this tolerance within at most 10 Picard
iterations, and every density update used the full step $s_k=1$.

\textbf{Convergence test.}
We use the triangle graph with
$\omega_{12}=1$, $\omega_{23}=1.2$, $\omega_{31}=0.8$, and
$\mu_{\rm tar}=(0.15,0.25,0.60)^T$.  We take $T=0.2$ and
$N=8,16,32,64,128$, and couple the spatial refinement to the time grid by
$J=55N/4$.  We additionally compute the level $(J,N)=(3520,256)$ for the final
double-grid comparison.  For this mesh, $h=\sqrt{2}(1-3\epsilon)/J$,
so $h/\tau=4\sqrt{2}(1-3\epsilon)/(55T)\approx0.499$ at every refinement level. Both 
$h$ and $\tau$ are halved at each refinement.  Since the exact master field is not
known, we compare consecutive grids.

Let $u_h^n$ be the computed nodal
field on the mesh of size $h$,
take $K=\cP_{0.15}$, and let $\mathcal X\subset K$ be the fixed sampling
lattice obtained by dividing each edge of $K$ into $24$ equal parts.  We report
\[
 E^{\max}_{h,h/2}
 =\max_{0\le n\le N}
 \max_{\mu\in\mathcal X}
 \norm{\mathcal I_{h,\epsilon}u_h^n(\mu)
 -\mathcal I_{h/2,\epsilon}u_{h/2}^{2n}(\mu)}. 
\] 
Table~\ref{table1} reports $\log_2(E^{\max}_{2h,h}/E^{\max}_{h,h/2})$ with observed rates close to one, which shows first-order convergence along the chosen refinement path.  
The level $(J,N)=(3520,256)$ supplies the finer master field in
the last double-grid comparison.

We next examine the convergence behavior predicted by Corollary~\ref{feedback} and
Theorem~\ref{actiontheorem}.  We use the same five pairs $(J,N)$ and the same $T$,
and take the initial compact set $K=\cP_{0.10}$.  The fixed lattice obtained by dividing every edge
of $K$ into six equal parts contains $28$ initial densities.  For each
$\mu^{(k)}$, we discretize the MFG system \eqref{mfg} independently of the
numerical master field by the time-staggered scheme in
\cite[Section~3]{CuiDangGraphMFG2026}. Our finest reference uses
$N_{\rm ref}=2048$ time steps and is denoted by
$(\rho_{\rm ref}^{(k)},\phi_{\rm ref}^{(k)})$. 
On the reference time grid, we define the reference feedback flux by the analogue of \eqref{flux}, and reconstruct $\overline\rho_{\rm ref}^{(k)}$ piecewise linearly and $\overline m_{\rm ref}^{(k)}$ piecewise constantly by the analogue of \eqref{reconstruction}. 
For each $h$, let
$(\rho_h^{(k),n},\phi_h^{(k),n})$ be the numerical feedback trajectory
generated from $u_h$ by \eqref{feedbackstep}.  We set
\[
 E(h):=\max_{1\le k\le28}\max_{0\le n\le N}
 \bigl(\|\rho_h^{(k),n}-\rho_{\rm ref}^{(k)}(t_n)\|
 +\|\phi_h^{(k),n}-\phi_{\rm ref}^{(k)}(t_n)\|\bigr).
\]
For each trajectory, the reconstructed density is piecewise affine in
time, while the logarithmic-mean action density depends nonlinearly on
the density. We therefore evaluate the time integral in
\eqref{action} for the reconstructed pair \eqref{reconstruction} using
four-point Gauss--Legendre quadrature on each time interval, and define
the reference action difference 
\[
 G_{\rm ref}(h):=\max_{1\le k\le28}
 \bigl\{\mathcal J_0(\overline\rho_h^{(k)},
 \overline m_h^{(k)})
 -\mathcal J_0(\overline\rho_{\rm ref}^{(k)},
 \overline m_{\rm ref}^{(k)})\bigr\}. 
\] Since the exact minimum $\cU(0,\mu^{(k)})$ is unavailable,
$G_{\rm ref}(h)$ is a computable proxy for the action gap, obtained by
replacing $\cU(0,\mu^{(k)})$ with the action of the corresponding
independently computed reference trajectory using
$N_{\rm ref}=2048$ time steps.  
Table~\ref{table1}
shows approximately first-order decay of $E(h)$ and second-order decay of
$G_{\rm ref}(h)$.   These observations are consistent with 
Corollary~\ref{feedback} and Theorem~\ref{actiontheorem}. All computed arrival densities remained in $\mathcal P_{0.01}$.

\begin{table}[htbp]
\caption{Convergence of the three-state master equation and its numerical
feedback.}
\label{table1}
\centering
\small
\setlength{\tabcolsep}{4pt}
\begin{tabular}{ccccccc}
\toprule
$h$ & $E^{\max}_{h,h/2}$ & rate & $E(h)$ & rate & $G_{\rm ref}(h)$ & rate\\
\midrule
$1.247\times10^{-2}$ & $7.575\times10^{-3}$ & -- & $3.028\times10^{-2}$ & -- & $5.964\times10^{-4}$ & --\\
$6.235\times10^{-3}$ & $3.842\times10^{-3}$ & $0.979$ & $1.529\times10^{-2}$ & $0.985$ & $1.525\times10^{-4}$ & $1.967$\\
$3.118\times10^{-3}$ & $1.933\times10^{-3}$ & $0.991$ & $7.639\times10^{-3}$ & $1.001$ & $3.857\times10^{-5}$ & $1.984$\\
$1.559\times10^{-3}$ & $9.694\times10^{-4}$ & $0.996$ & $3.772\times10^{-3}$ & $1.018$ & $9.688\times10^{-6}$ & $1.993$\\
$7.794\times10^{-4}$ & $4.854\times10^{-4}$ & $0.998$ & $1.829\times10^{-3}$ & $1.045$ & $2.419\times10^{-6}$ & $2.002$\\
\bottomrule
\end{tabular}
\end{table}

\textbf{Density feedback on the triangle graph.}
We next visualize the density velocity generated by the numerical master
field for the triangle graph with $d=3$.  We take $T=6$,
$h=4.287\times10^{-2}$, and $N=960$. 
Figure~\ref{velocity} displays the density-velocity field
$
 \mu\mapsto
 \mathcal A\bigl(\mu,\mathcal I_{h,\epsilon}u_h^n(\mu)\bigr)
$ 
generated by the computed master feedback. The arrows indicate the
velocity direction, the colors encode its magnitude, and the marked
point identifies the sampled density with the smallest speed. Unlike a
single characteristic trajectory, each panel shows the feedback
direction and speed over the computational probability simplex. At each displayed
time, the vectors near the boundary point toward the interior,
displaying an inward tendency of the feedback density dynamics.

\begin{figure}[htbp]
\centering
\includegraphics[width=.42\textwidth]
{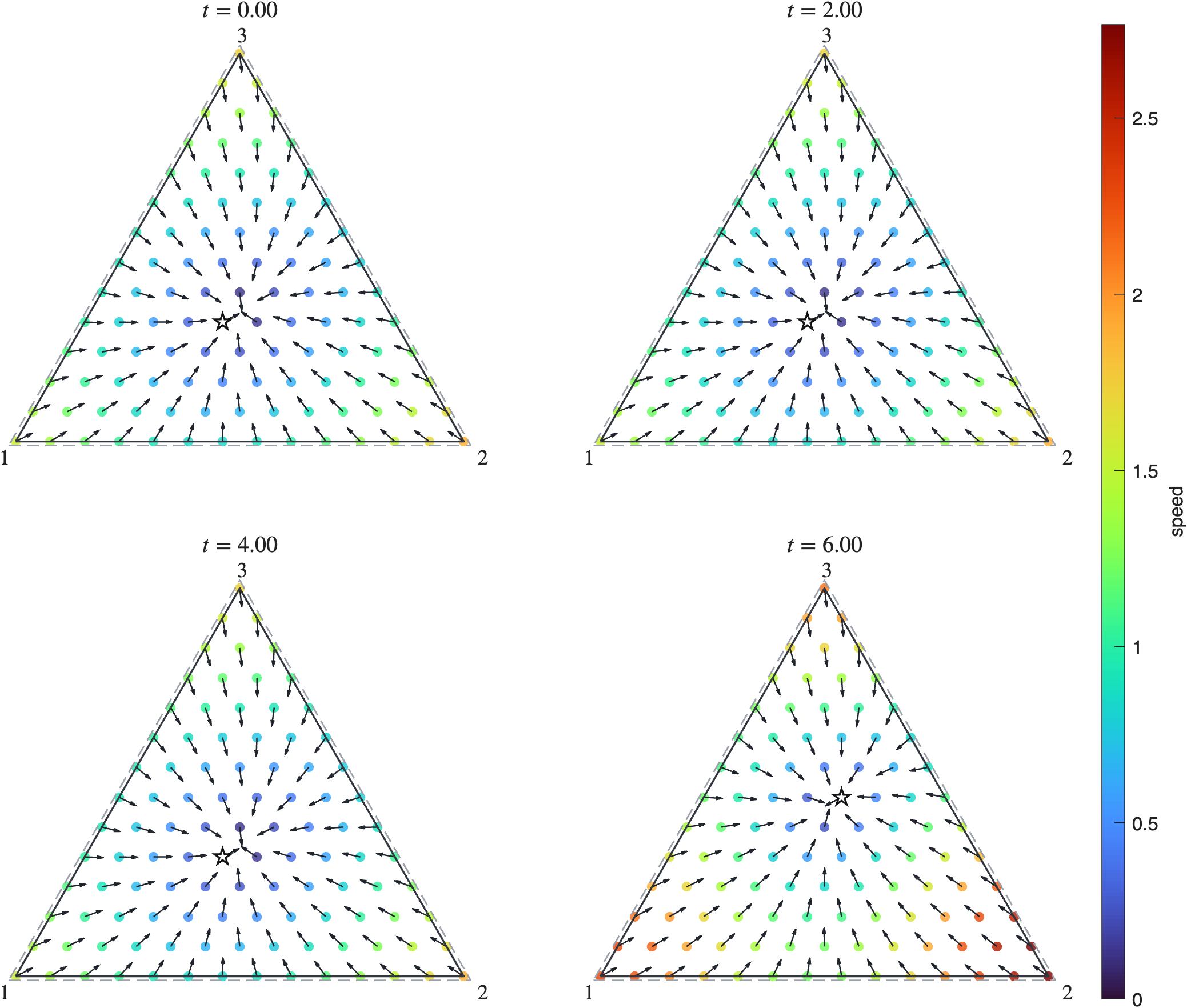}
\caption{The numerical density velocity
$\mathcal A(\mu,\mathcal I_{h,\epsilon}u_h^{n} (\mu) )$ on the triangle graph at the indicated grid times $t_n$, where the arrows indicate the direction and the colors encode the speed.  
The marked point is the sampled
minimum-speed density.}
\label{velocity}
\end{figure}

\textbf{Response to a redistribution shock and the effect of graph topology.}
In this part,
we first illustrate the feedback reuse discussed in Section~\ref{role}
by examining how the numerical master feedback propagates a density
trajectory after an instantaneous redistribution. We then compare the
post-redistribution recovery across different graph topologies. 
The triangle graph, target density, and
model parameters are the same as in the preceding $d=3$ experiment. We take $T=5$,
$h=4.287\times10^{-2}$ and $N=800$, and start from 
$\mu=(0.24,0.30,0.46)^T$.  At $t=2$, mass $0.12$ is transferred from
vertex $3$ to vertex $1$, after which the master feedback is evaluated at the
perturbed density.  Figure~\ref{shock} shows that the perturbed
and unperturbed trajectories become nearly indistinguishable, with their
distance decreasing from $0.170$ to $3.29\times10^{-5}$. The approximately linear decay of this distance on the logarithmic scale indicates a restoring response under the computed master feedback. 

\begin{figure}[htbp]
\centering
\includegraphics[width=.6\textwidth]
{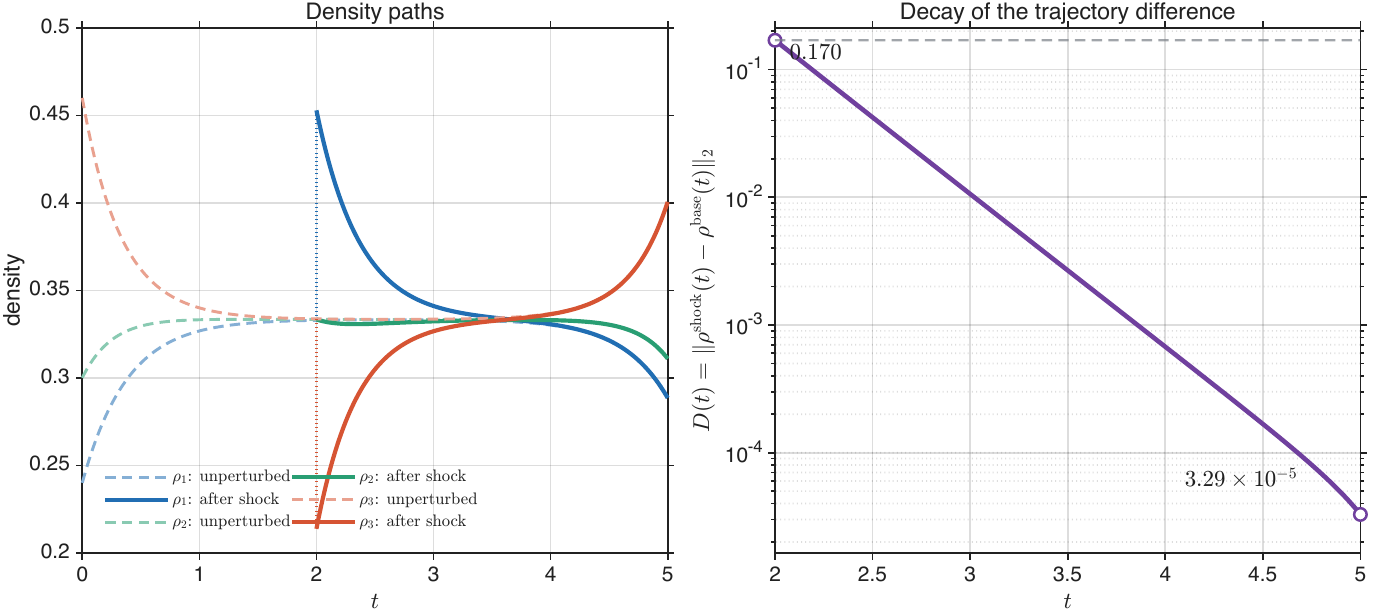}
\caption{Feedback after a redistribution shock.  The left panel compares the perturbed and unperturbed
density paths; the right panel plots their distance on a logarithmic scale.}
\label{shock}
\end{figure}

We next vary the graph topology and investigate the effect of graph topology on recovery under the master feedback.  We take $d=4$, $T=2$,
$h=1.358\times10^{-1}$, and $N=100$,
and consider the
path $P_4$, the star $S_4$, the cycle $C_4$, and the
complete graph $K_4$, with unit weight on every edge, as shown in Figure \ref{topology}.   We use the same
$\mu_0=(0.55,0.15,0.15,0.15)^T$, target
$\mu_{\rm tar}=(0.10,0.15,0.20,0.55)^T$, and model parameters for all four
graphs.  Thus the simplex and its mesh are unchanged, and 
only the graph operators and the resulting master field differ. 
At $t=0.8$, we transfer mass $0.08$ from vertex $4$ to vertex $1$ and
continue the perturbed path with the corresponding master feedback.
For each graph $G$, let $\rho_G^{\rm base}$ and $\rho_G^{\rm shock}$ denote
the unperturbed and perturbed density paths, respectively, and set
$D_G(t):=\norm{\rho_G^{\rm shock}(t)-\rho_G^{\rm base}(t)}$, $0.8\le t\le2$.
Figure~\ref{topology} compares the normalized difference
$D_G(t)/D_G(0.8)$.  The left panel displays its decay on a logarithmic scale.  The right panel reports the
empirical recovery rate
$-\log\bigl(D_G(2)/D_G(0.8)\bigr)/(2-0.8)$ against
$\lambda_2(-\Delta_G)$. Here, $\lambda_2(-\Delta_G)$ is the smallest positive eigenvalue of $-\Delta_G$, also
called the algebraic connectivity of $G$.  A larger $\lambda_2(-\Delta_G)$
indicates stronger connectivity of the graph.
For these four test graphs, Figure \ref{topology} shows that the recovery rate increases with
$\lambda_2(-\Delta_G)$, suggesting a positive association between larger algebraic connectivity and
faster recovery. A theoretical characterization of how graph topology affects the
recovery rate is an interesting direction for future work.

\begin{figure}[htbp]
\centering
\includegraphics[width=.6\textwidth]
{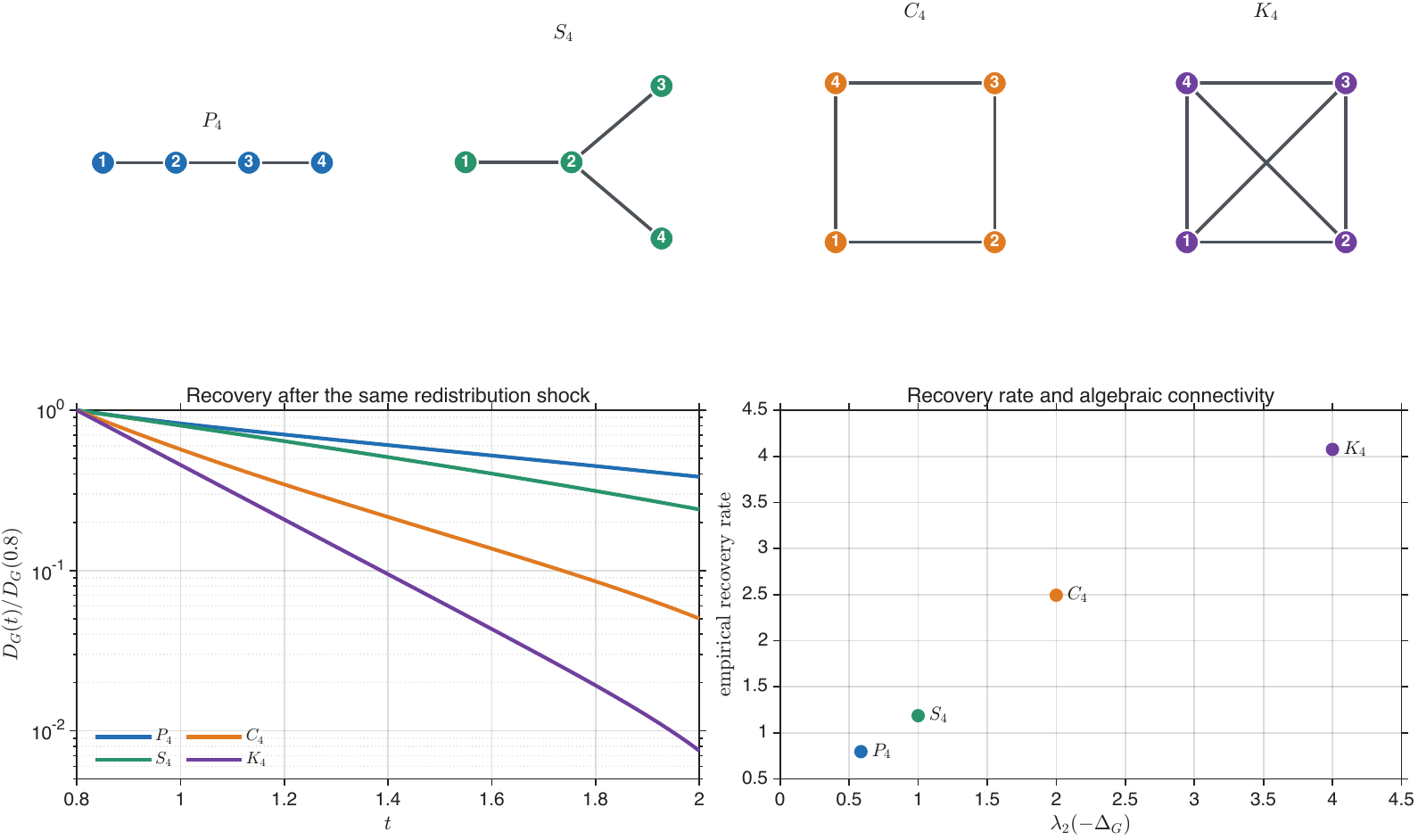}
\caption{Effect of graph topology on recovery under the master feedback,
with unit weight on every displayed edge.  The lower left panel shows the recovery curves after the same
redistribution shock; the lower right panel compares their empirical recovery
rates with the algebraic connectivity.}
\label{topology}
\end{figure}

\section{Conclusion}

In this paper, we study a fully discrete semi-Lagrangian approximation of the
graph master equation with individual noise. The analysis shows that the
approximation can be controlled on compact regions of the open
probability simplex, where the coefficients of the master equation remain uniformly regular.
The characteristic localization keeps the active numerical regions away from the simplex boundary and allows the off-grid reconstruction and nonlinear updates to be treated simultaneously.
The resulting error control propagates from the master field to the induced
MFG trajectories. In the potential setting, it further
translates into a higher-order estimate in the associated
planning problem. 

Several directions merit further investigation. 
First, extending the convergence analysis toward the whole probability
simplex, where the coefficients may degenerate at the boundary, would complement the present convergence 
 result. Second, 
developing efficient numerical methods for master equations on graphs with a larger number of vertices, such as sparse-grid or low-rank
methods \cite{EigelSchneiderSommer2023}, is another
natural direction. 
Moreover, higher-order discretizations \cite{BokanowskiPicarelliReisinger2018}  and alternative structure-preserving spatial approximations are
also of interest, particularly those that retain density interiority and
stability while preserving Lasry--Lions monotonicity at the fully discrete
level.  

\iffalse
\vspace{2mm}
\textbf{Declaration on AI use.} AI-assisted tools were used to improve the English and
presentation and to assist with numerical experiments. All mathematical statements,
proofs, implementations, references, and computational results were independently
checked by the authors, who take responsibility for the content. 
\fi

\appendix
\section{}\label{appendixproofs}

\begin{proof}[Proof of Proposition~\ref{regularity}]
We split the proof into two steps. We first prove the well-posedness and monotonicity \eqref{umono} of the classical master equation based on \cite{GangboMunozWuZhang2026}. We then show the $\mathcal C^2$-regularity and \eqref{dulip} for the solution. 

\textit{Step 1:  Well-posedness of the classical master solution and its $\mathcal C^1$-regularity.} The differential condition \eqref{dll} implies the strict
Lasry--Lions inequality required in
\cite[(3.5)]{GangboMunozWuZhang2026}.  Indeed, for
$(\mu,p)\ne(\nu,\widetilde p)$, set
$(\mu_s,p_s)=(\nu+s(\mu-\nu),\widetilde p+s(p-\widetilde p))$.  Then
\begin{align}\label{integrated-dll}
 \ip{\bm H(\mu,p)-\bm H(\nu,\widetilde p)}{\mu-\nu}
 -\ip{\bm B(\mu,p)-\bm B(\nu,\widetilde p)}{p-\widetilde p}_E
 =\int_0^1\mathcal M_{\mu_s,p_s}(\mu-\nu,p-\widetilde p)\,\mathrm ds<0.
\end{align} 
Together with \eqref{data}, the mobility bound
\eqref{boundary}, and the terminal condition \eqref{terminal},  
the arguments of
\cite[Propositions~3.5--3.6 and Remark~3.7]{GangboMunozWuZhang2026}
can be applied here. 
For each fixed interior initial density and bounded families of continuous
value paths, the positivity estimate in 
\cite[Theorem~1.1]{GangboMunozWuZhang2026}
keeps the density paths in a common compact subset of the open simplex.
Thus the interior regularity assumptions, together with the bound on $g$,
suffice for the classical master solution and its $\mathcal C^1$-regularity.

We next prove \eqref{umono}. Fix $t\in[0,T]$ and
$\mu,\nu\in\cP^\circ(G)$. Let $(\rho,\phi)$ and $(\sigma,\psi)$ be the
solutions of \eqref{mfg} with $\rho(t)=\mu$ and $\sigma(t)=\nu$,
respectively.
Using \eqref{mfg}, \eqref{ibp}, and the symmetry of $\DeltaG$, we obtain
\begin{align}\label{continuousidentity}
 \frac{\mathrm d}{\mathrm ds}
 \ip{\phi-\psi}{\rho-\sigma}
 &=\ip{\bm H(\rho,\gradG\phi)-\bm H(\sigma,\gradG\psi)}
 {\rho-\sigma}\notag\\
 &-
 \ip{\bm B(\rho,\gradG\phi)-\bm B(\sigma,\gradG\psi)}
 {\gradG\phi-\gradG\psi}_E.
\end{align} 
The differential condition \eqref{dll}, integrated along the segment
between the two density--gradient pairs in \eqref{continuousidentity},
shows that its right-hand side is nonpositive. Hence by \eqref{terminal}, 
\begin{align*}
 \ip{\phi(t)-\psi(t)}{\rho(t)-\sigma(t)}
 &\ge\ip{\phi(T)-\psi(T)}{\rho(T)-\sigma(T)}\\
 &=\ip{g(\rho(T))-g(\sigma(T))}{\rho(T)-\sigma(T)}\ge0.
\end{align*}
Since $u(t,\mu)=\phi(t)$ and $u(t,\nu)=\psi(t)$, this proves
\eqref{umono}.

\textit{Step 2: Assuming now the additional $\mathcal C^2$ regularity of $\bm H$, $\bm B$, and $g$, we prove that the classical solution is of class $\mathcal C^2$.}  
Following the proof of
\cite[Proposition~3.6]{GangboMunozWuZhang2026}, we rescale the forward--backward characteristic system from $[t,T]$ to $[0,1]$ and apply the implicit
function theorem to the resulting problem. 
For $t\in[0,T]$, set
$\widehat\phi(\vartheta)=\phi(t+(T-t)\vartheta)$ and
$\widehat\rho(\vartheta)=\rho(t+(T-t)\vartheta)$ for
$0\le\vartheta\le1$.   Let
\[
 X:=\mathcal C^1([0,1];\bbR^d)\times
 \mathcal C^1([0,1];\bbR_0^d),
 \qquad
 Y:=\mathcal C([0,1];\bbR^d)\times
 \mathcal C([0,1];\bbR_0^d)\times\bbR^d\times\bbR_0^d,
\]
and let $\mathscr X^\circ$ be the set of pairs
$(v,\nu)$ with
$v\in\mathcal C^1([0,1];\bbR^d)$ and
$\nu\in\mathcal C^1([0,1];\bbR_1^d)$ satisfying
$\nu_i>0$ on $[0,1]$ for every $i$.
Identifying $(v,\nu)$ with
$(v,\nu-d^{-1}\bm1)\in X$, we regard $\mathscr X^\circ$ as an open
subset of $X$. 
\iffalse
Indeed, for
$(\widehat\phi,\widehat\rho)\in\mathscr X^\circ$, continuity and compactness
ensure that
$ 
 \delta:=\min_{1\le i\le d}\min_{0\le\vartheta\le1}
 \widehat\rho_i(\vartheta)>0.
$ 
If $(\psi,\eta)\in X$ and
$\norm{\psi}_{\mathcal C^1}+\norm{\eta}_{\mathcal C^1}<\delta/2$, then
\[
 \widehat\rho_i(\vartheta)+\eta_i(\vartheta)
 \ge\delta-\norm{\eta}_{\mathcal C^0}>\frac{\delta}{2},
 \qquad
 \sum_{i=1}^d\bigl(\widehat\rho_i(\vartheta)+\eta_i(\vartheta)\bigr)=1.
\]
Hence $(\widehat\phi+\psi,\widehat\rho+\eta)\in\mathscr X^\circ$, which verifies that $\mathscr X^{\circ}$ is open in $X$. \fi
We also view $\cP^\circ(G)$ as an open subset of
$\bbR_1^d$, and let $I\subset\bbR$ be an open interval containing
$[0,T]$.  Define
$\mathfrak F:I\times\cP^\circ(G)\times\mathscr X^\circ\to Y$
by
{\small\begin{align*}
 &\quad \mathfrak F(t,\mu,v,\nu)\\&
 =\big(v'
-(T-t)\bigl[\bm H(\nu,\gradG v)
             -\DeltaG v\bigr],\nu'
 -(T-t)\bigl[\divG\bm B(\nu,\gradG v)
             +\DeltaG\nu\bigr],v(1)-g(\nu(1)),
 \nu(0)-\mu\big).
\end{align*}}
\!\!For $t\in[0,T]$, a pair $(\phi,\rho)$ solves the characteristic
system~\eqref{mfg} on $[t,T]$ with $\rho(t)=\mu$ if and only if its
rescaling $(\widehat\phi,\widehat\rho)$ satisfies $
 \mathfrak F(t,\mu,\widehat\phi,\widehat\rho)=0.$ 

To apply the implicit function theorem and prove the
$\mathcal C^2$-dependence of the rescaled characteristic pair
$(\widehat\phi,\widehat\rho)$ on $(t,\mu)$,
we verify the following two properties:
\[
 \text{\rm(i)}\quad \mathfrak F\ \text{is of class }\mathcal C^2;
 \qquad
 \text{\rm(ii)}\quad
 D_{(\widehat\phi,\widehat\rho)}
 \mathfrak F(t,\mu,\widehat\phi,\widehat\rho)
 \text{ is an isomorphism}.
\]

 (i) We verify $\mathcal C^2$-regularity of $\mathfrak F$. Indeed, for each
$(v,\nu)\in\mathscr X^\circ$, positivity of
$\nu$ and compactness of $[0,1]$ provide a neighborhood in
$X$ on which the arguments of $\bm H$ and $\bm B$ range in a fixed compact
subset of 
$\mathcal P^{\circ}(G) \times\bbS^{d\times d}$.  
The assumed $\mathcal C^2$
regularity of $\bm H$, $\bm B$, and $g$ therefore implies that $\mathfrak F$ is
of class $\mathcal C^2$.

(ii) Fix $t\in[0,T)$, $\mu\in\cP^\circ(G)$, and
$(\widehat\phi,\widehat\rho)\in\mathscr X^\circ$ satisfying
$
 \mathfrak F(t,\mu,\widehat\phi,\widehat\rho)=0.$ We verify that $\mathfrak L:=
 D_{(\widehat\phi,\widehat\rho)}
 \mathfrak F(t,\mu,\widehat\phi,\widehat\rho):X\to Y
 \text{ is an isomorphism}.$  
For
$(\psi,\eta)\in X$, one can calculate  
{\small\begin{align*}
 \mathfrak L[\psi,\eta]&=\Bigl(\psi'
-(T-t) \bigl[D_\mu\bm H(\widehat\rho,\gradG\widehat\phi )[\eta]
 +D_p\bm H(\widehat\rho,\gradG\widehat\phi )[\gradG\psi]
 -\DeltaG\psi\bigr],\\
 &\eta'-(T-t)\bigl[\divG\bigl(
 D_\mu\bm B(\widehat\rho,\gradG\widehat\phi )[\eta]
 +D_p\bm B(\widehat\rho,\gradG\widehat\phi )[\gradG\psi]\bigr)
 +\DeltaG\eta\bigr],\psi(1)-Dg(\widehat\rho(1))[\eta(1)],\eta(0)\Bigr).
\end{align*}}  
\!The trajectory image
\[
 K_*:=\bigl\{(\widehat\rho(\vartheta), \gradG\widehat\phi(\vartheta)):
 0\le\vartheta\le1\bigr\}
 \Subset\cP^\circ(G)\times\bbS^{d\times d}
\]
is compact. The boundedness of the coefficient derivatives on $K_*$ and that of the
graph operators imply
$
 \norm{\mathfrak L[\psi,\eta]}_Y
 \le C\norm{(\psi,\eta)}_X.
$ 
Hence $\mathfrak L:X\to Y$ is bounded. We now show that $\mathfrak L$ is bijective. 

We first prove that $\mathfrak L$ is injective. 
Suppose first that $\mathfrak L[\psi,\eta]=0$.  By the strict differential monotonicity condition \eqref{dll}, applying
\cite[Remark~2.1]{CuiDangGraphMFG2026} to $K_*$ gives
$c_*>0$ such that \begin{align*}
 &\ip{D_\mu\bm H(\widehat\rho,\gradG\widehat\phi)[\eta]}{\eta}
 +\ip{D_p\bm H(\widehat\rho,\gradG\widehat\phi)[r]}{\eta}\\
 &\quad-\ip{D_\mu\bm B(\widehat\rho,\gradG\widehat\phi)[\eta]}{r}_E
 -\ip{D_p\bm B(\widehat\rho,\gradG\widehat\phi)[r]}{r}_E
 \le-c_*\bigl(\norm{\eta}^2+\norm{r}_E^2\bigr)
\end{align*} 
for $\eta\in\bbR_0^d$ and $r\in\bbS^{d\times d}$. 
 Using \eqref{ibp}, we obtain
\begin{align*}
 \frac{\mathrm d}{\mathrm d\vartheta}\ip{\psi}{\eta}
 &=(T-t) \ip{D_\mu\bm H[\eta]
 +D_p\bm H[\gradG\psi]}{\eta}\\
 &\quad-(T-t) \ip{D_\mu\bm B[\eta]
 +D_p\bm B[\gradG\psi]}{\gradG\psi}_E\le-(T-t)c_*\bigl(\norm{\eta}^2
 +\norm{\gradG\psi}_E^2\bigr),
\end{align*}
where all coefficients are evaluated at
$(\widehat\rho,\gradG\widehat\phi)$.  Since $\eta(0)=0$ and
$\psi(1)=Dg(\widehat\rho(1))[\eta(1)]$, by \eqref{terminal}, we have 
\begin{align*}
0\le\ip{Dg(\widehat\rho(1))[\eta(1)]}{\eta(1)}
 =\ip{\psi(1)}{\eta(1)}\le-(T-t)c_*\int_0^1
 \bigl(\norm{\eta}^2+\norm{\gradG\psi}_E^2\bigr)
 \,\mathrm d\vartheta\le0.
\end{align*}
Hence $\eta=0$ and $\gradG\psi=0$. 
Since $G$ is connected,
the kernel of $\gradG$ is $\operatorname{span}\{\bm1\}$, so
$\psi(\vartheta)\in\operatorname{span}\{\bm1\}$ for every
$\vartheta\in[0,1]$. With $\eta=0$ and $\gradG\psi=0$, the first
component of $\mathfrak L[\psi,\eta]=0$ reduces to
$\psi'=0$. The terminal component reduces to $\psi(1)=0$.
Consequently, $\psi=0$ on $[0,1]$. Thus $\psi=0$ and $\eta=0$,
so $\ker\mathfrak L=\{0\}$ and $\mathfrak L$ is injective.

We next prove that $\mathfrak L$ is surjective.
Fix $(f_1,f_2,\xi,\zeta)\in Y$.  For each
$y\in\bbR^d$, let $(\psi_y,\eta_y)$ solve the two differential equations in
the first two components of 
$
 \mathfrak L[\psi,\eta]=(f_1,f_2,\xi,\zeta)$ 
with initial data $\psi_y(0)=y$ and
$\eta_y(0)=\zeta$.
The coefficient derivatives are continuous along
$(\widehat\rho,\gradG\widehat\phi)$, while $\gradG$, $\divG$, and $\DeltaG$ are
bounded linear operators.  Standard linear ODE theory therefore gives a
unique solution $(\psi_y,\eta_y)\in X$ on $[0,1]$.
The remaining third component is satisfied precisely when
$M(y)=0$, where
\[
 M(y):=\psi_y(1)-Dg(\widehat\rho(1))[\eta_y(1)]-\xi.
\]
The dependence of the solution on $y$ shows that $M$ is affine. If
$M(y_1)=M(y_2)$, then
$(\psi_{y_1}-\psi_{y_2},\eta_{y_1}-\eta_{y_2})$ belongs to
$\ker\mathfrak L$. The injectivity of $\mathfrak L$ therefore gives
$y_1=y_2$, so $M:\bbR^d\to\bbR^d$ is injective. Since $M$ is an
affine map between spaces of the same finite dimension, it is bijective. 
Hence there is a unique $y$ such that $M(y)=0$. For this choice of $y$,
$ 
 \mathfrak L[\psi_y,\eta_y]=(f_1,f_2,\xi,\zeta),
$ 
so $\mathfrak L$ is surjective.  
Thus $\mathfrak L:X\to Y$ is a bounded bijection.
  Since $X$ and $Y$ are Banach
spaces, the bounded inverse theorem ensures
$\mathfrak L^{-1}\in\mathcal L(Y,X)$. Hence $\mathfrak L$ is an isomorphism, which proves~{\rm(ii)}.  

With {\rm(i)} and {\rm(ii)}, we now apply the implicit
function theorem. Fix $(t_0,\mu_0)\in[0,T)\times\cP^\circ(G)$ and let
$(\widehat\phi_0,\widehat\rho_0)$ be the corresponding rescaled
characteristic pair. Then
$ 
\mathfrak F(t_0,\mu_0,\widehat\phi_0,\widehat\rho_0)=0. 
$ 
The implicit function theorem provides a neighborhood
$U$ of $(t_0,\mu_0)$ and a unique map
$ 
(t,\mu)\mapsto
 (\widehat\phi^{t,\mu},\widehat\rho^{t,\mu})\in \mathscr X^\circ
$ 
of class $\mathcal C^2$ on $U$ satisfying
$\mathfrak F(t,\mu,\widehat\phi^{t,\mu},
\widehat\rho^{t,\mu})=0$ and $
(\widehat\phi^{t_0,\mu_0},\widehat\rho^{t_0,\mu_0})
 =(\widehat\phi_0,\widehat\rho_0).$  
For the terminal time $t=T$, at $(T,\mu)$, the unique rescaled solution is 
$\widehat\rho\equiv\mu$ and $\widehat\phi\equiv g(\mu)$.  The operator
$\mathfrak L$ then becomes
$(\psi,\eta)\mapsto
 \bigl(\psi',\eta',\psi(1)-Dg(\mu)[\eta(1)],\eta(0)\bigr).$ 
For $(f_1,f_2,\xi,\nu)\in Y$, its inverse is
\begin{align*}
 \eta(\vartheta)&=\nu+\int_0^\vartheta f_2(s)\,\mathrm ds,
 \qquad \psi(\vartheta)=y+\int_0^\vartheta f_1(s)\,\mathrm ds,\\ y&=\xi+Dg(\mu)\Bigl[\nu+\int_0^1f_2(s)\,\mathrm ds\Bigr]
 -\int_0^1f_1(s)\,\mathrm ds.
\end{align*}
Thus this operator is an isomorphism. Since $I$ is open and contains $T$, the implicit function theorem also
applies at $t=T$.   Uniqueness of
the characteristic solution shows that these local solution maps agree on
their overlaps with $t\in[0,T]$. Therefore the rescaled solution map is of class $\mathcal C^2$ on
$[0,T]\times\cP^\circ(G)$.  Since
$u(t,\mu)=\widehat\phi^{t,\mu}(0)$ and evaluation at zero is a bounded
linear map on $X$, it follows that
$u\in\mathcal C^2([0,T]\times\cP^\circ(G);\bbR^d)$.  

Finally, set
$L_K:=\sup_{(t,\varrho)\in[0,T]\times K}
\norm{D_{\mu\mu}^2u(t,\varrho)}$.  Since $K$ is convex,
\[
 D_\mu u(t,\mu)-D_\mu u(t,\nu)
 =\int_0^1D_{\mu\mu}^2u\bigl(t,\nu+s(\mu-\nu)\bigr)
 [\mu-\nu,\cdot]\mathrm ds.
\] 
Taking the operator norm proves \eqref{dulip}.
\end{proof}

\raggedbottom
\bibliographystyle{abbrv}
\bibliography{references.bib}

\end{document}